\documentclass[11pt,reqno]{amsart}
\usepackage[headings]{fullpage}

\usepackage{booktabs,array,ragged2e}
\usepackage{amssymb}
\usepackage{url}
\usepackage{equipert}
\usepackage[square,numbers,sort&compress]{natbib}

\usepackage{subfigure}
\usepackage{multirow}

\usepackage{tikz}
\usetikzlibrary{calc}

\usepackage{pgfplots} %

\usetikzlibrary{positioning,calc}

\usepackage{amsfonts}
\usepackage{amsmath}
\usepackage{enumerate}

\usepackage{hyperref}

\DeclareMathOperator    \aff                    {aff}

\DeclareMathOperator    \conv           {conv}

\DeclareMathOperator    \intr                   {int}

\DeclareMathOperator    \relint         {rel\,int}

\DeclareMathOperator    \verts          {vert}

\newcommand{\bb}{\mathbb}

\newcommand{\R}{\bb R}

\newcommand{\Z}{\bb Z}
\newcommand{\N}{\bb N}

\newcommand\st{\mid}

\def\ve#1{\mathchoice{\mbox{\boldmath$\displaystyle\bf#1$}}
{\mbox{\boldmath$\textstyle\bf#1$}}
{\mbox{\boldmath$\scriptstyle\bf#1$}}
{\mbox{\boldmath$\scriptscriptstyle\bf#1$}}}

\newcommand{\bpi}{\bar \pi}

\newcommand{\setcond}[2]{\left\{\, #1 \,\st\, #2 \,\right\}}
\newcommand{\I}{\mathcal{P}}  %
\newcommand{\Itri}[1][q]{\I_{#1,\tri}}
\newcommand{\Idiag}[1][q]{\I_{#1,\diag}}
\newcommand{\Ivert}[1][q]{\I_{#1,\ver}}

\newcommand{\Ihor}[1][q]{\I_{#1,\hor}}
\newcommand{\Iedge}[1][q]{\I_{#1,\edge}}
\newcommand{\Ipoint}[1][q]{\I_{#1,\point}}

\renewcommand{\P}{\mathcal{P}}

\newcommand{\T}{\mathcal T}

\newcommand{\point}{{\EquiPoint}}
\newcommand{\ver}{{\EquiVertical}}
\newcommand{\hor}{{\EquiHorizontal}}
\newcommand{\diag}{{\EquiDiagonal}}
\newcommand{\edge}{{\EquiEdge}}

\newcommand{\tri}{{\EquiTriangle}} %

\newcommand\FundaTriangleLower{{}^{}_{\ve0}\EquiTriangleLower}  %
\newcommand\FundaTriangleUpper{{}^{}_{\ve0}\EquiTriangleUpper}  %

\newcommand{\E}{\mathcal{E}}

\newcommand{\rx}{{\ve r}}
\newcommand{\x}{{\ve x}}
\newcommand{\y}{{\ve y}}
\newcommand{\z}{{\ve z}}
\renewcommand{\v}{{\ve v}}

\newcommand{\e}{{\ve e}}
\renewcommand{\u}{{\ve u}}

\newcommand{\f}{{\ve f}}
\newcommand{\0}{{\ve 0}}
\newcommand{\1}{{\ve 1}}

\newcommand{\p}{{\ve p}}
\renewcommand{\t}{{\ve t}}
\newcommand{\w}{{\ve w}}
\renewcommand{\b}{{\ve b}}
\renewcommand{\d}{{\ve d}}
\newcommand{\cve}{{\ve c}}
\newcommand{\h}{{\ve h}}

\newcommand{\B}{B}

\def\st{\mid}
\newenvironment{psmallmatrix}{\left(\smallmatrix}{\endsmallmatrix\right)}
\newenvironment{psmallmatrixbig}{\bigl(\smallmatrix}{\endsmallmatrix\bigr)}
\newcommand\InlineFrac[2]{#1/#2}  %
\newcommand\ColVec[3][\relax]%
{
  \ifx#1\relax
  \bgroup\let\frac=\InlineFrac\begin{psmallmatrixbig}#2\vphantom{/}\\#3\vphantom{/}\end{psmallmatrixbig}\egroup
  \else
  \bgroup\let\frac=\InlineFrac\begin{psmallmatrixbig}\ifx#200\else#2/#1\fi\\\ifx#300\else#3/#1\fi\end{psmallmatrixbig}\egroup
  \fi
}

\newcommand\CVcoordleft[4]{#1}
\newcommand\CVcoordright[4]{#2}
\newcommand\CVcoordbottom[4]{#3}
\newcommand\CVcoordtop[4]{#4}

\newcommand{\diagGrid}[1][0101]{
\node[anchor=east,draw=none,font=\tiny,inner sep=2pt,yshift=1pt] at (0,0) {$\CVcoordbottom#1$};
\node[anchor=east,draw=none,font=\tiny,inner sep=2pt,yshift=-1pt] at (0,5) {$\CVcoordtop#1$};
\node[anchor=north,draw=none,font=\tiny,inner sep=2pt,xshift=1pt] at (0,0) {$\CVcoordleft#1$};
\node[anchor=north,draw=none,font=\tiny,inner sep=2pt,xshift=-1pt] at (5,0) {$\CVcoordright#1$};
\draw[step=1.0,black,thin] (0,0) grid (5,5);
\draw[black,thin]  (0,1) -- (1,0);
\draw[black,thin] (0,2) -- (2,0);
\draw[black,thin] (0,3)--(3,0);
\draw[black,thin] (0,4)--(4,0);
\draw[black,thin] (0,5)--(5,0);
\draw[black,thin]  (5,1) -- (1,5);
\draw[black,thin] (5,2) -- (2,5);
\draw[black,thin] (5,3)--(3,5);
\draw[black,thin] (5,4)--(4,5);}

\makeatletter
\renewcommand{\pod}[1]%
{\allowbreak\mathchoice{\mkern18mu}{\mkern8mu}{\mkern8mu}{\mkern8mu}(#1)}
\makeatother

\chardef\Myunderscore=`\_
\newcommand\underscore{\Myunderscore\allowbreak}

\newcommand\githubsearchurl{https://github.com/mkoeppe/infinite-group-relaxation-code/search}
\pgfkeyssetvalue{/sagefunc/california_ip}{\href{\githubsearchurl?q=\%22def+california_ip(\%22}{\sage{california\underscore{}ip}}}%
\pgfkeyssetvalue{/sagefunc/automorphism}{\href{\githubsearchurl?q=\%22def+automorphism(\%22}{\sage{automorphism}}}%
\pgfkeyssetvalue{/sagefunc/multiplicative_homomorphism}{\href{\githubsearchurl?q=\%22def+multiplicative_homomorphism(\%22}{\sage{multiplicative\underscore{}homomorphism}}}%
\pgfkeyssetvalue{/sagefunc/projected_sequential_merge}{\href{\githubsearchurl?q=\%22def+projected_sequential_merge(\%22}{\sage{projected\underscore{}sequential\underscore{}merge}}}%
\pgfkeyssetvalue{/sagefunc/restrict_to_finite_group}{\href{\githubsearchurl?q=\%22def+restrict_to_finite_group(\%22}{\sage{restrict\underscore{}to\underscore{}finite\underscore{}group}}}%
\pgfkeyssetvalue{/sagefunc/restrict_to_finite_group_3}{\href{\githubsearchurl?q=\%22def+restrict_to_finite_group_3(\%22}{\sage{restrict\underscore{}to\underscore{}finite\underscore{}group\underscore{}3}}}%
\pgfkeyssetvalue{/sagefunc/interpolate_to_infinite_group}{\href{\githubsearchurl?q=\%22def+interpolate_to_infinite_group(\%22}{\sage{interpolate\underscore{}to\underscore{}infinite\underscore{}group}}}%
\pgfkeyssetvalue{/sagefunc/two_slope_fill_in}{\href{\githubsearchurl?q=\%22def+two_slope_fill_in(\%22}{\sage{two\underscore{}slope\underscore{}fill\underscore{}in}}}%
\pgfkeyssetvalue{/sagefunc/ll_strong_fractional}{\href{\githubsearchurl?q=\%22def+ll_strong_fractional(\%22}{\sage{ll\underscore{}strong\underscore{}fractional}}}%
\pgfkeyssetvalue{/sagefunc/hildebrand_2_sided_discont_1_slope_1}{\href{\githubsearchurl?q=\%22def+hildebrand_2_sided_discont_1_slope_1(\%22}{\sage{hildebrand\underscore{}2\underscore{}sided\underscore{}discont\underscore{}1\underscore{}slope\underscore{}1}}}%
\pgfkeyssetvalue{/sagefunc/hildebrand_2_sided_discont_2_slope_1}{\href{\githubsearchurl?q=\%22def+hildebrand_2_sided_discont_2_slope_1(\%22}{\sage{hildebrand\underscore{}2\underscore{}sided\underscore{}discont\underscore{}2\underscore{}slope\underscore{}1}}}%
\pgfkeyssetvalue{/sagefunc/hildebrand_discont_3_slope_1}{\href{\githubsearchurl?q=\%22def+hildebrand_discont_3_slope_1(\%22}{\sage{hildebrand\underscore{}discont\underscore{}3\underscore{}slope\underscore{}1}}}%
\pgfkeyssetvalue{/sagefunc/dr_projected_sequential_merge_3_slope}{\href{\githubsearchurl?q=\%22def+dr_projected_sequential_merge_3_slope(\%22}{\sage{dr\underscore{}projected\underscore{}sequential\underscore{}merge\underscore{}3\underscore{}slope}}}%
\pgfkeyssetvalue{/sagefunc/chen_4_slope}{\href{\githubsearchurl?q=\%22def+chen_4_slope(\%22}{\sage{chen\underscore{}4\underscore{}slope}}}%
\pgfkeyssetvalue{/sagefunc/gmic}{\href{\githubsearchurl?q=\%22def+gmic(\%22}{\sage{gmic}}}%
\pgfkeyssetvalue{/sagefunc/gj_2_slope}{\href{\githubsearchurl?q=\%22def+gj_2_slope(\%22}{\sage{gj\underscore{}2\underscore{}slope}}}%
\pgfkeyssetvalue{/sagefunc/gj_2_slope_repeat}{\href{\githubsearchurl?q=\%22def+gj_2_slope_repeat(\%22}{\sage{gj\underscore{}2\underscore{}slope\underscore{}repeat}}}%
\pgfkeyssetvalue{/sagefunc/dg_2_step_mir}{\href{\githubsearchurl?q=\%22def+dg_2_step_mir(\%22}{\sage{dg\underscore{}2\underscore{}step\underscore{}mir}}}%
\pgfkeyssetvalue{/sagefunc/kf_n_step_mir}{\href{\githubsearchurl?q=\%22def+kf_n_step_mir(\%22}{\sage{kf\underscore{}n\underscore{}step\underscore{}mir}}}%
\pgfkeyssetvalue{/sagefunc/gj_forward_3_slope}{\href{\githubsearchurl?q=\%22def+gj_forward_3_slope(\%22}{\sage{gj\underscore{}forward\underscore{}3\underscore{}slope}}}%
\pgfkeyssetvalue{/sagefunc/drlm_backward_3_slope}{\href{\githubsearchurl?q=\%22def+drlm_backward_3_slope(\%22}{\sage{drlm\underscore{}backward\underscore{}3\underscore{}slope}}}%
\pgfkeyssetvalue{/sagefunc/drlm_2_slope_limit}{\href{\githubsearchurl?q=\%22def+drlm_2_slope_limit(\%22}{\sage{drlm\underscore{}2\underscore{}slope\underscore{}limit}}}%
\pgfkeyssetvalue{/sagefunc/drlm_2_slope_limit_1_1}{\href{\githubsearchurl?q=\%22def+drlm_2_slope_limit_1_1(\%22}{\sage{drlm\underscore{}2\underscore{}slope\underscore{}limit\underscore{}1\underscore{}1}}}%
\pgfkeyssetvalue{/sagefunc/bhk_irrational}{\href{\githubsearchurl?q=\%22def+bhk_irrational(\%22}{\sage{bhk\underscore{}irrational}}}%
\pgfkeyssetvalue{/sagefunc/bccz_counterexample}{\href{\githubsearchurl?q=\%22def+bccz_counterexample(\%22}{\sage{bccz\underscore{}counterexample}}}%
\pgfkeyssetvalue{/sagefunc/drlm_3_slope_limit}{\href{\githubsearchurl?q=\%22def+drlm_3_slope_limit(\%22}{\sage{drlm\underscore{}3\underscore{}slope\underscore{}limit}}}%
\pgfkeyssetvalue{/sagefunc/dg_2_step_mir_limit}{\href{\githubsearchurl?q=\%22def+dg_2_step_mir_limit(\%22}{\sage{dg\underscore{}2\underscore{}step\underscore{}mir\underscore{}limit}}}%
\pgfkeyssetvalue{/sagefunc/rlm_dpl1_extreme_3a}{\href{\githubsearchurl?q=\%22def+rlm_dpl1_extreme_3a(\%22}{\sage{rlm\underscore{}dpl1\underscore{}extreme\underscore{}3a}}}%
\pgfkeyssetvalue{/sagefunc/hildebrand_5_slope_22_1}{\href{\githubsearchurl?q=\%22def+hildebrand_5_slope_22_1(\%22}{\sage{hildebrand\underscore{}5\underscore{}slope\underscore{}22\underscore{}1}}}%
\pgfkeyssetvalue{/sagefunc/hildebrand_5_slope_24_1}{\href{\githubsearchurl?q=\%22def+hildebrand_5_slope_24_1(\%22}{\sage{hildebrand\underscore{}5\underscore{}slope\underscore{}24\underscore{}1}}}%
\pgfkeyssetvalue{/sagefunc/hildebrand_5_slope_28_1}{\href{\githubsearchurl?q=\%22def+hildebrand_5_slope_28_1(\%22}{\sage{hildebrand\underscore{}5\underscore{}slope\underscore{}28\underscore{}1}}}%
\pgfkeyssetvalue{/sagefunc/kzh_7_slope_1}{\href{\githubsearchurl?q=\%22def+kzh_7_slope_1(\%22}{\sage{kzh\underscore{}7\underscore{}slope\underscore{}1}}}%
\pgfkeyssetvalue{/sagefunc/kzh_28_slope_1}{\href{\githubsearchurl?q=\%22def+kzh_28_slope_1(\%22}{\sage{kzh\underscore{}28\underscore{}slope\underscore{}1}}}%
\pgfkeyssetvalue{/sagefunc/extremality_test}{\href{\githubsearchurl?q=\%22def+extremality_test(\%22}{\sage{extremality\underscore{}test}}}%
\pgfkeyssetvalue{/sagefunc/plot_2d_diagram}{\href{\githubsearchurl?q=\%22def+plot_2d_diagram(\%22}{\sage{plot\underscore{}2d\underscore{}diagram}}}%
\pgfkeyssetvalue{/sagefunc/generate_example_e_for_psi_n}{\href{\githubsearchurl?q=\%22def+generate_example_e_for_psi_n(\%22}{\sage{generate\underscore{}example\underscore{}e\underscore{}for\underscore{}psi\underscore{}n}}}%
\pgfkeyssetvalue{/sagefunc/chen_3_slope_not_extreme}{\href{\githubsearchurl?q=\%22def+chen_3_slope_not_extreme(\%22}{\sage{chen\underscore{}3\underscore{}slope\underscore{}not\underscore{}extreme}}}%
\pgfkeyssetvalue{/sagefunc/psi_n_in_bccz_counterexample_construction}{\href{\githubsearchurl?q=\%22def+psi_n_in_bccz_counterexample_construction(\%22}{\sage{psi\underscore{}n\underscore{}in\underscore{}bccz\underscore{}counterexample\underscore{}construction}}}%
\pgfkeyssetvalue{/sagefunc/gomory_fractional}{\href{\githubsearchurl?q=\%22def+gomory_fractional(\%22}{\sage{gomory\underscore{}fractional}}}%
\pgfkeyssetvalue{/sagefunc/not_extreme_1}{\href{\githubsearchurl?q=\%22def+not_extreme_1(\%22}{\sage{not\underscore{}extreme\underscore{}1}}}%
\pgfkeyssetvalue{/sagefunc/bhk_irrational_extreme_limit_to_rational_nonextreme}{\href{\githubsearchurl?q=\%22def+bhk_irrational_extreme_limit_to_rational_nonextreme(\%22}{\sage{bhk\underscore{}irrational\underscore{}extreme\underscore{}limit\underscore{}to\underscore{}rational\underscore{}nonextreme}}}%
\pgfkeyssetvalue{/sagefunc/zhou_two_sided_discontinuous_cannot_assume_any_continuity}{\href{\githubsearchurl?q=\%22def+zhou_two_sided_discontinuous_cannot_assume_any_continuity(\%22}{\sage{zhou\underscore{}two\underscore{}sided\underscore{}discontinuous\underscore{}cannot\underscore{}assume\underscore{}any\underscore{}continuity}}}%
\pgfkeyssetvalue{/sagefunc/drlm_gj_2_slope_extreme_limit_to_nonextreme}{\href{\githubsearchurl?q=\%22def+drlm_gj_2_slope_extreme_limit_to_nonextreme(\%22}{\sage{drlm\underscore{}gj\underscore{}2\underscore{}slope\underscore{}extreme\underscore{}limit\underscore{}to\underscore{}nonextreme}}}%
\pgfkeyssetvalue{/sagefunc/not_minimal_2}{\href{\githubsearchurl?q=\%22def+not_minimal_2(\%22}{\sage{not\underscore{}minimal\underscore{}2}}}%
\pgfkeyssetvalue{/sagefunc/drlm_not_extreme_1}{\href{\githubsearchurl?q=\%22def+drlm_not_extreme_1(\%22}{\sage{drlm\underscore{}not\underscore{}extreme\underscore{}1}}}%
\pgfkeyssetvalue{/sagefunc/kzh_2q_example_1}{\href{\githubsearchurl?q=\%22def+kzh_2q_example_1(\%22}{\sage{kzh\underscore{}2q\underscore{}example\underscore{}1}}}%

\DeclareRobustCommand\sage[1]{\texttt{#1}}

\usepackage[normalem]{ulem}

\usepackage{xtab}
\usepackage{bm}
\usepackage[most]{tcolorbox}
\newtcolorbox{myboxedeq}{
  colback=white, 
  colframe=gray, 
  boxrule=0.5pt, 
  rounded corners, 
  boxsep=5pt, 
  box align=center,
  valign=center,
  enhanced, %
  breakable, %
}

\usepackage{algorithm}

\usepackage{quiver}  %

\usepackage{todonotes}
\newtheorem{theorem}{Theorem}[section]

\makeatletter
\newcommand\MkNewTheorem[2]{%
  \newtheorem{#1}{#2}
  \expandafter\def\csname c@#1\endcsname{\c@theorem}
  \expandafter\def\csname p@#1\endcsname{\p@theorem}
  \expandafter\def\csname the#1\endcsname{\thetheorem}
  \expandafter\def\csname #1name\endcsname{#2}
}

\MkNewTheorem{corollary}{Corollary}
\MkNewTheorem{lemma}{Lemma}
\MkNewTheorem{proposition}{Proposition}
\MkNewTheorem{prop}{Proposition}
\MkNewTheorem{claim}{Claim}
\MkNewTheorem{observation}{Observation}
\MkNewTheorem{obs}{Observation}
\MkNewTheorem{conjecture}{Conjecture}
\MkNewTheorem{openquestion}{Open question}

\theoremstyle{definition}
\MkNewTheorem{example}{Example}
\MkNewTheorem{exercise}{Exercise}
\MkNewTheorem{notation}{Notation}
\MkNewTheorem{assumption}{Assumption}
\MkNewTheorem{definition}{Definition}
\MkNewTheorem{remark}{Remark}
\MkNewTheorem{goal}{Goal}

\usepackage{comment} %

\renewcommand{\footnote}[1]{}

 \usetikzlibrary{patterns}

\newcommand\Autoref[1]{\@first@ref#1,@}
\def\@throw@dot#1.#2@{#1}%
\def\@set@refname#1{%
    \edef\@tmp{\getrefbykeydefault{#1}{anchor}{}}%
    \xdef\@tmp{\expandafter\@throw@dot\@tmp.@}%
    \ltx@IfUndefined{\@tmp autorefnameplural}%
         {\def\@refname{\@nameuse{\@tmp autorefname}s}}%
         {\def\@refname{\@nameuse{\@tmp autorefnameplural}}}%
}
\def\@first@ref#1,#2{%
  \ifx#2@\autoref{#1}\let\@nextref\@gobble%
  \else%
    \@set@refname{#1}%
    \@refname~\ref{#1}%
    \let\@nextref\@next@ref%
  \fi%
  \@nextref#2%
}
\def\@next@ref#1,#2{%
   \ifx#2@ and~\ref{#1}\let\@nextref\@gobble%
   \else, \ref{#1}%
   \fi%
   \@nextref#2%
}

\makeatletter
\let\savemathbb=\mathbb
\let\OurMathBbAux=\mathbb
\DeclareRobustCommand\OurMathBb{\OurMathBbAux}
\let\mathbb=\OurMathBb
\let\bfseries=\undefined
\DeclareRobustCommand\bfseries
{\not@math@alphabet\bfseries\mathbf
  \boldmath\fontseries\bfdefault\selectfont\let\OurMathBbAux=\mathbf}
\def\@thm#1#2#3{%
  \ifhmode\unskip\unskip\par\fi
  \normalfont
  \trivlist
  \let\thmheadnl\relax
  \let\thm@swap\@gobble
  \thm@notefont{\fontseries\mddefault\upshape\unboldmath\let\OurMathBbAux=\savemathbb}%
  \thm@headpunct{.}%
  \thm@headsep 5\p@ plus\p@ minus\p@\relax
  \thm@space@setup
  #1%
  \@topsep \thm@preskip               %
  \@topsepadd \thm@postskip           %
  \def\@tempa{#2}\ifx\@empty\@tempa
    \def\@tempa{\@oparg{\@begintheorem{#3}{}}[]}%
  \else
    \refstepcounter{#2}%
    \def\@tempa{\@oparg{\@begintheorem{#3}{\csname the#2\endcsname}}[]}%
  \fi
  \@tempa
}

\makeatother
\let\originaltau\tau
\renewcommand{\tau}{\boldsymbol{\originaltau}}

\renewcommand{\otimes}{\hspace{1pt}{\textnormal{\texttt{\#}}}\hspace{1pt}}

\renewcommand{\Tuple}{}
\newcommand{\Etuple}{\mathcal{E}}
\usepackage{accents} %

\renewcommand{\p}{\mathtt{p}}
\newcommand{\ZZ}{\mathcal Z}

\usetikzlibrary{arrows.meta}
\usepackage{xtab}

\usetikzlibrary{patterns}

\newcommand{\mytriangle}[4]{%
    \fill[#1, opacity=#2] (#3,#4) -- (#3+1,#4) -- (#3,#4+1) -- cycle;
}
\newcommand{\mytriangleRemove}[2]{%
    \fill[pattern=crosshatch] (#1,#2) -- (#1+1,#2) -- (#1,#2+1) -- cycle;
}

\newcommand{\mytriangleReflect}[4]{%
    \fill[#1, opacity=#2] (#3,#4) -- (#3-1,#4) -- (#3,#4-1) -- cycle;
}

\newcommand{\mytriangleRemoveReflect}[2]{%
    \fill[pattern=crosshatch] (#1,#2) -- (#1-1,#2) -- (#1,#2-1) -- cycle;
}

\newcommand{\myvertEdge}[3]{%
    \draw[#1, line width=1.3] (#2,#3) -- (#2,#3+1);
}

\newcommand{\myhorEdge}[3]{%
    \draw[#1, line width=1.3] (#2,#3) -- (#2+1,#3);
}

\newcommand{\mydiagEdge}[3]{%
    \draw[#1, line width=1.3] (#2,#3) -- (#2+1,#3-1);
}

\newcommand{\mypoint}[3]{%
    \fill[#1] (#2,#3) circle[radius=3pt];
}
\newcommand{\mypointHole}[3]{%
    \fill[#1, draw = black] (#2,#3) circle[radius=3pt];
}

\newcommand{\myscale}{0.35}
\newcommand\tikzarrow{\quad\begin{tikzpicture}[scale=\myscale]
    \node at (0,5) {}; %
    \node at (0,0 ) {}; %
    \node at (0,2.5) {\Large$\Rightarrow$};
\end{tikzpicture}\quad}

\newcommand\tikzLRarrow{\quad\begin{tikzpicture}[scale=\myscale]
    \node at (0,5) {}; %
    \node at (0,0 ) {}; %
    \node at (0,2.5) {\Large$\Leftrightarrow$};
\end{tikzpicture}\quad}

\newcommand{\drawGridred}{    %
    \diagGrid[{}{}{}{}]
    \foreach \x in {0,1,2,3,4,5} {
        \foreach \y in {0,1,2,3,4,5} {
            \mypoint{red}{\x}{\y}
        }
    }}

\newcommand{\drawGridredNoLines}{    %
    \foreach \x in {0,1,2,3,4,5} {
        \foreach \y in {0,1,2,3,4,5} {
            \mypoint{red}{\x}{\y}
        }
    }}
    
\newcommand{\additive}{green}
\newcommand{\edgeadditive}{green!30!black}
\newcommand{\zero}{red}

\usepackage{makecell}

\newcommand{\stepTwo}{

\begin{tikzpicture}[scale=\myscale]
\drawGridred
    \mytriangle{\additive}{0.4}{2}{1} %
    \mytriangleReflect{\additive}{0.4}{3}{5} %
    \mytriangle{\additive}{0.4}{0}{3} %
\end{tikzpicture}
\tikzLRarrow
\begin{tikzpicture}[scale=\myscale]
\drawGridred
    \mytriangle{\zero}{0.4}{2}{1} %
    \mytriangleReflect{\zero}{0.4}{3}{5} %
    \mytriangle{\zero}{0.4}{0}{3} %
\end{tikzpicture}

\begin{tikzpicture}[scale=\myscale]
\drawGridred
    \mydiagEdge{\additive}{2}{1} 
    \mydiagEdge{\additive}{1}{3}
    \mydiagEdge{\additive}{4}{4} 
\end{tikzpicture}
\tikzLRarrow
\begin{tikzpicture}[scale=\myscale]
\drawGridred
    \mydiagEdge{\zero}{2}{1} 
    \mydiagEdge{\zero}{1}{3}
    \mydiagEdge{\zero}{4}{4} 
\end{tikzpicture}
}

\usepackage{array}

\title[Equivariant Perturbation IV]{Equivariant Perturbation in \\Gomory and Johnson's Infinite Group
  Problem.\\ IV. The General Unimodular Two-Dimensional Case}

\author{Robert Hildebrand}
\address{Robert Hildebrand: Grado Department of Industrial and Systems Engineering}
\email{rhil@vt.edu}

\author{Matthias K\"oppe}
\address{Matthias K\"oppe: Dept.\ of Mathematics, University of California, Davis}
\email{mkoeppe@math.ucdavis.edu}

\author{Luze Xu}
\address{Luze Xu: Dept.\ of Mathematics, University of California, Davis}
\email{lzxu@ucdavis.edu}

\begin{document}

\begin{abstract}We study an abstract setting for cutting planes for integer programming called the infinite group problem.  In this abstraction, cutting planes are computed via \emph{cut generating function} that act on the simplex tableau.  In this function space, cut generating functions are classified as minimal, extreme, and facets as a proxy for understanding the strength or potential importance of these functions.   Prior work developed algorithms for testing minimality, extremality, and facetness for cut generating functions applied to 1-row tableau and to some 2-row tableau in a restricted setting.  We complement and generalize this work by  giving an algorithm for testing the extremality of a large class of
  minimal valid functions for the two-dimensional infinite group problem.  
  Along the way, we develop results of independent interest on functional equations and infinite systems of linear equations.
\end{abstract}

\maketitle

\section{Introduction}
Since the mid-1990s, the Gomory mixed integer cut \cite{gomory1963algorithm, Gom60} has undergone a resurgence in the field of mixed-integer optimization. With advancements in computational techniques, these cuts, once deemed impractical, proved to be highly effective in combination with modern implementations of the simplex method and branch-and-bound techniques \cite{balas1993lift, balas1996mixed}. This unexpected discovery spurred renewed interest in general-purpose cutting planes, leading to significant breakthroughs in mixed-integer programming technology \cite{cornuejols2007revival}.

A powerful framework for understanding and generating cutting planes in integer programming has since emerged from the foundational work on the \textit{infinite group problem}, introduced by Gomory and Johnson in 1972 \cite{infinite, infinite2}. This framework has become central to the study of cut-generating functions, which seek to generalize cutting-plane techniques beyond single-row formulations and into multi-row settings \cite{conforti2013cut}. The infinite group problem captures the essence of cutting-plane generation by abstracting the underlying structure of integer programs into a geometric and algebraic problem that can be analyzed independently of any specific optimization problem.

The recent emphasis on multi-row cuts, derived from the infinite group problem, has further advanced the field. These cuts hold significant promise for solving large-scale mixed-integer programs by providing a more versatile and powerful toolset for cutting-plane generation. The framework of cut-generating functions, developed in part from Gomory and Johnson's work, now serves as a key approach in both the theoretical study and practical implementation of cutting planes \cite{conforti2013cut}. This line of research has proven to be pivotal in bridging the gap between theory and computation, with significant implications for the continued advancement of integer programming algorithms.

For a more comprehensive historical account of cutting planes and their development, see the survey \cite{igp_survey, igp_survey_part_2}.

 The study of cut-generating functions for $k=1$ is referred to
  as the {\em single-row} problem, and the general $k\geq 2$ case is referred
  to as the {\em multi-row} problem in the literature. Algorithms used in practice for solving mixed-integer problems have so far used only insights from the single-row problem. It is believed that the general multi-row analysis can lead to stronger cutting planes that can significantly boost the performance of state-of-the-art algorithms.

Cut generating functions are classified as \emph{valid}, \emph{minimal}, \emph{extreme}, or \emph{facet}, with extreme and facet being the most desirable properties.  See Section~\ref{sec:basic-notions} for definitions.  

\subsection{Recent Literature}
Since the publication of the survey articles \cite{igp_survey, igp_survey_part_2}, numerous advancements have addressed open questions and significantly expanded the research field on cut-generating functions. Notably, substantial software was developed to test the extremality of functions~\cite{hong-koeppe-zhou:software-paper}. These advancements have led to new discoveries about the extreme functions’ potential for unusual or ``wild" behavior, as documented in several studies~\cite{koeppe-zhou:crazy-perturbation, koeppe-zhou:param-full-paper, koeppe-zhou-ISCO, Kppe2016, Basu2018, Basu2020}. 

Efforts to better understand the structure and properties of these functions also progressed, with detailed examinations of their functional spaces presented in~\cite{hildebrand-koeppe-zhou:algo-paper-abstract-ipco, OJMO_2022__3__A5_0, Kppe2017}. Additionally, theoretical analyses have provided new insights into the strength of cut-generating functions~\cite{Basu2019}. Finally, several computational studies have been done on how to find effective cuts efficiently~\cite{cheng-basu-2024}.

Although it is clear that multi-row cuts are stronger than single row cuts, a key question seems to remain: How can we efficiently compute strong multi-row cuts that will be effective in solvers?

\subsection{Contributions}
In the paper, we generalize results from~\cite{bhk-IPCOext} about the 2-row problem by removing restrictive assumptions.  In particular, we show how to analyze piecewise linear cut generating functions and show that they are extreme and facets.  This is done by establishing a correspondence of infinite settings to finite  settings.  This results in an algorithm that proves which piecewise linear functions are extreme and facets in the 2-row setting. See \autoref{thm:extreme-restriction-m-2d-all} for our main result.  Along the way, we prove new results of potentially independent interest about functional equations and infinite systems of linear equations.  

\subsection{Outline} 
We attempt to keep this paper self-contained.  As such, we devote a significant portion of the paper to reviewing relevant definitions and results as presented in~\cite{igp_survey, igp_survey_part_2}.  That said, some of the tools used require some generalization to approach this more difficult problem.  At a glance, we extensively discuss necessary prior work and notation in Section~\ref{sec:basic-notions}. We save other sections for new developments.  

Our main result is \autoref{thm:extreme-restriction-m-2d-all}, which is completed in Section~\ref{sec:main-proof}.  For this, along the way, we develop novel techniques and results including new results on functional equations, and a discretization lemma that reduces structured infinite systems of equations to finite systems of equations.  We work to abstract these types of results so that they can be appreciated and understood independently from the cut-generating function application.

 In more detail, the paper is laid out as follows.

Section~\ref{sec:basic-notions} mostly reviews prior work that is important for this paper.  To begin, it formally introduces the problem, the main objects of study such as {\em valid functions, minimal valid functions, extreme functions, and facets}, and their basic properties.  
Section~\ref{sec:piecewise} introduces the notations and concepts from discrete geometry required for analyzing the problem, and collects foundational techniques for the general $k$-row problem. 
Section~\ref{sec:7-tuples-intro} introduces the perturbation space in the notation of generalized additivities that we refer to as 7-tuples.
Section~\ref{sec:foundations} recalls higher-dimensional variants of the celebrated {\em Interval Lemma}.
Section~\ref{sec:alg-restriction-finite-groups} recalls a result from \cite{igp_survey_part_2} that identifies the extremality of functions for the 1-row problem with a finite group problem.  

We now begin new results for this paper.

Section~\ref{sec:7-tuples} builds notation to study general types of additivities for piecewise linear functions and characterizes the types of important generalized additivities with respect to a 2-dimension polyhedral complex.
Section~\ref{sec:interval-lemma-applications-new-results} applies interval lemma results to our context.  We also develop a new result that is a   variant that we refer to as the \emph{hidden interval lemma} (\autoref{lem:hidden_interval_lemma}).  This lemma combines different additivities to reveal an underlying interval lemma. We are not aware of any result of this kind, even in the functional analysis literature.
Section~\ref{sec:underlying-theory} shows how some infinite systems of equations can be understood by studying only a finite system of equations.  We work in a general notation and thus this can be read independently of the other sections.  We present these results (\autoref{lem:infinte-system-to-finite-system-edges}) as such to highlight the underlying structures that are driving some of the main results.
Section~\ref{sec:2D-main-result} develops results on the perturbation function space.
Finally, in Section~\ref{sec:main-proof}, we prove our main result (\autoref{thm:extreme-restriction-m-2d-all}) that improves the work of \cite{bhk-IPCOext}.%

\section{The Infinite Group Problem}\label{sec:basic-notions}

Gomory and Johnson, in their work
in~\cite{infinite,infinite2}, introduced the so-called {\em infinite group problem}. %
It has its roots in Gomory's \emph{group problem}~\cite{gom}, which was
introduced by him as an algebraic relaxation of pure integer linear
optimization problems. 
We introduce this next as it will be useful for formulating many of our results in a unified language. One considers an
abelian group $G$, written additively, and studies
the set $R_{\ve f}(G,S)$ of functions $y \colon G \to \R$ satisfying the following constraints:  
\begin{equation}
  \label{GP} 
  \begin{aligned}
    &\sum_{\rx \in G} \rx\, y(\rx) \in \ve f + S \\
    &y(\rx) \in \mathbb{Z}_+ \ \ \textrm{for all $\rx \in G$}  \\
    &y \textrm{ has finite support}, 
  \end{aligned}
\end{equation}
where $S$ is a subgroup of $G$ and $\ve f$ is a given element in $G\setminus S$; so $\ve
f + S$ is the coset containing the element~$\ve f$. 
We are interested in studying the convex hull $R_{\ve f}(G,S)$ of
the set of all functions $y\colon G \to \R$ satisfying the constraints in~\eqref{GP}. $R_{\ve f}(G,S)$ is a convex subset of the vector space $\R^{(G)}$, which is infinite-dimensional when $G$ is an infinite group, i.e., of infinite order. The nomenclature {\em $k$-row infinite group problem} is reserved for the situation when $G =\R^k$ is taken to be the group
of real $k$-dimensional vectors under addition, and $S= \Z^k$ is the subgroup of the
integer vectors%
. When $k=1$, we refer to it as the {\em single-row infinite group problem}.

 \subsection{Valid inequalities and valid functions} 
 
Following Gomory and Johnson, we are interested in the description of $R_{\ve f}(G,S)$ as
the intersection of halfspaces in $\R^{(G)}$.  We first describe the general
form that these halfspaces take and then a standard normalization that leads
to the idea of cut-generating functions.  

\subsubsection{Valid inequalities}
Any halfspace in $\R^{(G)}$ is given by a pair $(\pi, \alpha)$, where $\pi \in \R^G$ and $\alpha \in
\R$, and the halfspace is the set of all $y \in \R^{(G)}$ that satisfy $\sum_{\rx \in G} \pi(\rx)y(\rx) \geq
\alpha$. The left-hand side of the inequality is a finite sum because $y$ has finite
support. Such an inequality is called a {\em valid inequality} for $R_{\ve f}(G,S)$
if $\sum_{\rx \in G} \pi(\rx)y(\rx) \geq \alpha$ for all $y \in R_{\ve
  f}(G,S)$, i.e., $R_{\ve f}(G,S)$ is contained in the halfspace defined by
$(\pi, \alpha)$. Note that the set of all valid inequalities $(\pi,\alpha)$ is a cone in
the space $\R^G \times \R$.

\subsubsection{Valid functions} Since data in finite-dimensional integer programs is usually rational, and this is our main motivation for studying the infinite group problem, it is
customary to concentrate on valid inequalities with $\pi \geq 0$; then we can choose, after a scaling, $\alpha = 1$ (otherwise, the inequality is implied by the nonnegativity of $y$). Thus, we only focus on
valid inequalities of the form $\sum_{\rx \in G} \pi(\rx)y(\rx) \geq 1$ with $\pi \geq 0$. Such functions $\pi \in \R^G$ are called {\em valid functions} for $R_{\ve f}(G,S)$. We remind the reader that this choice comes at a price because of {\cite[Proposition  2.4]{igp_survey, igp_survey_part_2}}; however, it can be shown that for rational {\em corner polyhedra}, which form an important family of relaxations for integer programs, all valid inequalities are restrictions of nonnegative valid functions for the infinite group problem. See~\cite{corner_survey} for a discussion.

 \subsection{Minimal functions, extreme functions and facets}\label{s:minimal-def}
We summarize a hierarchy on the set of valid
 functions defined by Gomory and Johnson~\cite{infinite,infinite2} that  captures the strength of the corresponding valid inequalities.

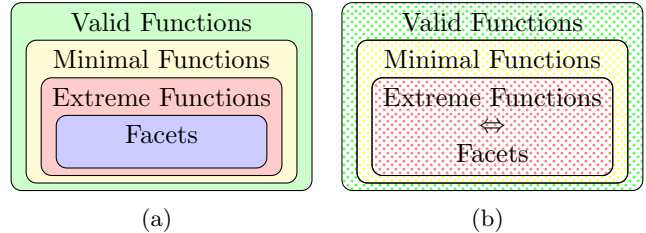
\begin{figure}[H]
  \subfigure[]{\label{subfig:infinite-hierarchy}
  \begin{tikzpicture}
  [font=\small,
  valid/.style={rounded corners,fill=green!20},
  minimal/.style={rounded corners,fill=yellow!20},
  extreme/.style={rounded corners,fill=red!20},
  facet/.style={rounded corners,fill=blue!20}]
\draw[valid] (0,0)--(0,2.5)--(4,2.5)--(4,0)--cycle; %
\draw[minimal] (0.2,0.1)--(0.2,2)--(3.8,2)--(3.8,0.1)--cycle; %
\draw[extreme] (0.4,0.2)--(0.4,1.5)--(3.6,1.5)--(3.6,0.2)--cycle; %
\draw[facet] (0.6,0.3)--(0.6,1)--(3.4,1)--(3.4,0.3)--cycle; %
  \node at (2, 2.25) { Valid Functions};
  \node at (2, 1.75) { Minimal Functions};
  \node at (2, 1.25) { Extreme Functions};
  \node at (2, 0.75) { Facets};
\end{tikzpicture}
}
\subfigure[]{\label{subfig:finite-hierarchy}
   \begin{tikzpicture}
  [font=\small,
  valid/.style={rounded corners,fill=green!30,pattern=crosshatch dots,pattern color=green!70},
  minimal/.style={rounded corners,fill=yellow!20,pattern=crosshatch dots,pattern color=yellow!70},
  extreme/.style={rounded corners,fill=red!20,pattern=crosshatch dots,pattern color=red!50},
  stain/.style={rounded corners,fill=white}]

\draw[valid] (0,0)--(0,2.5)--(4,2.5)--(4,0)--cycle; %
\draw[stain] (0.2,0.1)--(0.2,2)--(3.8,2)--(3.8,0.1)--cycle;
\draw[minimal] (0.2,0.1)--(0.2,2)--(3.8,2)--(3.8,0.1)--cycle; %
\draw[stain] (0.4,0.2)--(0.4,1.5)--(3.6,1.5)--(3.6,0.2)--cycle;
\draw[extreme] (0.4,0.2)--(0.4,1.5)--(3.6,1.5)--(3.6,0.2)--cycle; %
  \node at (2, 2.25) { Valid Functions};
  \node at (2, 1.75) { Minimal Functions};
  \node at (2, 1.25) { Extreme Functions};
  \node at (2, 0.85) { $\Leftrightarrow$};
  \node at (2, 0.5) { Facets};

\end{tikzpicture}
}
\caption{The hierarchy of valid, minimal, and extreme functions and facets
    and weak facets. (a) General case. (b) Situation in the finite-dimensional
    case.   As a consequence of this work, we can show some cases of 1-row and 2-row  continuous piecewise linear functions with
    rational breakpoints where this correspondence also holds.
  }
  \label{fig:hierarchy-with-facets}
  \label{fig:hierarchy}
\end{figure}

\subsubsection{Minimal functions}
A valid function $\pi$ for $R_{\ve f}(G,S)$ is said to be
\emph{minimal} for $R_{\ve f}(G,S)$ if there is no valid function $\pi' \neq \pi$
such that $\pi'(\rx) \le \pi(\rx)$ for all $\rx \in G$.  For every valid
function $\pi$ for $R_{\ve f}(G,S)$, there exists a minimal valid function $\pi'$
such that $\pi' \leq \pi$ \cite[Theorem~1.1]{bhkm}, and thus non-minimal valid
functions are redundant in the description of $R_{\ve f}(G,S)$.  

A function $\pi\colon G \rightarrow \mathbb{R}$ is \emph{subadditive} if
$\pi(\x + \y) \le \pi(\x) + \pi(\y)$ for all $\x,\y \in G$. We say that  $\pi$ is
\emph{symmetric} (or \emph{satisfies the symmetry condition}) if $\pi(\x) +
\pi(\f - \x) = 1$ for all $\x \in G$.

\begin{theorem}[Gomory and Johnson \cite{infinite}] \label{thm:minimal} Let $G$ be an abelian group, $S$ be a subgroup of $G$ and $\f \in G\setminus S$. Let
  $\pi \colon G \rightarrow \mathbb{R}$ be a nonnegative function. Then $\pi$
  is a minimal valid function for $R_{\ve f}(G,S)$ if and only if $\pi(\ve z) = 0$ for
  all $\ve z\in S$, $\pi$ is subadditive, and $\pi$ satisfies the symmetry
  condition. (The first two conditions imply that $\pi$ is periodic modulo
  $S$, that is, $\pi(\x) = \pi(\x + \ve z)$ for all $\ve z \in S$, and the symmetry condition implies that the values of minimal functions are bounded between $0$ and $1$.) 
\end{theorem}

See \cite[Theorem~5.4]{corner_survey} for a proof. Note that for $R_\f(\R^k, \Z^k)$, any minimal function is then periodic with respect to $\Z^k$.  

\begin{figure}[H]%

\begin{center}
\includegraphics[width=.31\linewidth]{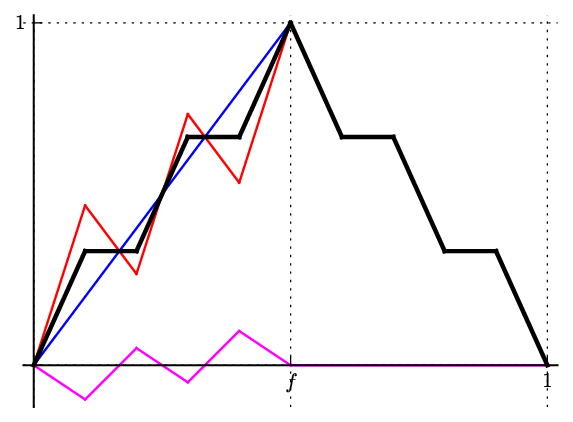}
\end{center}
\caption{This function %
is
  minimal, but not extreme (and  hence also not a facet).
  We can see this as 
  for any distinct minimal $\pi^1 = \pi + \bar\pi$ (\emph{blue}), $\pi^2 = \pi
  - \bar\pi$ (\emph{red}) such that $\pi = \tfrac{1}{2}\pi^1
  + \tfrac{1}{2} \pi^2$, the functions $\pi^1$ and $\pi^2$ are continuous
  piecewise linear with the same breakpoints as $\pi$. %
  A finite-dimensional extremality
  test can be used to find the two linearly independent perturbations $\bar\pi$ (\emph{magenta}), as shown in~\cite{basu-hildebrand-koeppe:equivariant}.
}
\label{fig:minimalNotExtreme}
\end{figure}

\subsubsection{Extreme functions}

A~valid function~$\pi$ is \emph{extreme}
for $R_{\ve f}(G,S)$ if it cannot be written as a convex combination of two other
valid functions for $R_{\ve f}(G,S)$, i.e., $\pi = \tfrac12(\pi^1 + \pi^2)$
implies $\pi = \pi^1 = \pi^2$ (see \autoref{fig:hierarchy} and \autoref{fig:minimalNotExtreme}).  
Extreme functions are easily seen to be minimal. In fact we may view this definition from a convex geometry perspective. By \autoref{thm:minimal}, the set of minimal valid functions is a convex subset of the infinite-dimensional space $\R^G$ of real-valued functions on $G$; this follows from the observation that all the properties in \autoref{thm:minimal} are preserved undertaking convex combinations of functions. %
\subsubsection{Facets}
A related notion is that of a \emph{facet}.  
Let $P(\pi)$ denote the set of all feasible solutions $y \in \R^{(G)}$ satisfying~\eqref{GP} such that $\sum_{r \in G} \pi(r) y(r) = 1$.  A valid function $\pi$ is called a \emph{facet} if for every valid function $\pi'$ such that $P(\pi) \subseteq P(\pi')$ we have that $\pi' = \pi$, as defined in~\cite{tspace}. Equivalently, a valid function $\pi$ is a facet if this condition holds for all such \emph{minimal} valid functions $\pi'$  (cf.~\cite{bhkm}).

\subsubsection{Relation between the three notions}
\label{subsec:relation-between-three-notions}
Facets are extreme functions (cf.~\cite[Lemma 1.3]{bhkm}), but it is unknown if all
extreme functions are facets.  %
See \autoref{fig:hierarchy-with-facets}\,(a). When $G$ is a finite abelian group, the set of minimal functions is a finite-dimensional polyhedron (given by constraints coming from \autoref{thm:minimal}).%
In this setting, it is well known that the notions of
facets and extreme inequalities are equivalent, and form the extreme points of
this polyhedron; see \autoref{fig:hierarchy-with-facets}\,(b).

\subsection{A roadmap for proving extremality and facetness}\label{s:roadmap}

An understanding of the set of points for which the subadditivity relations of a minimal function hold at equality is crucial to the study of both extreme functions and facets. This motivates the following definition.\begin{definition}
Let $\pi\colon \R^k \to \R$.
  For $\ve \sigma \in \R^3$, define the function of~$\ve \sigma\cdot \pi \colon \R^k \times\R^k \times \R^k \to \R$ as 
  \begin{equation}
     \label{eq:sigma-pi-definition}
     \ve \sigma \cdot \pi(\x^1,\x^2,\x^3) = \sum_{i=1}^3 \sigma_i \pi(\x^i)
   \end{equation}
  and the 
  \emph{additivity domain} of~$\pi$ as
  \begin{equation}
    \label{eq:Epi}
    E(\pi) := \setcond{(\x^1, \x^2) \in \R^k \times \R^k} { \ve (1,1,-1) \cdot \pi(\x^1,\x^2,\x^1 + \x^2) = 0}.
  \end{equation}
\end{definition}

\subsubsection{Finite test for minimality}\label{section:minimalityTest}
  
One of the main advantages of working with minimal valid functions that are piecewise linear is their combinatorial structure, which avoids many analytical complexities. Moreover, it is possible to give a finite description of~$\pi$ and $\Delta \pi$.  
From this, one can show a simple algorithm for testing minimality by analyzing the $(1,1,-1) \cdot \pi$ function \cite[Theorem 3.10, Remark 3.11]{bhk-IPCOext}].

\subsubsection{Conditions for Extreme and Facet}

The main technique used to show a function $\pi$ is extreme is to assume that $\pi = \tfrac{1}{2}(\pi^1 + \pi^2)$ where $\pi^1, \pi^2$ are valid functions, and then show that $\pi = \pi^1 = \pi^2$. 
Depending on the properties of the function being tested, various lemmas regarding continuity can be employed to deduce properties of related minimal valid functions. 

As noted in \cite[Theorem 1.4]{bhk-IPCOext} and \cite[Lemma 2.11]{igp_survey}, the following results we proven  in~\cite[Lemma 1.4]{MR0479416} and~\cite[Theorem 2.9]{basu-hildebrand-koeppe:equivariant} for the 1-row case, but easily extend to the $k$-row case.
\begin{theorem}[\cite{MR0479416,basu-hildebrand-koeppe:equivariant}]
\label{lem:minimality-of-pi1-pi2}
\label{Theorem:functionContinuous-all-cases}
If $\pi \colon \R^k \to \R$ is a minimal valid function, and $\pi = \frac{1}{2}\pi^1 + \frac{1}{2}\pi^2$, where $\pi^1, \pi^2$ are valid functions, then $\pi^1, \pi^2$ are both minimal. Moreover, if $\limsup_{\ve h\to \0} \frac{\lvert\pi(\ve h)\rvert}{\lvert \ve h\rvert} < \infty$, then this condition also holds for $\pi^1$ and $\pi^2$. This implies that $\pi, \pi^1$ and $\pi^2$ are all Lipschitz continuous.
\end{theorem}
Conveniently, if $\pi$ is a continuous piecewise linear function, then the condition $\limsup_{\ve h\to \0} \frac{\lvert\pi(\ve h)\rvert}{\lvert \ve h\rvert} < \infty$ is satisfied.

Suppose we choose a $\bar \pi \colon \R^k \to \R$ such that $\pi^1 := \pi + \bar \pi$ and $\pi^2:= \pi - \bar \pi$ are minimal valid functions.  Then we call $\bar \pi$ a \emph{perturbation function}.  Perturbation functions provide a more concise approach to studying extremality of $\pi$, and many properties required of $\pi^1, \pi^2$ also hold for $\bar \pi$. 
Indeed, following \autoref{lem:minimality-of-pi1-pi2}, if $\pi$ is a minimal valid function that is continuous (or just that the limit of the theorem holds), then $\bar \pi$ must also be continuous.
For convenience, we will stick to this setting where $\pi$ is continuous, and hence all possible perturbation functions are also continuous.

The following are some other related conclusions established in prior research that we will not need to reference.
\begin{enumerate}
    \item Subadditivity relations that hold tightly for \(\pi\) also hold tightly for \(\pi^1\) and \(\pi^2\), i.e., \(E(\pi) \subseteq E(\pi^1) \cap E(\pi^2)\)~\cite[Proof of Theorem 3.3]{infinite}.
    \item In the case \(k = 1\), if \(\pi\) is continuous from one side at zero, then \(\pi^1\) and \(\pi^2\) maintain continuity at all points where \(\pi\) is continuous~\cite[Theorem 2]{dey1}.
\end{enumerate}

We will assume that the input function is continuous, which will simplify some of our work.

To prove that a valid inequality is a facet, the main tool is the so-called \emph{Facet Theorem}, originally proved by Gomory and Johnson~\cite{tspace}  for the one-row case; it extends verbatim to
the $k$-row case.\footnote{Gomory and Johnson's original proof actually holds only for weak facets, and not for facets as claimed in~\cite{tspace}.} We present a stronger version of the theorem, which
first appeared in \cite{bhkm}.\footnote{In contrast to Gomory--Johnson's Facet Theorem, 
  the condition that $E(\pi) \subseteq E(\pi')$ implies $\pi' = \pi$ only
  needs to be tested on minimal valid functions, not all valid functions.} 
\begin{theorem}[Facet Theorem \cite{tspace}, {\cite[Theorem 3.1]{bhkm}}]
\label{thm:facet}
Let $\pi$ be a minimal valid function.  Suppose for every minimal valid function $\pi'$, $E(\pi) \subseteq E(\pi')$ implies $\pi' = \pi$.  Then $\pi$ is a facet.
\end{theorem}

As noted in~\cite{igp_survey}, we can use~\autoref{thm:facet} to prove that all facets are extreme.

The condition that $E(\pi) \subseteq E(\pi')$ implies $\pi' = \pi$ for every minimal valid function $\pi'$ is established along the following lines. First, structural properties of $\pi$ can be used to obtain a structured description of $E(\pi)$.  For example, the fact that $\pi$ is piecewise linear often shows that $E(\pi)$ is the union of many full-dimensional convex sets. $E(\pi')$ shares this structure with $E(\pi)$ because of the assumption that $E(\pi) \subseteq E(\pi')$. Then, results such as the {\em Interval Lemma}, discussed in \autoref{sec:foundations}, are used to show that $\pi'$ must be affine on the set of points contributing to $E(\pi')$. Finally, the conditions that all minimal valid functions are $0$ at the origin and $1$ at $\f + \Z^k$ puts further restrictions on the values that $\pi'$ can take, and ultimately force $\pi' = \pi$.\footnote{Sometimes certain continuity arguments need to be made, where results like Lemma~\ref{lem:tightness} (iii), (iv) and (v) are helpful. In such situations, the proof of extremality is usually slightly simpler than a proof for facetness, owing to Lemma~\ref{lem:tightness} (iii); see \autoref{rem:facet-extreme} and \autoref{rem:facet-non-piece}.}

\subsection{The $k$-dimensional theory of piecewise linear minimal valid functions}
\label{sec:piecewise}

\subsubsection{Polyhedral complexes and piecewise linear functions}\label{sec:def-complex-piecewise}

We introduce the notion of polyhedral complexes, which serves two purposes.
First, it provides a framework to define piecewise linear functions,
generalizing the familiar situation of functions of a single real variable. 
Second it is a tool for studying subadditivity and additivity relations of these functions.
This exposition follows \cite{bhk-IPCOext} with some updates to work with some more generalized notation.

\begin{definition}
\label{def:polyhedralComplex}
A (locally finite) {\em polyhedral complex} is a collection $\P$ of polyhedra in $\R^k$ such that:
\begin{enumerate}[\rm(i)]
\item $\emptyset \in \P$,
\item if $I \in \P$, then all faces of $I$ are in $\P$,
\item the intersection $I \cap J$ of two polyhedra $I,J \in \P$ is a face of both $I$ and $J$,
\item any compact subset of $\R^k$ intersects only finitely many faces in $\P$.
\end{enumerate}
A polyhedron $I$ from $\P$ is called a {\em face} of the complex.
A polyhedral complex~$\P$ is said to be {\em pure} if all its maximal faces (with
respect to set inclusion) have the same
dimension. In this case, we call the maximal faces of $\P$ the {\em cells}
of~$\P$. The zero-dimensional faces of $\P$ are called {\em vertices} and the set of vertices of $\P$ will be denoted by $\verts(\P)$. A polyhedral complex~$\P$ is said to be {\em complete} if the union of all faces of
the complex is~$\R^k$. A pure and complete polyhedral complex $\P$ is called a {\em triangulation} of $\R^k$ if every maximal cell is a simplex.  
\end{definition}

\begin{example}[Breakpoint intervals in $\R^1$ \cite{basu-hildebrand-koeppe:equivariant}]\label{ex:1d-breakpoint-complex}
  Let $0=x_0 < x_1 < \dots < x_{n-1} < x_n=1$ be a list of %
  ``breakpoints'' in $[0,1]$.  
  We extend it periodically as
  \begin{math}
    \B = \{\, x_0 + t, x_1 + t, \dots, x_{n-1}+t\st
    t\in\Z\,\}
  \end{math}.
  Define the set of 0-dimensional faces to be the collection of singletons,
  \begin{math}
    \I_{\B,\EquiOneDimPoint} = \bigl\{\, \{ x \} \st x\in B\,\bigr\},
  \end{math}
  and the set of one-dimensional faces to be the collection of closed intervals,
  \begin{math}
    \I_{\B,\EquiOneDimEdge} = \bigl\{\, [x_i+t, x_{i+1}+t] \st i=0, \dots, {n-1} \text{
      and } t\in\Z \,\bigr\}. 
  \end{math}
  Then $\I_{\B} = \{\emptyset\} \cup \I_{\B,\EquiOneDimPoint} \cup \I_{\B,\EquiOneDimEdge}$ is a
  locally finite %
  polyhedral
  complex%
  .
\end{example}

\begin{example}[Standard triangulations of~$\R^2$ \cite{bhk-IPCOext}]\label{ex:2d-standard-triangulation}
   Let $q$ be a positive integer.  Consider the
arrangement~$\mathcal H_q$ of all hyperplanes (lines) of~$\R^2$ of the form
$\ColVec{0}{1}\cdot \x = b$, $\ColVec{1}{0}\cdot \x = b$, and $\ColVec{1}{1}\cdot\x  = b$,
where $b \in \tfrac{1}{q}\Z$.  The complement of the arrangement~$\mathcal
H_q$ consists of two-dimensional cells, whose closures are the triangles
$$\FundaTriangleLower = \tfrac1q \conv(\{ \ColVec{0}{0}
, \ColVec{1}{0}
, \ColVec{0}{1}
\})\qquad\text{and}\qquad \FundaTriangleUpper = \tfrac1q \conv(\{\ColVec{1}{0}
, \ColVec{0}{1}
,
\ColVec{1}{1}
\})$$ and their translations by elements of the lattice $\smash[t]{\frac1q\Z^2}$. 
We denote by $\P_q$ the collection of these triangles and the vertices and
edges that arise as intersections of the triangles, and the empty set.  Thus $\P_q$ is a locally finite
polyhedral complex%
.  Since all
nonempty faces of~$\P_q$ are simplices, it is a triangulation of the
space~$\R^2$.
\end{example}

We give a precise definition of affine linear functions over a domain, suitable for
the general $k$-dimensional case.
\begin{definition}
Let $U \subseteq \R^k$. We say $\pi\colon U \to \R$ is {\em affine (or affine linear)} over $U$
if there exists a \emph{gradient} $\cve \in \R^k$ such that 
for any $\u_1, \u_2 \in U$ we have $$\pi(\u_2) - \pi(\u_1) =  \cve\cdot (\u_2 - \u_1).$$ 
\end{definition}

Given a pure and complete polyhedral complex $\P$, we call a function $\pi\colon \R^k\to \R$ {\em piecewise linear} over $\P$ if it is affine linear over the {\em relative interior} of each face of the complex. Under this definition, piecewise linear functions can be discontinuous. We say the function $\pi$ is \emph{continuous piecewise linear over $\P$} if it is affine over each of the cells of~$\P$ (thus automatically imposing continuity). Most of the results presented in this survey will be about continuous piecewise linear functions. %

Motivated by Gomory--Johnson's characterization of minimal valid functions
(\autoref{thm:minimal}), we are interested in functions~$\pi\colon \R^k\to \R$
that are periodic modulo~$\Z^k$, i.e., for all $\ve x\in \R^k$ and all vectors 
$\ve t\in\Z^k$, we have $\pi(\ve x+\ve t) = \pi(\ve x)$.  If $\pi$ is periodic modulo~$\Z^k$
and continuous piecewise linear over a pure and complete complex~$\P$, then we
can assume without loss of generality that $\P$ is
also \emph{periodic modulo~$\Z^k$}, i.e., for all $I\in \P$ and all vectors
$\ve t\in\Z^k$, the translated polyhedron~$I + \ve t$ also is a face of~$\P$.
This is the case in Examples \ref{ex:1d-breakpoint-complex} and~\ref{ex:2d-standard-triangulation}.

\begin{remark}[{\cite[Remark 3.5]{igp_survey, igp_survey_part_2}}]
\label{rem:p-finite} If all the cells of the polyhedral complex
  are bounded, the value of a continuous piecewise linear function at any
  point $\x$ can be obtained by interpolating the values of the function at
  the vertices of the minimal face containing $\x$. 
  Moreover, for a periodic continuous piecewise
  linear function over a periodic complex, we can give a finite description
  for $\pi$ by further restricting to the values in $\verts(\P) \cap  D$ where
  $ D = [0,1]^k$ or any set such that $ D + \Z^k = \R^k$. The finiteness of
  the set $\verts(\P) \cap  D$ is guaranteed by the assumption of local
  finiteness in \autoref{def:polyhedralComplex} (iv).
\end{remark}

\subsection{7-Tuples and Perturbation Space: Generalized types of additivities}
\label{sec:7-tuples-intro}
We begin by lifting the set of additivities to a notation we refer to as 7-tuples.
\begin{equation}
\label{eq:mathcalEpi}
\mathcal E(\pi) := \{ (\x^1, \x^2, \x^1 + \x^2, (1,1,-1), \0) : (\x^1, \x^2) \in E(\pi)\}
\end{equation}
In this setting, we will think of the point $(\x^1, \x^2, \x^1 + \x^2, (1,1,-1), \0)$ as a polyhedron or convex set.

More generally, we will think of sets of points that have structured additivities defined by a 7-tuple $(I_1, I_2, I_3, \ve \sigma, \t)$ where 
 $I_i \subseteq \R^k$ be convex sets, $\ve \sigma \in \{-1,1\}^3$ for $i=1,2,3$, and $\t \in \R^k$. Next, define the convex set in $\R^k \times \R^k \times \R^k$
 \begin{equation}
 F(I_1, I_2, I_3, \ve \sigma, \t)
 = \{(\x^1, \x^2, \x^3) \in  I_1 \times I_2 \times I_3 \st \sum_{i=1}^3 \sigma_i \x^i = \t\}.
 \end{equation}
 We say that a function $\pi$ is \emph{additive over $(I_1, I_2, I_3, \ve \sigma, \t)$}
 provided that 
 $$
\sum_{i=1}^3 \sigma_i \pi(\x^i) = 0 \text{ for all } (\x^1, \x^2, \x^3) \in F(I_1, I_2, I_3, \ve \sigma, \t).
 $$
 To simplify notation we often write $\tau = (I_1, I_2, I_3, \ve \sigma, \t)$ and discuss $\pi$ as additive over $\tau$.

For a set $F \subseteq \R^k \times \R^k \times \R^k$, define the projections 
\begin{equation}
\label{eq:projections}
    p_i \colon \R^k \times \R^k \times \R^k \to \R^k \ \text{ by } \ \x^1, \x^2, \x^3) \mapsto \x^i \ \text{ for } i=1,2,3.
\end{equation}
For a set of 7-tuples $\E$, we define the projections of faces defined by these 7-tuples as 
\begin{equation}
\label{eq:projE}
p(\E) :=\bigcup\{p_i(F(\tau)): i = 1,2,3, \tau\in \E\}.
\end{equation}

Next, we define the perturbation space.  We describe a general notation for later use.  In particular, let $\P$ be a collection of subsets of $\R^2$, $\E$ be a set of 7-tuples and $\mathcal Z, p(\E) \subseteq \P$.  

Furthermore, suppose that $\P$ is $\Z^2$-periodic.  That is, $\P = \P + \Z^2$.  We then define the perturbation space with respect to $\mathcal E, \mathcal Z, \P$ as
 \begin{align}
 \label{eq:bPimathcalE}
\begin{aligned}
   \bar \Pi^{\mathcal{E}}_{\mathcal{Z}, \P} :=& \{\bpi\colon \textstyle\bigcup\P \to \R && \st&&
  \bpi|_{\bigcup \mathcal{Z}} \equiv 0,\, \bpi(\0) = 0, \,
 \bpi(\f) = 0, && \text{(Zeros)}\\
 &&&&& \bpi(\x) = \bpi(\x + \t) \text{ for all } \x\in\textstyle\bigcup\P \text{ and }  \t\in\Z^2, && \text{(Periodic)}\\
 &&&&& \bpi \text{ is additive over $\tau$ for all } \tau \in \mathcal{E} && \text{(Additive)} \}.
\end{aligned}
\end{align}
This flexible notation allows us to consider different types of additivities, different subsets of domains, and record values where perturbation functions must vanish.

As a shorthand, we define
\begin{equation}
\label{eq:bPishorthand}
    \bar \Pi^{\mathcal E} := \bar \Pi^{\mathcal E}_{\emptyset, \R^2}
\end{equation}
Of particular interest, and as a starting point in our analysis, is the set $\bar \Pi^{\mathcal E(\pi)}$: this is the space of perturbation functions for the function $\pi$ (in other works, this has been denoted $\bar \Pi^{E(\pi)}$).

 Similar to the face decomposition of $E(\pi)$, we consider polyhedral decompositions of the set of $\mathcal E(\pi)$. We introduce the definition of valid 7-tuple to define the faces of $\mathcal E(\pi)$.

\begin{definition}[Valid 7-tuple]
$( I_1, I_2,  I_3, \ve \sigma,  \t)$ is a \emph{valid 7-tuple} if the projections of $F=F( I_1, I_2,  I_3, \ve \sigma,  \t)$ are the sets $I_i$, i.e., $p_i(F) = I_i$ for $i=1,2,3$.
\end{definition}

Studying valid 7-tuples is important since, for example,  we might have a case where $F( I_1, I_2,  I_3, \ve \sigma,  \t) = \emptyset$, even though $I_1, I_2, I_3 \neq \emptyset$.  Thus, we show that using only valid 7-tuples, we appropriately capture the set of additivity of interest.

We conclude this subsection with several results about the structure of valid 7-tuples.

 \begin{remark}
 \label{lemTripleSwap}
 Let $\mathrm{perm}\colon\{1,2,3\}\to \{1,2,3\}$ be a permutation, $( I_1, I_2,  I_3, \ve \sigma,  \t)$ is a valid 7-tuple, and $\pi$ is a function.  Then the following are equivalent:
 \begin{enumerate}
 \item $\pi$ is additive over $( I_1, I_2,  I_3, \ve \sigma,  \t)$,
 \item $\pi$ is additive over $( I_{\mathrm{perm}(1)}, I_{\mathrm{perm}(2)},  I_{\mathrm{perm}(3)}, \mathrm{perm}(\ve \sigma),  \t)$, where $\mathrm{perm}(\ve \sigma)=(\sigma_{\mathrm{perm}(1)}, \sigma_{\mathrm{perm}(2)}, \sigma_{\mathrm{perm}(3)})$,
 \item $\pi$ is additive over $(I_1, I_2,  I_3, -\ve \sigma,  -\t)$.
 \end{enumerate}
 \end{remark}
 \begin{proof}
 This is obvious because of the symmetry in the notation.
 \end{proof}

\begin{lemma}[Projections of valid 7-tuples]
\label{lem:projections}
For any set $F =F( I_1, I_2,  I_3, \ve \sigma,  \t)$, we have the following formulas for the projections:
\begin{enumerate}
\item $p_1(F) = I_1 \cap [\sigma_1 \t - (\sigma_1 \sigma_2 I_2 + \sigma_1 \sigma_3 I_3)]$,
\item $p_2(F) = I_2 \cap [\sigma_2 \t - (\sigma_2 \sigma_1 I_1 + \sigma_2 \sigma_3 I_3)]$,
\item $p_3(F) = I_3 \cap [\sigma_3 \t - (\sigma_3 \sigma_2 I_2 + \sigma_3 \sigma_1 I_1)]$.
\end{enumerate}
\end{lemma}
\begin{proof}
Since all these cases are symmetric, we will just prove the first one.  By definition, we have
\begin{align*}
p_1(F)
&= \{ \x^1  \st \x^i \in I_i \text{ for } i=1,2,3, \sum_{i=1}^3 \sigma_i \x^i = \t\}\\
&= I_1 \cap \{\x^1 = \sigma_1 \t - (\sigma_1 \sigma_2 \x^2 + \sigma_1 \sigma_3 \x^3)\st \x^2 \in I_2, \x^3 \in I_3\}\\
  & = I_1 \cap [\sigma_1 \t - (\sigma_1 \sigma_2 I_2 + \sigma_1 \sigma_3 I_3)].
\end{align*}
\end{proof}

\begin{corollary}
\label{validCharacterization}
$( I_1, I_2,  I_3, \ve \sigma,  \t)$ with $\ve\sigma = (\sigma_1, \sigma_2, \sigma_3)$ is a valid 7-tuple if and only if the following hold
\begin{enumerate}
\item $ I_1 \subseteq \sigma_1 \t - (\sigma_1 \sigma_2 I_2 + \sigma_1 \sigma_3 I_3)$, \label{item:first_containment}
\item $ I_2 \subseteq \sigma_2 \t - (\sigma_2 \sigma_1 I_1 + \sigma_2 \sigma_3 I_3)$,
\item $ I_3 \subseteq \sigma_3 \t - (\sigma_3 \sigma_2 I_2 + \sigma_3 \sigma_1 I_1)$.
\end{enumerate}
\end{corollary}
\begin{proof}
$(I_1, I_2, I_3, \ve\sigma, \t)$ is a valid 7-tuple, then
\begin{align*}
I_1
= p_1(F( I_1, I_2,  I_3, \ve \sigma,  \t))
= I_1 \cap [\sigma_1 \t - (\sigma_1 \sigma_2 I_2 + \sigma_1 \sigma_3 I_3)]
  \subseteq   [\sigma_1 \t - (\sigma_1 \sigma_2 I_2 + \sigma_1 \sigma_3 I_3)].
\end{align*}
On the other hand, if $ I_1 \subseteq [\sigma_1 \t - (\sigma_1 \sigma_2 I_2 + \sigma_1 \sigma_3 I_3)]$, then
\begin{align*}
p_1(F( I_1, I_2,  I_3, \ve \sigma,  \t))   =
I_1 \cap [\sigma_1 \t - (\sigma_1 \sigma_2 I_2 + \sigma_1 \sigma_3 I_3)]
=I_1
\end{align*}
Therefore, applying this logic to $I_2$ and $I_3$ as well finishes the result.
\end{proof}

\begin{corollary}
\label{validPointCharacterization}
The 7-tuple $( I_1, \{\v\},  I_3, \ve \sigma,  \t)$ with $\ve\sigma = (\sigma_1, \sigma_2, \sigma_3)$ is a valid 7-tuple if and only if
$$ I_1 = \sigma_1 \t - (\sigma_1 \sigma_2 \{\v\} + \sigma_1 \sigma_3 I_3).$$
\end{corollary}
\begin{proof}
By \autoref{validCharacterization},
$( I_1, \{\v\},  I_3, \ve \sigma,  \t)$ is a valid 7-tuple if and only if
\begin{enumerate}
\item $ I_1 \subseteq \sigma_1 \t - (\sigma_1 \sigma_2 \{\v\} + \sigma_1 \sigma_3 I_3)$
\item $ \{\v\} \subseteq \sigma_2 \t - (\sigma_2 \sigma_1 I_1 + \sigma_2 \sigma_3 I_3)$
\item $ I_3 \subseteq \sigma_3 \t - (\sigma_3 \sigma_2 \{\v\} + \sigma_3 \sigma_1 I_1)$
\end{enumerate}
Rewriting (1) and (3) to have $I_1$ on the left hand side yields
\begin{enumerate}
\item $ I_1 \subseteq \sigma_1 \t - (\sigma_1 \sigma_2 \{\v\} + \sigma_1 \sigma_3 I_3)$
\setcounter{enumi}{2}
\item $ I_1 \supseteq \sigma_1 \t - (\sigma_1 \sigma_2 \{\v\} + \sigma_1 \sigma_3 I_3)$
\end{enumerate}
Thus, because of double containment, if $( I_1, \{\v\},  I_3, \ve \sigma,  \t)$ is a valid 7-tuple, then $I_1 = \sigma_1 \t - (\sigma_1 \sigma_2 \{\v\} + \sigma_1 \sigma_3 I_3)$.

If $I_1 = \sigma_1 \t - (\sigma_1 \sigma_2 \{\v\} + \sigma_1 \sigma_3 I_3)$, (1), (3) also hold. Moreover, there exist $\x^1 \in I_1$ and $\x^3\in I_3$ such that $\x^1 = \sigma_1 \t  - \sigma_1 \sigma_2 \v - \sigma_1 \sigma_3 \x^3$, which implies $\v = \sigma_2 t - (\sigma_2\sigma_1\x^1 + \sigma_2 \sigma_3 \x^3)\in \sigma_2 \t - (\sigma_2 \sigma_1 I_1 + \sigma_2 \sigma_3 I_3)$. Thus, $( I_1, \{\v\},  I_3, \ve \sigma,  \t)$ is a valid 7-tuple.
\end{proof}

\subsubsection{Combinatorializing the additivity domain}
\label{section:additivity-discretized}

Let $\pi \colon \R^k \to \R$ be a continuous piecewise linear function over a  pure, complete polyhedral complex $\P$.  

We now give a combinatorial representation of this set
using the faces of~$\P$. Let 
\begin{equation}
\label{eq:EpiP}
\mathcal E (\pi, \P) := \{ \tau : \tau = (I_1, I_2, I_3, (1,1,-1), \0), I_i \in \P, i=1,2,3, \text{ and } \pi \text{ additive over } \tau\}.
\end{equation}
  Extremality is more easily studied if we restrict ourselves the additivities $E(\pi,\P)$. 
  
If $\pi$ is continuous, then $(1,1,-1)\cdot \pi$ is continuous.
Under this continuity assumption, we could further refine the set of additivities by only considering a set of maximal faces in $E(\pi, \P)$.   But this is not necessary for the results in this paper.

\begin{lemma}[Adapted from {\cite[Lemma 3.12]{bhk-IPCOext}}]\label{lemma:covered-by-maximal-valid-triples}
\begin{equation*}
\bigcup \mathcal E(\pi) =  \bigcup \mathcal E(\pi,\P)
\end{equation*}
\end{lemma}
This combinatorial
representation can then be made finite by choosing representatives (recall that if $\pi$ is minimal, then it is periodic~\autoref{thm:minimal}). 

\subsubsection{Perturbation functions}\label{subsec:perturbations}

We now discuss how to prove that a given minimal function is not a facet or not extreme. We consider the \emph{space of perturbation functions} $\bar \Pi^{\mathcal E(\pi)}$with prescribed additivities. 
 Clearly this is a linear space.

If $\pi$ is not a facet of $R_\f(\R^k, \Z^k)$, then  by the Facet Theorem, \autoref{thm:facet},  there exists a nontrivial $\bar \pi \in \bar \Pi^{\E(\pi)}$
such that $\pi' = \pi + \bar \pi$ is a minimal valid function.   Note that this last statement is not an if and only if statement.

Suppose $\pi$ is piecewise linear on a polyhedral complex $\P$.  We will often consider a refinement $\P'$ of $\P$ on which we can find a continuous piecewise linear perturbation $\bar \pi$ such that $\pi$ is not extreme.

The basic idea is that if one can find a non-zero function $\bar\pi$ in the linear
subspace of functions $\bar \Pi^{\mathcal E(\pi)}$ then the finite, combinatorial
description of $(1,1,-1)\cdot\pi$ (since $\pi$ and therefore $(1,1,-1)\cdot \pi$ is
piecewise linear) %
allows small perturbations from $\pi$ in the direction of $\bar\pi$ while maintaining minimality.

\begin{theorem}[Perturbation {\cite[Theorem 3.13]{bhk-IPCOext}}]
\label{corPerturb}
Let $\P$ be a pure, complete, polyhedral complex  in $\R^k$ that is periodic modulo $\Z^k$ and every cell of $\P$ is bounded.
Suppose $\pi$ is minimal and continuous piecewise linear over $\P$.  Suppose
$\bar \pi \not\equiv 0$ is continuous piecewise linear over a refinement $\P'$ of $\P$, is periodic
modulo $\Z^k$ and satisfies $\bar \pi \in \bar \Pi^{\mathcal E(\pi)}$.
Then $\pi$ is not extreme. Furthermore, given $\bar\pi$, there
exists an $\epsilon > 0$ such that $\pi^1 = \pi + \epsilon \bar \pi$ and
$\pi^2 = \pi - \epsilon \bar \pi$ are distinct minimal functions that are
continuous piecewise linear over $\P$ such that $\pi = \tfrac12(\pi^1 +
\pi^2)$. 
\end{theorem}

\subsubsection{Two-row case using a standard triangulation of~$\R^2$}
\label{subsec:standard-tri-R2}

For the case of the standard triangulations~$\P_q$ of~$\R^2$
(\autoref{ex:2d-standard-triangulation}),
\cite{basu-hildebrand-koeppe:equivariant-2,bhk-IPCOext} 
describe an algorithm of the above scheme for a special class of piecewise
linear functions over this complex, which are said to be \emph{diagonally
  constrained}.

Let 
\begin{equation}
\label{eq:A-matrix}
A =  \begin{bmatrix}  1 & -1 &  0 & 0 &  1 & -1 \\
 0 & 0 & 1 & -1  & 1 & -1  \end{bmatrix}^\top.
\end{equation}
Then for every face $I \in \P_q$, there exists a vector $\ve b \in \frac1q\Z^6$ such that $I = \{\, \x \st A \x \leq \ve b\,\}.$ Furthermore, for every vector $\ve b \in \frac1q\Z^6$, the set $\{\, \x \st A \x \leq \ve b\,\}$ is a union of faces of $\P_q$ (possibly empty), since each inequality corresponds to a hyperplane in the arrangement $\mathcal{H}_q$.
A key element in the following results is that the matrix $A$ is totally unimodular.  In fact, it posses an even stronger condition than total unimodularity.

\begin{example}[{\cite[Example 7.5]{igp_survey, igp_survey_part_2}}]\label{ex:diag-constrained-function} \label{ex:diag-constrained-function-continued2}
  \autoref{figure:diagonallyConstrained-new-figure} shows the complex~$\P_5$
  with an example of a minimal valid continuous piecewise linear function
  on~$\P_5$ with $\f = \ColVec[5]{2}{2}$ that is periodic modulo~$\Z^2$. Note that, due the
  periodicity of the function modulo~$\Z^2$, the values of the function on the
  left and the right edge (and likewise on the bottom and the top edge) of $D = [0,1]^2$ match. 

  \end{example} 

\begin{figure}[H]

\begin{center}
\includegraphics[scale=0.47]{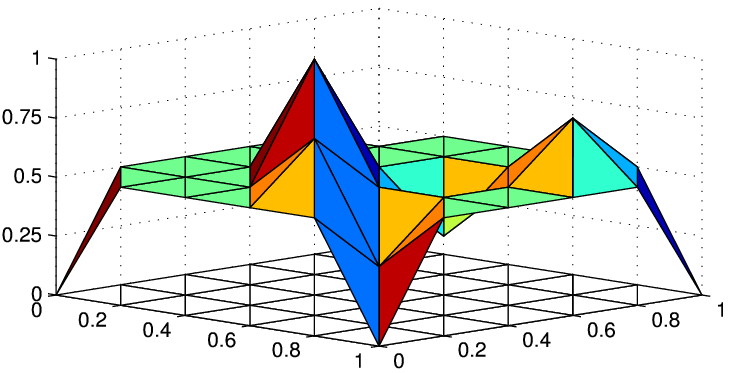} \ \ \ \ 
\begin{tikzpicture}[scale=0.7]
\begin{scope}[xshift = 2.5cm, yshift=2.5cm]
    \node {\includegraphics[scale=.18]{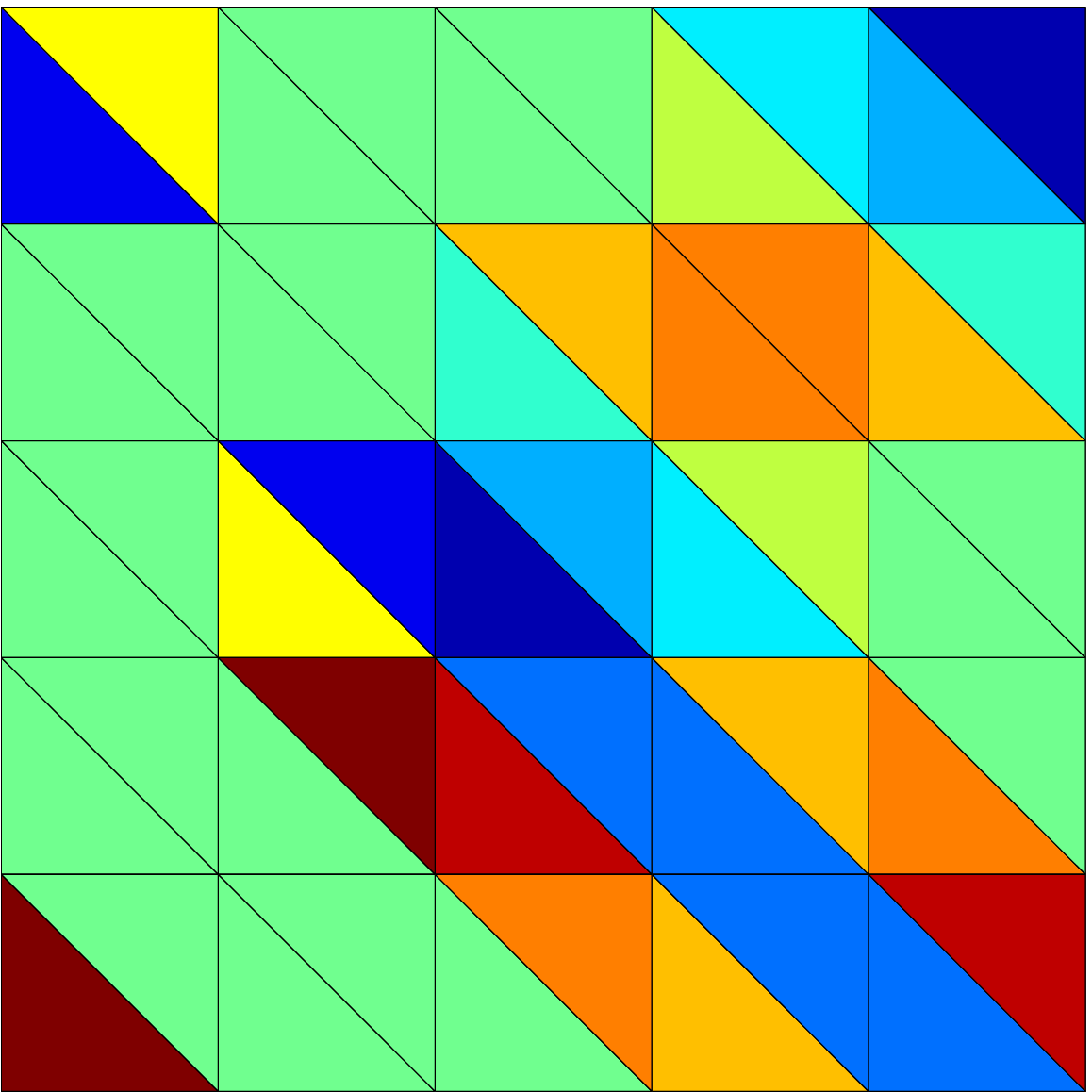}};
    \end{scope}
 
\begin{scope}[shift={(0,0)}]
\def\functionValues{{
      { 0,     2,     2,     2,     2,     0},
    { 2,     2,     2,     3,     1,     2},
    { 2,     2,     4,     2,     2,     2},
    { 2 ,    2,     2,     1,     2,     2},
    { 2,     2,     2,     2,     3,     2},
    { 0,     2,     2,     2,     2,     0}}}

\foreach \i in {0,...,5} {%
	\foreach \j in {0,...,5} {%
	\node[draw, fill = white, scale = 0.65] at (\j,\i) {\pgfmathparse{\functionValues[\i][\j]}\pgfmathresult};
	}
	}
	
	\end{scope}
\end{tikzpicture}
\end{center}
\caption{A minimal valid, continuous, piecewise linear function over the
  polyhedral complex $\P_5$.
  \emph{Left}, the three-dimensional plot of the function on $D = [0,1]^2$. \emph{Right},
  the complex $\P_5$, restricted to $D$ and colored according to slopes to match the 3-dimensional plot, and decorated
  with values $v$ at each vertex of $\P_5$ where the function takes value
  $\tfrac{v}{4}$. }

\label{figure:diagonallyConstrained-new-figure}
\end{figure}
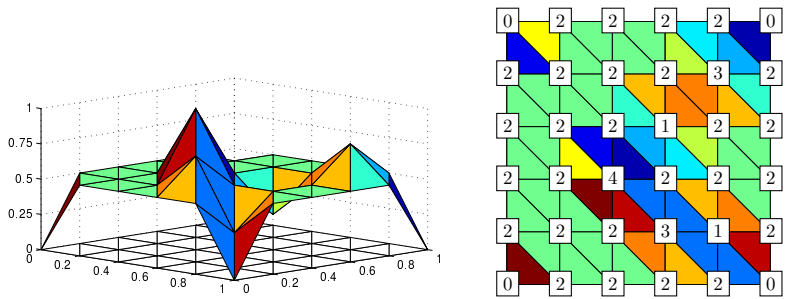

    \subsection{The Interval Lemma and its $k$-dimensional generalizations}\label{sec:foundations}

To prove that a given minimal valid function $\pi$ is a facet (or an extreme function), we use the additivity domain $E(\pi)$ of a subadditive function $\pi\colon\R^k \to\R$. The goal is to show that $\pi$ and other minimal valid functions are affine when restricted to projections of $E(\pi)$. This is achieved by decomposing the additivity domain into convex sets, which can then be analyzed independently.

One of the key tools for this decomposition is based on the \textit{Interval Lemma}, introduced by Gomory and Johnson \cite{tspace, infinite2}. The Interval Lemma allows us to deduce that if a function behaves additively on certain bounded convex sets, it must exhibit affine properties on those sets. In other words, the lemma provides a way to reduce an infinite set of functional equations into a finite-dimensional problem where the function behaves linearly over specific intervals. This dimension-reduction technique simplifies the analysis and plays a crucial role in proving the extremality of minimal valid functions.
The Interval Lemma was later generalized to higher dimensions to apply to convex sets in $\R^k$. This generalization is essential for understanding additive properties across more complex domains, and has been utilized in works such as \cite{bhkm} to prove the extremality of minimal valid functions in higher-dimensional settings.

For the purpose of this paper dealing with valid 7-tuples, we will present the high-dimensional settings where a convex set in $\R^k\times \R^k$ is given for defining the additivity. %

For these statements, we need projections $\R^4 \to \R^2$ (in particular, to conveniently reuse theorems from the literature). Thus, 
define the projections $\p_i \colon \R^k \times \R^k \times \R^k \to \R^k$ by $(\x^1, \x^2) \mapsto \x^i$ for $i=1,2$ and $\p_3$ defined by $(\x^1, \x^2) \mapsto \x^1 + \x^2$.

\begin{theorem}[Convex additivity domain lemma, full-dimensional
  version {\cite[Theorem 1.7]{bhk-IPCOext}}]\label{lem:projection_interval_lemma_fulldim} 
  Let $f,g,h \colon \R^k \to \R$ be bounded functions. 
  Let $F \subseteq \R^k \times \R^k$ be a full-dimensional convex set
  such that $f(\u) + g(\v) = h(\u+\v)$ for all $(\u, \v) \in F$. 
  Then there exists a vector $\cve\in \R^k$ such that $f, g$ and $h$ are
  affine with the same gradient $\cve$ over $\intr(\p_1(F))$,
  $\intr(\p_2(F))$ and $\intr(\p_3(F))$, respectively. 
\end{theorem}

\autoref{lem:projection_interval_lemma_fulldim} can be established in a significantly more general setting, which
takes care of  situations in which the set $F$ is not full-dimensional
(\autoref{lem:projection_interval_lemma}). Affine properties are deduced with respect to certain subspaces, which is important for the classification of extreme functions in two or more dimensions.

\begin{definition}
Let $U \subseteq \R^k$. 
Given a linear subspace $L \subseteq \R^k$, we say $\pi\colon U \to \R$ is
{\em affine with respect to $L$ over $U$} if there exists $\cve \in \R^k$ such
that $\pi(\u^2) - \pi(\u^1) =  \cve\cdot( \u^2 - \u^1)$ for any $\u^1, \u^2
\in U$ such that $\u^2 - \u^1 \in L$.

\end{definition}

\begin{definition}
  For a linear space $L \subseteq \R^k$ and a set $U \subseteq \R^k$ such that
  for some $\u \in \R^k$ we have $\aff(U) \subseteq L + \u$, we will denote by
  $\intr_L(U)$ the interior of $U$ in the relative topology of $L + \u$.
\end{definition}
Note that $\intr_L(U)$ is well defined because either $\aff(U) = L + \u$, or
$\intr_L(U) = \emptyset$.   
We now state our most general Interval Lemma theorem on a convex domain.

\begin{theorem}[Convex additivity domain lemma; {\cite[Theorem 2.11]{bhk-IPCOext}}]
\label{lem:projection_interval_lemma}
Let $f,g,h \colon \R^k \to \R$ be bounded functions. Let $F \subseteq \R^k \times \R^k$ be a convex set
such that $f(\u) + g(\v) = h(\u+\v)$ for all $(\u, \v) \in F$. 
Let $L$ be a linear subspace of $\R^k$ such that $(L \times L) + F \subseteq \aff(F)$.
Let $(\u^0, \v^0) \in \relint(F)$. 
Then there exists a vector $\cve\in \R^k$ such that $f, g$ and $h$ are affine with gradient $\cve$ over $\intr_L((\u^0 + L) \cap \p_1(F))$, $\intr_L((\v^0 + L) \cap \p_2(F))$ and $\intr_L((\u^0 + \v^0 + L) \cap \p_3(F))$, respectively.
\end{theorem}

Theorem~\ref{lem:projection_interval_lemma_fulldim} follows when $L=\R^k$.

Note that under the continuous functions assumption, the affine properties can extend to the boundary in higher dimensions.
Without this assumption, this is not possible in general (see \cite[Remark~2.12]{bhk-IPCOext}). 

\begin{corollary}[Convex additivity domain lemma for continuous functions; {\cite[Corollary 2.14]{bhk-IPCOext}})]
\label{lem:projection_interval_lemma-corollary}
Let $f,g,h \colon \R^k \to \R$ be continuous functions. Let $F \subseteq \R^k \times \R^k$ be a convex set
such that $f(\u) + g(\v) = h(\u+\v)$ for all $(\u, \v) \in F$. 
Let $L$ be a linear subspace of $\R^k$ such that $L \times L + F \subseteq \aff(F)$.
Let $(\u^0, \v^0) \in \relint(F)$. 
Then there exists a vector $\cve\in \R^k$ such that $f, g$ and $h$ are affine with gradient $\cve$ over $(\u^0 + L) \cap \p_1(F)$, $(\v^0 + L) \cap \p_2(F)$ and $(\u^0 + \v^0 + L) \cap \p_3(F)$, respectively.
\end{corollary}

\subsection{Algorithm using restriction to finite group problems}
\label{sec:alg-restriction-finite-groups}
In this section, we discuss connections between infinite group problems and finite group problems.  We begin with a discussion of testing extremality for finite group problems.  Later we show that in certain settings, a function is extreme for an infinite group problem if and only if its restriction to a finite group is extreme for the finite group problem.  Hence, this connection provides an alternative algorithm from those settings for testing extremality and facetness.
\subsubsection{Algorithm for finite group problems}\label{s:finite-group}
When $S$ has finite index in $G$, we call $R_\f(G,S)$ a \emph{finite group problem}.
Note that there is a correspondence between $R_\f(G,S)$ and $R_\f(G/S,0)$ by aggregation of variables, and it is convenient to study
the finite-dimensional problem $R_{\bar\f}(G/S,0)$.  
The fundamental theorem of finitely generated abelian groups shows that $G/S
\cong (\tfrac{1}{q_1} \Z \times \dots \times \tfrac{1}{q_k} \Z)/\Z^k$ for some
$q_i \in \N$ for $i=1, \dots, k$.  Therefore, it suffices to consider
$G = \tfrac{1}{q_1} \Z \times \dots \times \tfrac{1}{q_k} \Z$ and $S = \Z^k$
where $q_i \in \N$.  
In the case of one row, $G/S = \tfrac{1}{q_1} \Z/\Z \cong \Z/q_1\Z$ is a cyclic
group.   Cyclic group problems were originally studied by
Gomory~\cite{gom} and have been the subject of many later studies.
See~\cite{Richard-Dey-2010:50-year-survey} for an excellent survey on these results.

 The set of minimal valid functions $\pi \colon G/S \to \R$ %
 is a (finite-dimensional) convex polytope~\cite{gom}.  Extreme functions are
 thus extreme points of this polytope.  As we noted in
 \autoref{subsec:relation-between-three-notions}, 
 standard polyhedral theory reveals that extreme functions are
 equivalent to weak facets and facets.  %
 Furthermore, extreme points of
 polytopes are characterized by points where the tight inequalities are of
 full rank.  Therefore, testing extremality of a function for a finite group
 problem can be done with simple linear algebra.  

 Note that there is a bijection
 between the minimal valid functions of $R_\f(G,S)$ and minimal valid functions
 for $R_{\bar \f}(G/S,0)$.  This is because 
 minimal valid functions for $R_\f(G,S)$ are $S$-periodic functions by
 \autoref{thm:minimal}.  
 Hence the extremality test 
 translates into the following statement about the space of perturbation functions.
\begin{theorem}
[Rephrased from {\cite[Theorem 8.1]{igp_survey, igp_survey_part_2}}]%
\label{thm:finite-extremality}
 Let $G = \tfrac{1}{q_1} \Z \times \dots \times \tfrac{1}{q_k} \Z$ and let $\f \in G$.  Let $\pi \colon G \to \R$ be a minimal valid function for $R_\f(G, \Z^k)$.  Then $\pi$ is extreme if and only if $\bar \Pi^{\E(\pi)}_{\emptyset, G}=  \{0\}$.
\end{theorem}

\subsubsection{Restriction and interpolation in the one-row problem}\label{s:interpolation}
Gomory and Johnson devised the infinite group problem as a way to study the finite group problem.  
They studied interpolations of valid functions of the finite group problems $R_f(\tfrac{1}{q} \Z, \Z)$ in order to connect the problems, but they never completed this program.  Due to the ease of testing extremality in the finite group problems, having this connection is useful for algorithms.

\begin{theorem}
[{\cite[Theorem 8.6]{igp_survey, igp_survey_part_2}}]
\label{thm:extreme-facet-and-restriction}
\label{thm:extreme-restriction-m}
  Let $m \in \Z_{\geq 3}$.  Let $\pi$ be a continuous piecewise linear  minimal
  valid function for $R_f(\R,\Z)$ with breakpoints in $\tfrac{1}{q} \Z$ and suppose $f \in \tfrac{1}{q} \Z$.    The following are equivalent:
\begin{enumerate}
\item $\pi$ is a facet for $R_f(\R,\Z)$,
\item $\pi$ is extreme for $R_f(\R,\Z)$,
\item $\pi|_{\frac{1}{mq}\Z}$ is extreme for $R_f(\tfrac{1}{mq} \Z, \Z)$.
\end{enumerate}
\end{theorem}
\noindent The proof can  be found in \cite[Theorem 8.6]{igp_survey, igp_survey_part_2} and serves as a nice warm-up to the proof of our main result in this paper.

\subsubsection{Restriction and interpolation for $k\geq 2$}

Some similar restriction results can be proved for the case of $k$ rows, but
this area is much more open.  Restrictions seem to require the use of nice
polyhedral complexes.  The only results known are for the polyhedral complex
$\P_q$ (Example~\ref{ex:2d-standard-triangulation}) in $\R^2$. 
\begin{theorem}[{\cite[Theorem 4.5
and Theorem 5.16
]{bhk-IPCOext}}]
\label{thm:GJ-restrictions-2D}
Let $\pi\colon \R^2 \to \R$ be a continuous piecewise linear function over $\P_q$ and suppose $\f \in \tfrac{1}{q} \Z^2$.  Then the following hold:
\begin{enumerate}
\item $\pi$ is minimal for $R_\f(\R^2,\Z^2)$ if and only if $\pi_{\frac{1}{q}\Z^2}$ is minimal for $R_\f(\tfrac{1}{q} \Z^2, \Z^2)$.
\item If $\pi$ is extreme for $R_\f(\R^2,\Z^2)$, then $\pi|_{\frac{1}{q} \Z^2}$ is extreme for $R_\f(\tfrac{1}{q} \Z^2,\Z^2)$.
\end{enumerate}
\end{theorem}
For $k\geq3$ rows, it is unclear when similar results are possible.

In the special case of diagonally constrained functions in $\R^2$, there is a similar result to \autoref{thm:extreme-restriction-m}.

\begin{theorem}[{\cite[Theorem 1.9]{bhk-IPCOext}}]
\label{thm:1/4q}
Let $\pi$ be a minimal continuous piecewise linear function over $\P_q$ that is diagonally constrained and $\f \in \verts(\P_q)$.  Fix $m \in \Z_{\geq 3}$.
Then $\pi$ is extreme for $R_{\ve f}(\R^2, \Z^2)$ if and only if the restriction $\pi\big|_{\frac{1}{mq}\Z^2}$ is extreme for $R_{\ve f}(\frac{1}{mq} \Z^2, \Z^2)$.  
\end{theorem}

Our main result implies that this theorem holds without the added restriction when we remove the assumption of diagonally constrained.

\section{Generalized Additive Functions and Valid 7-tuples}
\label{sec:7-tuples}
We begin with a table of notation, some defined in the previous section, and other notation to come.  

\begin{table}[H]
\begin{tabular}{llr}
    Symbol & Description & Reference\\
    \hline
         $R_\f(G,S)$ &Group problem with group $G$ and subgroup $S$ & \eqref{GP}\\
     $R_\f(\R^2,\Z^2)$ & 2-Row Infinite Group Problem& \\
     $R_\f(\tfrac{1}{mq} \Z^2, \Z^2)$ & 2-Row group problem on grid& \\
    $E(\pi)$ &  Set of additive pairs $(\x^1, \x^2)$ &\eqref{eq:Epi}\\
    $\mathcal E(\pi)$ & Set of  7-tuples $(\x^1, \x^2, \x^1 + \x^2, (1,1,-1), \0)$ &\eqref{eq:mathcalEpi}\\
    $\mathcal E(\pi, \P)$& Set of additive 7-tuples from faces of $\P$ &\eqref{eq:EpiP}\\
    $p_i(F)$ & Projection on to the $i$th tuple of variables from $F$ & \eqref{eq:projections}\\
    $p(\E)$ & Collection of projections of the 7-tuples from $\E$ & \eqref{eq:projE}\\
            $\mathcal Z$ & Set of faces of $\P$ where we set perturbation to be 0 & \eqref{eq:bPimathcalE}\\
    $\bar \Pi^{\mathcal{E}}_{\mathcal{Z}, \mathcal{F}}$ & Perturbation space with domain $\mathcal F$ satisfying 7-tuples $\mathcal E$, zeros $\mathcal Z$ & \eqref{eq:bPimathcalE}\\
    $\bar \Pi^{\mathcal E}$ & Shorthand for  $\bar \Pi^{\mathcal E}_{\emptyset, \R^2}$ & \eqref{eq:bPishorthand}\\
     $\otimes$ & Refinement operator for polyhedral complexes & \eqref{eq:refinement}
\end{tabular}
\caption{Notation: Informal descriptions of notation with links to formal definitions.}
\label{tab:notation}
\end{table}

We continue with some more definitions.  
We say polyhedral complex $\P$ is $\Z^2$-periodic if $\P = \P + \Z^2$.

For two polyhedral complexes $\P, \mathcal Q$ we define the (symmetric) refinement of $\P$ and $\mathcal Q$ as 
\begin{equation}
\label{eq:refinement}
\P \otimes \mathcal Q := \{ P \cap Q : P \in \mathcal P, Q \in \mathcal Q, P \cap Q \neq \emptyset\}.
\end{equation}

\subsection{Special Properties for $\P_q$}
Before continuing, we explain our unique notation.  Within our complex $\P_q$, there are points, three types of edges, and two types of triangles.  When discussing faces of this complex, we will add extra notation to allow the reader to remember which type of face it is.  To do so, we optionally provide the reminder in pictorial format in front of the variable representing the face.  For example, if $I \in \Itri$ is a triangle, we may refer to it as $\tri I$.  
In this notation, we may for example express 7-tuples as 
$\Tuple(\diag I_1, \ver I_2, \tri I_3, \ve\sigma, \t)$.
\begin{lemma}[Valid 7-tuples (adaptation of Lemma 4.9, \cite{bhk-IPCOext})]\label{lemma:cases}
Suppose that $I_1, I_2, I_3 \in \P_q$ and $\Tuple( I_1, I_2,  I_3, \ve \sigma,  \t)$ is a valid 7-tuple.  Up to permutation of $I_1, I_2, I_3$ via \autoref{lemTripleSwap}, one of the following is true.
\begin{enumerate}
\item $I_1,I_2,I_3 \in \Ipoint$,
\item $I_1\in \Ipoint$ and $I_2, I_3 \in \P_{q,*}$ for $* = \ver, \hor, \diag$, \label{item:point-edge}
\item $I_1 \in \Ivert$, $I_2 \in \Ihor$, $I_3 \in \Idiag$, \label{item:three-diags}
\item $I_1\in \Ipoint$ and $I_2, I_3 \in \P_{q,\tri}$,
\item $I_1 \in \P_{q,*}$ for $* = \ver, \hor$ or $\diag$ and $I_2, I_3 \in \Itri$,
\item $I_1 \in \P_{q, *_1}$, $I_2 \in \P_{q,*_2}$ for $*_1 \neq *_2$, $*_1,*_2 \in \{\ver, \hor, \diag\}$, and $I_3 \in \Itri$,
\item $I_1,I_2,I_3 \in \I_{q,*}$ for $* = \ver, \hor, \diag$, or $\tri$
\end{enumerate}
\end{lemma}
\begin{remark}
The type (2) corresponds to the translation and reflection between domains.
\end{remark}

\begin{proof}
Fix $\ve\sigma \in \{\pm1\}^3$ and $\t \in \R^2$.
The sets $I_1, I_2, I_3$ could be in any of the 5 sets $\Ipoint, \I_{q,\ver},\P_{q,\hor},\P_{q,\diag},\Itri$.  Therefore there are $5^3 = 125$ possible $\Tuple(I_1, I_2, I_3, \ve\sigma, \t)$ .
Of these 125, there are 1 in case 1, $3\times 3 = 9$ in case 2, $3! = 6$ in case 3, $3$ in case 4, $3\times 3 =9$ in case 5, $3\times 3\times 2=18$ in case 6, $4$ in case 7.

This yields a total of $1 + 9 + 6 + 3 + 9 + 18 + 4 =  50$.  We will show that the other 75 possibilities are not valid 7-tuples.    For the sake of contradiction, suppose $\Tuple(I_1, I_2, I_3, \ve\sigma, \t)$ is a valid 7-tuple.  We proceed by cases and handle each case up to permutation of $I_1, I_2, I_3$.

\textbf{Case 1.} {\em Suppose $I_1,I_2 \in \Ipoint$, $I_3 \in \P_{q,*}$ for $* = \ver, \hor, \diag, \tri$.}

  Let $\u^i \in \verts(\P_q)$ such that $I_i = \{\u^i\}$ for $i=1,2$.
  By \autoref{validCharacterization},  $ I_3 \subseteq [\sigma_3 \t - (\sigma_3 \sigma_2 I_2 + \sigma_3 \sigma_1 I_1)] =  [\sigma_3 \t - (\sigma_3 \sigma_2 \{\u^2\} + \sigma_3 \sigma_1 \{\u^1\})]$ which is a singleton.  This is a contradiction since $I_3$ is supposed to be an edge or a triangle.
  There are 12 possibilities in this case.

\textbf{Case 2.} {\em Suppose $I_1, I_2 \in \P_{q,*_1}$ for $*_1 = \ver, \hor$ or $\diag$  and $I_3 \in \P_{q,*_2}$ for $*_2 = \ver, \hor, \diag$ or $\tri$, and $*_1 \neq *_2$. }

  Since $I_1, I_2$ are the same type of edges, there exists $\h \in \{(1,0), (0,1), (1,1)\}$, $c_1, c_2 \in \R$ such that $I_i \subseteq \{\x \st \h \cdot \x = c_i\}$ for $i=1,2$.
  By \autoref{validCharacterization},
  \begin{align*}
  I_3 &\subseteq [\sigma_3 \t - (\sigma_3 \sigma_2 I_2 + \sigma_3 \sigma_1 I_1)] \\
& = \{\x^3 \st \x^3 = \sigma_3 \t - (\sigma_3 \sigma_2 \x^2 + \sigma_3 \sigma_1 \x^1), \text{ with } \x^1 \in I_1, \x^2 \in I_2\}\\
  &\subseteq  \{\x^3 \st \h \cdot \x^3 = \h \cdot (\sigma_3 \t - (\sigma_3 \sigma_2 \x^2 + \sigma_3 \sigma_1 \x^1)), \text{ with } \x^1 \in I_1, \x^2 \in I_2\}\\
  &=  \{\x^3 \st \h \cdot \x^3 = \sigma_3 \h \cdot\t - (\sigma_3 \sigma_2 c_2 + \sigma_3 \sigma_1 c_1), \text{ with } \x^1 \in I_1, \x^2 \in I_2\}.
  \end{align*}
  Therefore $I_3$ is contained in a parallel hyperplane,  and hence $I_3$ must either be an edge of the same type as $I_1, I_2$ or a vertex.  This is a contradiction.
There are 27 possibilities in this case.

\textbf{Case 3.} {\em Suppose $I_1 \in \Ipoint$, $I_2 \in \I_{q,*_1}$ for $*_1 = \ver, \hor, \diag$ or $\tri$,  $I_3 \in \I_{q,*_2}$ for $*_2 = \ver, \hor, \diag$ or $\tri$, and $*_1 \neq *_2$. }

Let $I_1 = \{\u\}$.   By \autoref{validCharacterization},
\begin{enumerate}
\item $ I_2 \subseteq \sigma_2 \t - (\sigma_2 \sigma_1 \{\u\} + \sigma_2 \sigma_3 I_3) \Rightarrow \sigma_3 \t - (\sigma_3 \sigma_2 I_2 + \sigma_3 \sigma_1\{\u\}) \subseteq I_3$
\item $ I_3 \subseteq \sigma_3 \t - (\sigma_3 \sigma_2 I_2 + \sigma_3 \sigma_1\{\u\})$
\end{enumerate}
Therefore, $I_3 = \sigma_3 \t - (\sigma_3 \sigma_2 I_2 + \sigma_3 \sigma_1\{\u\})$, that is to say that $I_2$ and $I_3$ are translates and possibly reflections of each other.  Therefore, $I_2$ and $I_3$ must be the same type of faces.   This is a contradiction.  There are 36 possibilities in this case.

Note that $12 + 27 + 36 = 75$, and thus we have covered all the possibilities.
\end{proof}

Let $\Delta_n$ be the standard simplex in $n$-dimensions.
We can define the faces of $\P_q$ choosing some $\u \in \frac{1}{q}\Z^2$, $\d \in \{\pm\e^1, \pm\e^2, \pm(\e^1-\e^2)\}$ with $\e^1=(1,0)$ and $\e^2=(0,1)$, as 
\begin{align*}
& \point I = \{\u\}, \ \ 
\edge I = \{\x \in \R^2 : \x = \u + \frac{1}{q} \lambda \d, \lambda \in \Delta_1\}, \ \ \\
& \tri I= \{\x \in \R^2 : \x = \u \pm \frac{1}{q} \left(\lambda_1 \ve{e}^1 + \lambda_2 \ve e^2\right), \ve \lambda \in \Delta_2\}.
\end{align*}
Note that there are multiple equivalent representations, for example, $\edge I = \{\x \in \R^2 : \x = \u + \frac{1}{q} \lambda \d, \lambda \in \Delta_1\}=\{\x \in \R^2 : \x = (\u+\frac1q\d) - \frac{1}{q} \lambda \d, \lambda \in \Delta_1\}$.

We next give explicit characterizations of the $F(\tau)$.
\begin{lemma}
\label{lem:F-formulas}
Consider various 7-tuples $\tau = (I_1, I_2, I_3, \ve \sigma, \t)$. %
The 7-tuple $\tau$ is valid if and only if $F(\tau)$ can be parametrized by the following formulas. Furthermore, if $I_1, I_2, I_3, \sigma$ are given, there is a unique $\t$ to make $\tau$ a valid 7-tuple.

\noindent \textbf{One dimensional Faces:}
\begin{equation*}
\begin{array}{r@{\hskip -.5pt}rrlr@{\hskip -.25pt}l@{\hskip -.25pt}l@{\hskip -.5pt}rrr@{\hskip -.5pt}r}
F(&\ver\hor\diag I_1,& \point \{\u^2\},& \ver\hor\diag I_3,& \ve \sigma, \t) 
&=& \{ (\u^1, \u^2, \u^3) + \tfrac{1}{q}\lambda ( &-\sigma_1\sigma_3\d,& \0,& \d) &: \lambda \in \Delta_1\},
\end{array}
\end{equation*}
where $\u^1 = \sigma_1\t - \sigma_1\sigma_2\u^2 -\sigma_1\sigma_3\u^3$, and $\d \in \{\pm\e^1, \pm\e^2, \pm(\e^1-\e^2)\}$ with $\e^1=(1,0)$ and $\e^2=(0,1)$.
\begin{equation*}
\begin{array}{r@{\hskip -.5pt}rrlr@{\hskip -.25pt}l@{\hskip -.25pt}l@{\hskip -.5pt}rrr@{\hskip -.5pt}r}
F(&\ver I_1,&  \hor I_2,& \diag I_3,& \ve \sigma, \t)
&=& \{(\u^1, \u^2, \u^3) + \tfrac{1}{q}\lambda  (&-\sigma_1 \sigma_3 \e^1,& \sigma_2\sigma_3 \e^2,& \e^1-\e^2) &: \lambda \in \Delta_1\},
\end{array}
\end{equation*}
where $\u^1 = \sigma_1\t - \sigma_1\sigma_2\u^2 -\sigma_1\sigma_3\u^3$, $\e^1=(1,0)$ and $\e^2=(0,1)$.

\noindent \textbf{Two dimensional Faces:}
\begin{equation*}
\begin{array}{r@{\hskip -.5pt}rrlr@{\hskip -.25pt}l@{\hskip -.25pt}l@{\hskip -.5pt}rrr@{\hskip -.5pt}r}
F(&\tri I_1,& \point \{\u^2\},& \tri I_3,& \ve \sigma, \t) 
&=& \{ (\u^1, \u^2, \u^3) + \frac{1}{q}(&-\sigma_1\sigma_3(\lambda_1 \ve d^1+\lambda_2 \ve d^2),&\0,&\lambda_1 \ve d^1+\lambda_2 \ve d^2) 
 &: \ve \lambda \in \Delta_2\},
\end{array}
\end{equation*}
where $\u^1 = \sigma_1\t - \sigma_1\sigma_2\u^2 -\sigma_1\sigma_3\u^3$, and $\d^1,\d^2 \in \{\pm\e^1, \pm\e^2, \pm(\e^1-\e^2)\}$, $\d^1 \neq \d^2$.
\begin{equation*}
\begin{array}{r@{\hskip -.5pt}rrlr@{\hskip -.25pt}l@{\hskip -.25pt}l@{\hskip -.5pt}rrr@{\hskip -.5pt}r}
F(&\hor I_1,& \ver I_2,& \tri I_3,& \ve \sigma, \t) 
&=& \{ (\sigma_1\u^1, \sigma_2\u^2, -\sigma_3\u^3) + \frac{1}{q}(&\lambda_1 \sigma_1\ve d^1,&\sigma_2\lambda_2 \ve d^2,&-\sigma_3(\lambda_1 \ve d^1+\lambda_2 \ve d^2)) 
 &: \ve \lambda \in \Delta_2\},
\end{array}
\end{equation*}
where $\u^3 = -\t + \u^1 +\u^2$, $\d^1=\sigma\e^1, \d^2=\sigma\e^2, \sigma\in\{\pm1\}$.
\end{lemma}
We save the proof for the appendix.

\begin{lemma}
\label{lem:F-formula-one-tri}
Suppose that $K_1, K_2$ are edges of $K$ and $\ve\mu=(1,1,-1)$.
The 7-tuple $\tau=(K_1, K_2, K, \ve \mu, \t)$ is valid if and only if $F(\tau)$ can be parametrized by the following formulas:
\begin{equation*}
\begin{array}{r@{\hskip -.5pt}rrlr@{\hskip -.25pt}l@{\hskip -.25pt}l@{\hskip -.5pt}rrr@{\hskip -.5pt}r}
F(&\hor K_1,& \ver K_2,& \tri K,& \ve \mu, \t) 
&=& \{ (\t, \t, \t) + \frac{1}{q}(&\lambda_1 \ve d^1,&\lambda_2 \ve d^2,&(\lambda_1 \ve d^1+\lambda_2 \ve d^2)) 
 &: \ve \lambda \in \Delta_2\},
\end{array}
\end{equation*}
where $\d^i$ is a direction of $K_i$ such that $\t$ is the vertex of $K$ in the intersection of the two directions and $\t+\frac1q\d^1, \t+\frac1q\d^2$ are the other two vertices.
\end{lemma}

\begin{lemma}
\label{lemSepEdge}
Suppose $\Tuple(I, J, K, \ve \sigma,  \t)$ is a valid 7-tuple with different types of edges $I,J \in \Iedge$ and a triangle $K \in \Itri$.  Then there exists $\u \in \verts(J)$ such that $\Tuple(I, \{\u\}, K_1, \ve \sigma,  \t)$ is a valid 7-tuple, where $K_1$ is an edge of $K$ with the same type as $I$.
\end{lemma}
\begin{proof}
We will prove this for the case where $I \in \P_{q,\hor}$, $J \in \P_{q,\ver}$. The other cases are similar so we will omit their proofs.
By \autoref{lem:F-formulas}, we can parametrize 
$$
F(\hor I, \ver J, \tri K, \ve \sigma, \t) 
= \{ (\sigma_1\u^1, \sigma_2\u^2, -\sigma_3\u^3) + \frac{1}{q}(\lambda_1 \sigma_1 \ve d^1,\lambda_2 \sigma_2\ve d^2,-\sigma_3(\lambda_1 \ve d^1+\lambda_2 \ve d^2)): \ve \lambda \in \Delta_2\},
$$
where $\sigma_2\u^2$, $\sigma_2\u^2+\frac1q\lambda_2\sigma_2\d^2$ are two vertices of $J$, and $\t = \u^1+\u^2 - \u^3$. 

Then $K_1 = \{-\sigma_3\u^3-\frac1q\sigma_3\lambda_1\d^1:\lambda_1\in[0,1]\}$, $\sigma_1\u^1 = \sigma_1\t - \sigma_1\sigma_2(\sigma_2\u^2) -\sigma_1\sigma_3(-\sigma_3\u^3)$,
and for $\d = -\sigma_3\d^1$,
$$F\Tuple(I, \{\sigma_2\u^2\}, K_1, \ve \sigma,  \t)=\{ (\sigma_1\u^1, \sigma_2\u^2, -\sigma_3\u^3) + \tfrac{1}{q}\lambda (-\sigma_1\sigma_3\d, \0, \d) : \lambda \in \Delta_1\}.$$ 
Therefore, $\Tuple(I, \{\sigma_2\u^2\}, K_1, \ve \sigma,  \t)$ is a valid 7-tuple.
\end{proof}

 \subsection{Triangulation properties}
\label{sec:triangulation}

The specific choice of the triangulation $\P_q$ lends itself to strong unimodularity properties.  
This will allow us to demonstrate that the important additivities can be written using the original faces of the  complex $\P_q$.

We begin with a structural lemma about $\P_q$ regarding Minkowski sums that seems to be a key ingredient in preserving the structure of the complex under projections of additive faces.

\begin{lemma}[Lemma 4.3, \cite{bhk-IPCOext}]\label{lemma:I+J-all-cases}
Let $I, J \in \P_q$.  Then $- I$ and $I + J$ are unions of faces in $\P_q$.
\end{lemma}

Consider a pair of matrices $(A,B)$  with $A \in \R^{m \times p}$ and $B \in \R^{m' \times p}$.  For an integer $n$, the matrix ${A \choose B}^{(n)}$  is the  $n$-fold matrix is the $(n\cdot m + m') \times (n\cdot p)$ block structured matrix
$$
{A \choose B}^{(n)} := \begin{pmatrix}
    A & \\
    & A\\
    & &  \ddots & \\
    & &   & A\\
    B & B & \cdots & B\end{pmatrix}
$$

\begin{lemma}
\label{lem:polyhedra-integral}
Let $A$ be defined in~\eqref{eq:A-matrix} and 
consider the polyhedron 
$$
P = \{(\x^1, \dots,\x^n) :  Ax^i \leq \b^i \text{ for } i=1, \dots, n, \text{ and } \sum_{i=1}^p \x^i - \sum_{i={p+1}}^n \x^i = \bf{0} \}
$$
If $\b^i \in \Z^6$ for all $i=1\dots, n$, then $P$ is an integral polyhedron (as a consequence of its constraint matrix being totally unimodular).  Furthermore, the projections $p_i(P)$ onto the variables $\x^i$ are all integral polyhedra as well.
\end{lemma}
\begin{proof}
    Because totally unimodularity is preserved under the operations multiplying a row or column by $-1$, adding or removing a repeated row or column, and adding or removing a row or column that is standard unit vector, we only need to show that ${\bar{A} \choose I}^{(n)}$ is totally unimodular for $\bar{A} = \begin{bmatrix}
1 & 1
\end{bmatrix}$. Note that this is a $(0,1)$-matrix with at most two nonzeros per column.
    By \cite{heller1956extension} (also see, \cite[Corollary 4.8]{ccz-ipbook}), we know that it is totally unimodular because it has an equitable bicoloring, i.e., the rows corresponding to $\bar{A}$ is one color and the rows corresponding to $I$ is another color.
\end{proof}
The properties in \autoref{lemma:vertices} provide an easy method to compute $E(\pi, \P_q)$
by using simple arithmetic and set
membership operations on vertices of~$\P_q$.  

\begin{lemma}[Variation of {\cite[Lemma 4.4]{bhk-IPCOext}}. 
 See also {\cite[Lemma 7.3]{igp_survey, igp_survey_part_2}}]
 \label{lemma:vertices}
\label{lemma:faces}
  Let $I_1, I_2, I_3$ be faces in the complex~$\P_q$, $\ve \sigma \in \{-1,0,1\}^3$, and $\t \in \tfrac{1}{q}\Z^2$.   Let  $F = F\Tuple(I_1, I_2, I_3, \ve \sigma, \t)$.

\begin{enumerate}
    \item[(i)] The projections $p_1(F)$, $p_2(F)$, and $p_3(F)$ are faces in the complex $\P_q$.  Call these faces $I'_1, I'_2, I'_3$, respectively.  Then 
    $F = F\Tuple(I'_1, I'_2, I'_3, \ve \sigma, \t)$.
    \item[(ii)]  The vertices of $F$ are in $\tfrac{1}{q}\Z^2 \times \tfrac{1}{q}\Z^2 \times \tfrac{1}{q}\Z^2$.
\end{enumerate}
\end{lemma}

\begin{proof}
By definition of $F$, we know that $F = F\Tuple(I_1, I_2, I_3, \ve{\sigma}, \t)$ for some $I_1, I_2, I_3 \in \P_q$. From \autoref{lem:projections}, we have the following explicit formulas for the projections:
\[
\begin{aligned}
p_1(F) &= I_1 \cap \left[\sigma_1 \t - \left(\sigma_1 \sigma_2 I_2 + \sigma_1 \sigma_3 I_3\right)\right], \\
p_2(F) &= I_2 \cap \left[\sigma_2 \t - \left(\sigma_2 \sigma_1 I_1 + \sigma_2 \sigma_3 I_3\right)\right], \\
p_3(F) &= I_3 \cap \left[\sigma_3 \t - \left(\sigma_3 \sigma_2 I_2 + \sigma_3 \sigma_1 I_1\right)\right].
\end{aligned}
\]

To show that $p_1(F)$ is a face of $\P_q$, note that $I_1 \in \P_q$ by assumption. Through multiple applications of  \autoref{lemma:I+J-all-cases}, it follows that $\left[\sigma_1 \t - \left(\sigma_1 \sigma_2 I_2 + \sigma_1 \sigma_3 I_3\right)\right]$ is a union of faces of $\P_q$. Since $\P_q$ is a polyhedral complex (and hence closed under intersections), the intersection of $I_1$ with this union of faces is a convex union of faces that is a subset of $I_1$, and thus must be a face of $I_1$, and therefore  a face of $\P_q$.  Thus, $I'_1 := p_1(F)$ is a face of $\P_q$.  This follows similarly for $I'_2 = p_2(F)$ and $I'_3 = p_3(F)$.

Therefore, $F \subseteq I'_1 \times I'_2 \times I'_3 \subseteq I_1 \times I_2 \times I_3$.  It follows that $F = F\Tuple(I'_1, I'_2, I'_3, \ve \sigma, \t)$.

The intersection with the translated scaled sets $\sigma_1 \t - (\sigma_1 \sigma_2 I_2 + \sigma_1 \sigma_3 I_3)$ preserves the structure of $I_1$ because the translation and scaling operations do not introduce elements outside $\P_q$. Hence, $p_1(F)$ is a face of $\P_q$. The same argument applies to $p_2(F)$ and $p_3(F)$.

Part (ii) follows from \autoref{lem:polyhedra-integral}.
\end{proof}

\section{Interval Lemma - Applications and Hidden Functional Equation}
\label{sec:interval-lemma-applications-new-results}
Now that we understand the faces that describe possibile additivies and generalized descriptions of these additivities, we can apply \autoref{lem:projection_interval_lemma-corollary} to understand its implications.

Applying to the valid 7-tuple, we have the following corollary.
\begin{corollary}\label{cor:two_triangle_one_edge}
If $\bpi$ is continuous and  additive over the valid tuple $\Tuple(\tri I, \hor J, \tri K, \ve \sigma, \0)$, then $\bpi$ is affine over $I,J,K$ in the $\hor$-direction.
\end{corollary}
\begin{proof}
Apply \autoref{lem:projection_interval_lemma-corollary} on $F=\{(\u,\v)\st \u\in \sigma_1 I, \v\in \sigma_2 J, \u+\v\in -\sigma_3 K\}$ and $L=\{(x,0)\st x\in\R\}$ corresponding to the edge type of $J$. Because the tuple is valid, we have $\p_1(F) = \sigma_1 I, \p_2(F)=\sigma_2 J, \p_3(F) = -\sigma_3 K$.
\end{proof}

\subsection{Hidden Cauchy functional equation through multiple relations}
\label{subsec:hidden-interval-lemma}
In prior work,  it sufficed to consider additives in tuples individually when applying functional equation results.  Here, however, we need to see how  multiple additivities can interact.   
We combine relations to find a hidden Cauchy Functional Equation on Additive Domains.  In particular, when relations have overlapping projections.  See for instance Figure~\ref{fig:overlapping_K} where the projection $\p_3$ overlaps for two additivities.
\begin{figure}[h]
\begin{minipage}{0.4\textwidth}
\begin{center}
\begin{tikzpicture}

\draw [pink, ultra thick] (2,0) -- (3,0) node[midway, below] {$I_1$};
\draw [pink, ultra thick] (0,3) -- (0,4) node[midway, left] {$J_1$};

\draw [blue, ultra thick] (0,1) -- (1,1) node[midway, below] {$I_2$};
\draw [blue, ultra thick] (2,2) -- (1,3) node[midway, left] {$J_2$};

\node at (2.25,3.5) {$K$};

\draw (2,3) -- (3,3) node[midway, below] {$K_1$};
\draw (2,3) -- (2,4) node[midway, left] {$K_2$};

\draw [blue, dashed, fill=blue, fill opacity=0.5] (1,4) -- (2,4) -- (3,3) -- (2,3) -- cycle;

\draw [pink, dashed, fill=pink, fill opacity=0.5] (2,3) -- (3,3) -- (3,4) -- (2,4) -- cycle;

\draw[->] (-0.5,0) -- (4.5,0) node[right] {$x$};
\draw[->] (0,-0.5) -- (0,4.5) node[above] {$y$};

\foreach \x in {1,2,3,4}
    \draw (\x cm,1pt) -- (\x cm,-1pt) node[anchor=north] {$\x$};
\foreach \y in {1,2,3,4}
    \draw (1pt,\y cm) -- (-1pt,\y cm) node[anchor=east] {$\y$};
\def \dotsize {2pt}
\draw[fill=pink] (2,0) circle (\dotsize);
\node at (2,0) [above] {$\u^1$};
\draw[fill=pink] (0,3) circle (\dotsize);
\node at (0,3) [right] {$\u^2$};
\draw[fill=blue] (2,2) circle (\dotsize);
\node at (2,2) [below right] {$\tilde
\u^2$};
\draw[fill=blue] (0,1) circle (\dotsize);
\node at (0,1) [above right] {$\tilde\u^1$};
\end{tikzpicture}
\end{center}
\end{minipage}
\begin{minipage}{0.55\textwidth}
\caption{
The plot features sets $I_1$, $J_1$ as pink line segments on x-axis (2,0)--(3,0) and y-axis (0,3)--(0,4), respectively. $I_2$ and $J_2$ are blue line segments from (0,1)--(1,1) and (2,2)--(1,3) respectively. The pink and blue polygons represent the Minkowski sums $I_1 + J_1$ and $I_2 + J_2$. Lastly, $K_1$ and $K_2$ denote sides of set $K$, outlined by the intersections of the pink polygon and the blue polygon.
}
\label{fig:overlapping_K}
\end{minipage}
\end{figure}

\begin{lemma}[Hidden Interval Lemma]
\label{lem:hidden_interval_lemma}
Let $f,g,h \colon \R^2 \to \R$ be continuous functions. Let $F, \tilde F \subseteq \R^2 \times \R^2$ convex
such that $f(\u) + g(\v) = h(\u+\v)$ for all $(\u, \v) \in F \cup \tilde F$. 

Suppose that $\p_3(F) = \p_3(\tilde F)$ and is full dimensional in $\R^2$.  Furthermore, suppose that 
$$
\p_i(F) = \{  \u^{i} + \lambda_i \d^{i} : \lambda_i \in [0,1]\}, 
\ \ \text{ and } 
\p_i(\tilde F) = \{ \tilde \u^{i} + \tilde \lambda_i \tilde \d^{i} : \tilde \lambda_i \in [0,1]\} \ \ \text{ for } i=1,2
$$
 such that $\d^{1} = \tilde \d^{1}$.  
 That is, the projections $\p_1$ for both convex sets $F, \tilde F$ are simply translates of each other.
 
 Furthermore, assume that the pairs $(\d^1, \d^2)$ , $(\d^1, \tilde \d^2)$, and $(\d^2, \tilde \d^2)$
 are all linearly independent. 
 
Then  there exists a vector $\cve\in \R^2$ such that $f, g$ and $h$ are affine with gradient  $\cve$ with respect to the linear space defined by $\d^{1}$ over $\p_i(F)$ and $\p_i(\tilde F)$ for $i=1,2,3$.
\end{lemma}
\begin{proof}
Without loss of generality, we will assume that $\d^1 = \tilde \d^1 = \e^1$.
Let 
$
D = \begin{bmatrix} \d^1 & \d^2\end{bmatrix}, \text{ and } \tilde D = \begin{bmatrix} \tilde \d^1 & \tilde \d^2 \end{bmatrix} =  \begin{bmatrix}  \d^1 & \tilde \d^2 \end{bmatrix}.
$
Since $\p_i(F), \p_i(\tilde F)$ are line segments for $i=1,2$, any point $\u \in \p_3(F) = \p_3(\tilde F)$ can be written uniquely as 
$$
\u = \u^1 + \u^2 + D \bm{\lambda}\ \ \  \text{ or equivalently } \ \ \ \u = \tilde \u^1 + \tilde \u^2 + \tilde D \tilde{\bm{\lambda}},
$$
where $\bm{\lambda}, \tilde{\bm{\lambda}} \in [0,1]^2$. 
Setting these relations equal, we can derive an expression for $\tilde{\bm \lambda}$ in terms of $\bm \lambda$ as
\begin{align*}
\tilde{\bm \lambda} &= \tilde D^{-1}(\u^1 + \u^2 - \tilde \u^1 - \tilde \u^2) + \tilde D^{-1} D \bm \lambda
= \bar\u + \begin{bmatrix} \e^1 & \bar{\d} \end{bmatrix}\bm \lambda
\end{align*}
for appropriate choices of $\bar{\u}$ and $\bar{\d}$.

So 
$\tilde{\lambda}_1 = \bar u_1 + \lambda_1  + \bar d_1\lambda_2$ and $
\tilde{\lambda}_2 = \bar u_2 + \bar d_2 \lambda_2.
$
Next, by substitution,
\begin{align*}
f( \u^1 +  \lambda_1 \d^1 ) + g(\u^2  + \lambda_2 \d^2)
&=h( \u^1 + \u^2 + D \bm{\lambda})\\
& = h( \tilde{\u}^1 + \tilde \u^2 + \tilde D \tilde{\bm \lambda})\\
&=f(\tilde  \u^1 +  \tilde \lambda_1 \tilde\d^1 ) + g(\tilde \u^2  + \tilde\lambda_2 \tilde\d^2)\\
& = f(\tilde  \u^1 +  \left(\bar u_1 + \lambda_1  + \bar d_1\lambda_2\right) \tilde\d^1 ) + g(\tilde \u^2  + \left(\bar u_2 + \bar d_2 \lambda_2 \right) \tilde\d^2)
\end{align*}
Rewriting, we have
\begin{equation*}
\underbrace{f( \u^1 +  \lambda_1 \d^1 ) }_{\hat f(\lambda_1)} + 
\underbrace{g(\u^2  + \lambda_2 \d^2)- g(\tilde \u^2  + \left(\bar u_2 + \bar d_2 \lambda_2 \right) \tilde\d^2)}_{\hat g(\lambda_2)}  = 
\underbrace{f(\tilde  \u^1 +  \left(\bar u_1 + \lambda_1  + \bar d_1\lambda_2\right) \tilde\d^1 )}_{\hat h(\lambda_1 + \bar d_1 \lambda_2)} .
\end{equation*}

Moreover, this holds for all $\{(\lambda_1, \lambda_2) \in [0,1]^2 : \lambda_1 + \lambda_2 \leq 1\}$.
By applying \autoref{lem:projection_interval_lemma_fulldim} to $\hat f, \hat g, \hat h$, we know that 
\begin{align*}
    f(\u^1 +  \lambda_1 \d^1) &= c\lambda_1\\
    g(\u^2  + \lambda_2 (\bar{d}_1 \d^1+\tilde{\d}^2))- g(\tilde \u^2  + \tilde\lambda_2 \tilde\d^2) &= g(\u^2  + \lambda_2 \d^2)- g(\tilde \u^2  + \left(\bar u_2 + \bar d_2 \lambda_2 \right) \tilde\d^2) = c\bar d_1\lambda_2\\
    f(\tilde  \u^1 +  \tilde \lambda_1 \tilde\d^1 )&=f(\tilde  \u^1 +  \left(\bar u_1 + \lambda_1  + \bar d_1\lambda_2\right) \tilde\d^1 ) = c(\lambda_1 + \bar d_1\lambda_2).
\end{align*}
Thus, $f$ is affine with gradient $\cve := c\, \d^1$ over $\p_1(F)$ and $\p_1(\bar{F})$, and $g$ is affine with gradient $\cve := c\, \d^1$ with respect to the linear space $\mathrm{span}(\d^1)$ over $\p_2(F)$ and $\p_2(\bar{F})$. By additivity, we know that $h$ is affine with gradient $\cve := c\, \d^1$ with respect to the linear space $\mathrm{span}(\d^1)$ over $\p_3(F)=\p_3(\bar{F})$.
\end{proof}

\begin{corollary}[Interval Lemma on Edges]
\label{lemILedges}
Pick distinct $\EquiParam{1}, \EquiParam{2}, \EquiParam{3} \in \{ \ver, \hor, \diag\}$.
Suppose $\bpi$ is additive over $\tau^1 = 
(I_1,J_1,\tri K, \ve \sigma^1, \t^1)$ and $\tau^2 = (I_2, J_2, \tri K, \ve \sigma^2, \t^2)$  with $K \in \Itri$, $I_1, I_2 \in \P_{q, \EquiParam{1}}$, $J_1 \in \P_{q,\EquiParam{2}}$ and $J_2 \in \P_{q, \EquiParam{3}}$.  Then $\pi$ is affine in the $\EquiParam{1}$ direction in the faces $I_1, I_2$ and $K$.
\end{corollary}

\newcommand{\F}{\mathcal F}
\newcommand{\fvec}{f}
\newcommand{\ff}{f}

\section{Underlying theory of perturbations with finite reflection groups}
\label{sec:underlying-theory}
We prove a fundamental lemma underlying some of the concepts for proofs in our papers on equivariant perturbations for the infinite group problem.  We do this without defining anything about the infinite group problem and instead, just show how certain function spaces cut by linear equations decompose to an infinite set of finite dimensional linear systems.

 We say that the domain $D\subseteq [0,1]$ is \emph{symmetric} if $D = 1 - D:= \{1 - x : x \in D\}$.
Let $K$ be a finite set and let $(a_k, b_k) \in \R$ for all $k \in K$.  Let $I_k, J_k \subseteq [\ell]:=\{1,2,\dots,\ell\}$ for $k \in K$, where $\ell$ is a positive integer.
    Consider the function space defined by linear equations parametrized by the symmetric domain \(D\):
    \begin{equation}
        \label{eq:linear-equations}
        \F (D) := \left\{ \fvec \colon D\times [\ell] \to \R : \sum_{i \in I_k} a_i \ff(x,i) + \sum_{j \in J_k} b_j \ff(1-x,j) = 0
        \quad \forall \; x \in D, \; k \in K\right\}.
    \end{equation}

Define 
\begin{equation}
\label{eq:delta-general}
\Delta_k \fvec(x) := \sum_{i \in I_k} a_i {\ff}(x,i) + \sum_{j \in J_k} b_j {\ff}(1 - x,j).
\end{equation}

\begin{definition}
    The \textit{orbit} of a point \(x \in [0, 1]\), denoted by \( \mathcal{O}(x) \), is defined as:
    \[
    \mathcal{O}(x) = \{(x, i), (1 - x, i) : i \in [\ell]\}.
    \]
    The orbit includes both the point \( x \) and its reflection \( 1 - x \), and all the indices in $[\ell]$. Note that the orbit of \( x=1/2 \) is \( \mathcal{O}(1/2) = \{(1/2, i) : i \in [\ell]\} \).
\end{definition}

\begin{lemma}
    \label{lem:orbits}
    The orbits \( \mathcal{O}(x) \) for points \( x \in [0, 1] \) have the following properties:
    \begin{enumerate}
        \item For any \( x \in [0, 1] \), we have \( \mathcal{O}(x) = \mathcal{O}(1 - x) \), meaning the system behaves symmetrically with respect to \(x\) and \(1 - x\).
        \item The orbits of distinct points \( x \) and \( x' \) in the interval \( [0, 1/2] \) are disjoint:
        \[
        \mathcal{O}(x) \cap \mathcal{O}(x') = \emptyset \quad \text{for} \; x \neq x', \; x, x' \in [0, 1/2].
        \]
    
        \item The set of all points \( (x, i) \) for \( x \in [0, 1] \) and \( i = 1, \dots, \ell \) can be expressed as the union of orbits for \( x \in [0, 1/2] \):
        \[
        \{(x, i) :  x \in [0, 1], i \in [\ell]\} = \bigsqcup_{x \in [0, 1/2]} \mathcal{O}(x).
        \]
    \end{enumerate}
\end{lemma}
The proof is somewhat trivial, so we provide it only in the appendix.

We next prove our main result of the section.  For a domain $D \subseteq [0,1]$ function $\fvec$ with $\fvec \colon D\times [\ell] \to \R$, we say that $\fvec$ is nontrivial if $\fvec \not\equiv 0$.  We say that $f$ is piecewise linear if each function $f(\cdot, i)$ is piecewise linear for $i \in [\ell]$.
\begin{lemma}[Nontrivial Infinite System  $\iff$ Nontrivial Finite System]
\label{lem:infinte-system-to-finite-system-edges}

    For any integer $m>2$, the following are equivalent:
    \begin{enumerate}
        \item There exists a nontrivial solution for $\F([0,1])$.
        \item There exists a nontrivial  continuous piecewise linear solution for $\F([0,1])$ with $f(0,i) = f(1,i) = 0$ for all $i=1, \dots, \ell$.
        \item There exists a nontrivial solution for $\F(\{\tfrac{1}{m}, \tfrac{m-1}{m}\})$.
    \end{enumerate}
\end{lemma}

\begin{proof}
We show $(2) \Rightarrow (1) \Rightarrow (3) \Rightarrow (2)$.\medbreak

\textbf{Step 1.} Clearly $(2) \Rightarrow (1)$.\medbreak

\textbf{Step 2.} We show that $(3) \Rightarrow (2)$.\\
Suppose there exists a nontrivial solution $\fvec$ for $\F\left(\left\{\tfrac{1}{m}, \tfrac{m-1}{m}\right\}\right)$. We construct continuous piecewise linear functions $\hat{\ff}(\cdot, i): [0,1] \to \mathbb{R}$ such that $\hat{\ff}(0,i) = \hat{\ff}(1,i) = 0$ for all $i$, and $\hat{\fvec}$ satisfies $\F([0,1])$.

\textit{Construction of $\hat{\ff}(\cdot, i)$:}
Let $B = \{x_0, x_1, x_2, x_3\} = \{0, \tfrac{1}{m}, \tfrac{m-1}{m}, 1\}$. Define $\hat{\ff}(x,i)$ by linear interpolation over $B$:
For $x \in [0,1]$, 
\[
\hat{\ff}(x,i) = \sum_t \lambda_t \ff(x_t,i)  \text{ 
where } \sum_t \lambda_t = 1, \lambda_t \geq 0, \mathrm{SOS2}(\lambda),\]
where $\mathrm{SOS2}(\lambda)$ is a Special Ordered Set of type 2, which means that at most two $\lambda_t$ are positive and if so, then they must be consecutive.
Since $\hat{\ff}(0,i) = \hat{\ff}(1,i) = 0$, and $\ff(\tfrac{1}{m},i)$, $\ff(\tfrac{m-1}{m},i)$ are given, $\hat{\ff}(x,i)$ is continuous and piecewise linear in $x$.

\textit{Verification that the system is satisfied:}
For all $x \in [0,1]$ and $k \in K$, we need to show $\Delta_k \hat \fvec(x) = 0$ (see \eqref{eq:delta-general}).  
On each interval $[x_t, x_{t+1}]$, the functions $\hat{\fvec}(x,i)$ and $\hat{\ff}(1 - x,j)$ are linear, so the left-hand side $\Delta_k \hat \fvec(x)$ is linear in $x$. At the breakpoints $x_t$, we have $\Delta_k \hat \fvec(x_t) = 0$
because the original solution satisfies the system at $x = \tfrac{1}{m}, \tfrac{m-1}{m}$, and $\hat{\ff}(0,i) = \hat{\ff}(1,i) = 0$.
Therefore, the left-hand side is zero at both ends of each interval, implying it is zero throughout the interval. Thus, $\hat{\fvec}$ satisfies $\F([0,1])$, and (2) holds.
\medbreak

 \textbf{Step 3:} We show that (1) implies (3).\\
Suppose there is a nontrivial $\fvec \in \F([0,1])$.  Then there exists an index $i \in [\ell]$ and a point $\bar x \in [0,1]$ such that $\ff(\bar x, i) \neq 0$.
Without loss of generality, $\bar x \in [0,\tfrac{1}{2}]$ and $i=1$.  

Let $\bar D = \{\bar x, 1-\bar x\}$. By \autoref{lem:orbits}, the function values used in the equation system $\F(\bar D)$ are restricted to $f(x, i)$ for $(x,i) \in \mathcal O(\bar x)$.   
Since $\fvec \in \F([0,1])$, it follows that $\fvec|_{\bar D \times [\ell]} \in \F( \bar D)$.

Next,  let $\tilde \fvec \colon \tilde D\times [\ell] \to \R$ with $\tilde D = \{\tfrac{1}{m}, \tfrac{m-1}{m}\}$ be defined as 
$$
\tilde f(\tfrac{1}{m},i) = f(\bar x,i), \qquad \tilde f(\tfrac{m-1}{m},i) = f(1-\bar x,i).
$$
Then $\fvec|_{\bar D \times[\ell]} \in \F( \bar D)$ implies that $\tilde \fvec \in \F(\tilde D)$.

This follows by noticing that $\Delta_k \fvec(\bar x) = 0$ implies that $\Delta_k \tilde \fvec(\tfrac{1}{m}) = 0$ (and respectively for $1 - \bar x$ and $\tfrac{m-1}{m}$), which is obtained by substituting $\ff(\bar x,i)$ with $\tilde \ff(\tfrac{1}{m},i)$ and $\ff(1 - \bar x,i)$ with $\tilde f(\tfrac{m-1}{m},i)$ in the system of equations.

Note that the above implication could be an ``if and only if" when $\bar x \neq \tfrac{1}{2}$, but since we allow that option, we only make this a one-way implication.
    \end{proof}

\subsection{Applying the general systems to edges}
We show how two particular edge relations can be formulated in terms of \eqref{eq:linear-equations}.  This lemma is the key to the final step of the proof of our main result.
\begin{lemma}
\label{lem:point-edge-system}
    Consider Items~\ref{item:point-edge} and~\ref{item:three-diags}  from \autoref{lemma:cases}.  
    That is, consider a valid 7-tuples\\ \(\tau^1 = (\ver\hor\diag I_1, \point I_2, \ver \hor\diag I_3, \ve \sigma, \t)\),
    \(\tau^2= (\ver I_1, \hor I_2, \diag I_3, \ve \sigma, \t)\).
The equations that describe a function $\bpi$  being additive over $\tau^1$ and $\tau^2$ can be transformed to the format of $\mathcal F(D)$ in \eqref{eq:linear-equations}.
\end{lemma}
\begin{proof}
For each tuple, we transform the equation $\sum_{i=1}^3 \sigma_i \pi(\x^i)$ for $(\x^1, \x^2,\x^3) \in F(\tau)$ into the format of $\F([0,1])$.  
By \autoref{lem:F-formulas}, $F(\tau^1)$ and $F(\tau^2)$ are both 1-dimensional polytopes that can be parametrized by some variable $\lambda \in [0,1]$.
For each $\tau^1, \tau^2$, by setting $f(\lambda, i) := \pi( \tilde \u^i + \frac{1}{q} \lambda \tilde \d^i)$ for $\tilde \u^i$ and $\tilde \d^i$ given from the formulas, in \autoref{lem:F-formulas}, we see that we exactly recover the format of the systems of equations $\F([0,1])$.
This completes the proof.
\end{proof}

\section{Space of Perturbations and 7-tuples}
\label{sec:2D-main-result}

\subsection{Systems of equations}

We now consider the set of  valid 7-tuples $\mathcal E(\pi, \P_q)$ and subdivide it by types of 7-tuples. This combinatorial representation can then be made finite by choosing representatives of $I_1$ and $I_2$ in $[0,1]^2$ under the action of~$\Z^2\times\Z^2$.

For
$\EquiParam1, \EquiParam2, \EquiParam3 \in \{\point, \ver, \hor, \diag,\tri\} $, define
$$
\mathcal E_{\EquiParam1, \EquiParam2, \EquiParam3} := \bigcup \{ (I_{\phi(1)}, I_{\phi(2)}, I_{\phi(3)}, \ve \sigma,  \t) \in \mathcal E \st I_{\phi(i)} \in \EquiParam{i} \
\text{for  $i=1,2,3$}, \phi \in \mathrm{Perm}(\{1,2,3\})\,\},
$$
where $\mathrm{Perm}(\{1,2,3\})$ is the set of permutation functions on $\{1,2,3\}$.

For example,
$$
\mathcal E_{\tri, \tri, \tri} = \{ (I_1, I_2, I_3, \ve \sigma,  \t) \in \mathcal E : I_i \in \tri\}.
$$

So
\begin{equation}
\label{eq:additive-union}
\mathcal E = \bigcup_{\EquiParam1, \EquiParam2, \EquiParam3 \in \{\point, \ver, \hor, \diag,\tri\}} \mathcal E_{\EquiParam1, \EquiParam2, \EquiParam3}.\end{equation}

\begin{remark}
\label{remContinuity}
By the definition of additivity over a valid 7-tuple, if $\bpi$ is additive over a valid 7-tuple $(I, J, K, \ve \sigma, t)$, then for any valid 7-tuple $(I', J', K', \ve\sigma, \t)$ with $I' \subseteq I$, $J' \subseteq J$, $K' \subseteq K$ with $I', J', K' \in \P_q$, we have $\bpi$ is additive over $(I', J', K', \ve\sigma, \t)$.
\end{remark}

  \subsection{Restricting to the lattice when working on restricted domain}

 We need to adapt {\cite[Lemma 3.14]{igp_survey, igp_survey_part_2}} to a more general setting that handles working on a restricted domain and uses more general constraints.

 In our proof structure of our main result, we reduce the domain by deducing certain relations about corresponding functions. So it could be that we work with only a partial triangulation of the space. 
 
We will need the following notation for the next result. Define $\boldsymbol\sigma \cdot \pi\colon\R^k \times \R^k \times\R^k \to \R$ by $(\x^1, \x^2, \x^3) \mapsto \sum_{i=1}^3 \sigma_i \pi(\x^i)$. Note that for $\boldsymbol{\sigma} = (1,1,-1)$, we have $\boldsymbol{\sigma} \cdot \pi(\x^1, \x^2, \x^1 + \x^2) =  \pi(\x^1) + \pi( \x^2) - \pi(\x^1 + \x^2)$.

For a polytope $P$, let $\verts(P)$ denote the vertices of $P$.  For a polyhedral complex $\P$, let $\verts(\P) := \bigcup \{ \verts(P) : P \in \P\}$.

A triangulation $\mathcal T$ in $\R^n$ is a polyhedral complex such that every $P \in \mathcal T$ is a simplex of some dimension $1 \leq d \leq n$ or the empty set.

For a function $\phi \colon \verts(\T) \to \R$, we define the piecewise linear interpolation $\mathrm{pwl}(\phi)$ by 
$\x \mapsto \sum_{\v \in \verts(P)} \lambda_\v \phi(\v)
$
where $P$ is a minimal face of $\T$ containing $\x$, $\sum_{\v \in \verts(P)} \lambda_\v \v = \x$ and $\sum_{\v \in \verts(P)} \lambda_\v =1$, $\lambda_\v \geq 0$.

The choice of $\lambda_\v$ is unique since $\T$ is a triangulation and $P$ is chosen as a minimal face.

\begin{lemma}
\label{lem:tri-restriction-restricted-domain-new}
 Suppose $\T$ be a subcomplex of a triangulation of $\R^k$ such that there exists $q \in \mathbb{N}$ such that $\verts(\T) \subseteq \tfrac{1}{q}\Z^k$.  
Let $\mathcal E$ be a set of valid 7-tuples $\tau = (I_1, I_2, I_3, \boldsymbol \sigma, \t)$ with $I_i \in \T$, $\boldsymbol{\sigma} \in \{-1,1\}^3$ and $\t \in \tfrac{1}{q}\Z^k$.  Suppose that for each $\tau \in \mathcal E$ we have that $p_i(F(\tau)) \subseteq \verts(\T)$ for $i=1,2,3$.

\begin{enumerate}
\item If $\bar \pi \in \bar \Pi^{\E}_{\ZZ, \T}$, then $\bar \pi|_{\tfrac{1}{q}\Z^k} \in \bar\Pi^{{\E}\cap\frac1{q}\Z^k}_{\mathcal{Z}\cap \frac1{q}\Z^k,\T \cap \frac1{q}\Z^k}$.\label{item:restriction-implies-perturbation-space}
\item Let $\bar \pi_{\T}$ be piecewise linear over $\T$.
 Then 
$\bar \pi_{\T} \in \bar \Pi^{\E}_{\mathcal{Z}, \T}$ if and only if 
  $\bar \pi_{\T}|_{\tfrac{1}{q}\Z^k} \in \bar\Pi^{{\E}\cap\frac1{q}\Z^k}_{\mathcal{Z}\cap \frac1{q}\Z^k,\T \cap \frac1{q}\Z^k}$. \label{lem:tri-restriction-part1-new}

  \item For every $\bar \pi \in \bar \Pi^{\E}_{\mathcal{Z}}$ there exists unique  $\bar \pi_{\T} \in \bar \Pi^{\E}_{\mathcal{Z}, \T}$ and $\bar \pi_0 \in \bar \Pi^{\E}_{\mathcal{Z} \cup \verts(\T), \T}$
  such that 
  $$
\bar \pi = \bar \pi_{\T} + \bar \pi_0.
  $$\label{lem:tri-restriction-part2-new}
  \end{enumerate}
\end{lemma}

\begin{proof}
\begin{enumerate}
    \item Trivially, $\bpi$ must be additive over $\E\cap \frac1q\Z^k$, i.e., $\boldsymbol{\sigma} \cdot \bpi(\x^1,\x^2,\x^3)=0$ for all $\sum_{i=1}^3\sigma_i\x^i=\t$ and $\x^i\in\frac1q\Z^k$, where $\x^i\in I_i$ and $\tau=(I_1, I_2, I_3, \ve\sigma,\t)\in\E$.  But since $\verts(\T) \subseteq \tfrac{1}{q}\Z^k$, we must have that $\bpi|_{\tfrac{1}{q} \Z^k}$ is additive over $\E\cap\frac1q\Z^k$.

\item Let $\tau = (I_1, I_2, I_3, \boldsymbol \sigma, \t) \in \mathcal E$ be a valid 7-tuple.
Since $\tau$ is a valid 7-tuple, $I_i \in \T$.   And since $\bar \pi_{\T}$ is piecewise linear over $\T$, we have that $\boldsymbol{\sigma} \cdot \pi_{\T}$ is affine over $F(\tau)$.  
Thus, 
$
\boldsymbol{\sigma} \cdot \pi_{\T}|_{F(\tau)} \equiv 0
$
if and only if 
$
\boldsymbol{\sigma} \cdot \pi_{\T}|_{\verts(F(\tau))} \equiv 0.
$
Since $p_i(\verts(F(\tau))) \subseteq \verts(\T) \subseteq \tfrac{1}{q}\Z^k$, we have that 
$
\boldsymbol{\sigma} \cdot \pi_{\T}|_{\verts(F(\tau))} \equiv 0
$
if and only if 
$
\left(\boldsymbol{\sigma} \cdot \bar \pi_{\T}|_{\tfrac{1}{q} \Z^k}\right)\bigg\rvert_{\verts(F(\tau))} \equiv 0.
$
Noting this for each $\tau \in \mathcal E$ completes the proof of part (2).

\item Next, let $\bar \pi \in \bar \Pi^{\E}_{\mathcal{Z}}$.  Let $\bar \pi_{\T}$ be the unique extension of $\bar \pi|_{\tfrac{1}{q} \Z^k}$ to a piecewise linear function over $\T$.  It follows that $\bar \pi_{\T}|_{\tfrac{1}{q} \Z^k} = \bar \pi|_{\tfrac{1}{q} \Z^k} \in \bar\Pi^{{\E}\cap\frac1{q}\Z^k}_{\mathcal{Z}\cap \frac1{q}\Z^k,\T \cap \frac1{q}\Z^k}$.  
Therefore, by part (2), $\bar \pi_{\T} \in \bar \Pi^{\E}_{\mathcal{Z}, \T}$.  Since $\bar \Pi^{\E}_{\mathcal{Z}, \T}$ is a vector space containing $\bar \pi$ and $\bar \pi_{\T}$, we have that $\pi_0 = \bar \pi - \bar \pi_{\T} \in \bar \Pi^{\E}_{\mathcal{Z}, \T}$, which establishes part (3).
\end{enumerate}
\end{proof}

\begin{remark}
\label{rem:tri-satisfy-lemma-new}
The polyhedral complexes $\P_{B}$ for $B = \tfrac{1}{q} \Z \cap [0,1)$ from \autoref{ex:1d-breakpoint-complex} and $\P_q$ from \autoref{ex:2d-standard-triangulation} are triangulations of $\R^1$ and $\R^2$, respectively, and satisfy the hypotheses of \autoref{lem:tri-restriction-restricted-domain-new}.  This is due to \autoref{lemma:vertices}. %
\end{remark}

\label{sec:main-result-space-transformations}

\subsection{Mappings on restriction and interpolation}

We construct a diagram to illustrate the interaction between restriction and interpolation across various perturbation spaces. This connection plays a crucial role in our main proof by demonstrating their correspondence, enabling us to relate these spaces while modifying the function’s domain to simplify our arguments.

\begin{lemma}
  \label{lemContinuousDiscrete}
  Suppose $\P$ is a $\Z^2$-periodic  collection of faces such that  $p(\E) \subseteq \P \otimes(\Ipoint \cup \Iedge)$, and $\verts(\P_q)\subseteq\mathcal{Z}\subseteq\P$.
  Then for $m\ge 1$, the following maps in the diagram exist:
  \begin{enumerate}
    \item a surjection by restriction $\phi_1 \colon \pi\mapsto \pi|_{\frac{1}{mq}\Z^2}$ from ${\widehat\Pi^{\E}_{\mathcal{Z},\P} }$ to ${\bar\Pi^{{\E}\cap\frac1{mq}\Z^2}_{\mathcal{Z}\cap\frac1{mq}\Z^2,\P \cap \frac1{mq}\Z^2}}$;%
    \item %
    a bijection
    by interpolation $\phi_2$ from ${\bar\Pi^{{\E}\cap\frac1{mq}\Z^2}_{\mathcal{Z}\cap\frac1{mq}\Z^2,\P \cap \frac1{mq}\Z^2}}$ to ${\bar\Pi^{\E}_{\mathcal{Z},\P} \cap \mathrm{Pwl}(\P\otimes\mathcal P_{mq})}$, where $\mathrm{Pwl}(\P\otimes\mathcal P_{mq})$ denotes the set of continuous piecewise linear functions on the refinement of $\P_q$ on the domain $\P$;
    
    \item an injection by identity embedding $\phi_3 \colon \pi \mapsto \pi$ from ${\bar\Pi^{\E}_{\mathcal{Z},\P} \cap \mathrm{Pwl}(\P\otimes\mathcal P_{mq})}$ to ${\widehat\Pi^{\E}_{\mathcal{Z},\P}}$.
  \end{enumerate}

\[
\begin{tikzcd}[row sep=0.1em]
  {\bar\Pi^{\E}_{\mathcal{Z},\P} \cap \mathrm{Pwl}(\P\otimes\mathcal P_{mq})} &&  && {\bar\Pi^{{\E}\cap\frac1{mq}\Z^2}_{\mathcal{Z}\cap \frac1{mq}\Z^2,\P \cap \frac1{mq}\Z^2}}\\
  && {\widehat\Pi^{\E}_{\mathcal{Z},\P}} &&
  \arrow["\phi_3"'{pos=0.7}, hook, from=1-1, to=2-3]
  \arrow["\phi_1 \textup{(restrict)}"', two heads, from=2-3, to=1-5]
  \arrow["\phi_2 \textup{ (interpolate)}"', curve={height=12pt}, tail, two heads, from=1-5, to=1-1]
\end{tikzcd}
\]

\end{lemma}

Note that the maps (1) and (3) also exist for any valid 7-tuples $\E$. The main proof ingredient for map (2) is if the additivity equation holds for the discrete function, then the additivity equation holds for the interpolation function.
\begin{proof}
\emph{Map (1).}
 Let $\bpi\in {\widehat \Pi^{\E}_{\mathcal{Z},\P} }$, the restriction $\bpi|_{\frac1{mq}\Z^2}$ is unique. The restriction map $\phi_1(\bpi)$ also inherits the additivity equations from $\bar\Pi^{\E}_{\mathcal{Z},\P}$, thus $\phi_1(\bpi)\in {\bar\Pi^{{\E}\cap\frac1{mq}\Z^2}_{\mathcal{Z}\cap\frac1{mq}\Z^2,\P \cap \frac1{mq}\Z^2}}$.

  To prove that $\phi_1$ is a surjection, for any $\bpi_{mq}\in {\bar\Pi^{{\E}\cap\frac1{mq}\Z^2}_{\mathcal{Z}\cap\frac1{mq}\Z^2,\P \cap \frac1{mq}\Z^2}}$, we have $\phi_1(\phi_2(\bpi_{mq}))=\bpi_{mq}$, where the interpolation $\phi_2(\bpi_{mq})$ is in ${\widehat\Pi^{\E}_{\mathcal{Z},\P} }$.\medbreak

\emph{Map (2).} Let $\bpi_{mq} \colon P \cap \tfrac{1}{mq} \Z^2 \to \R$. Let $\mathrm{pwl}(\bpi_{mq})$ be the continuous piecewise linear function obtained from $\bpi_{mq}$ via interpolation. By \autoref{lem:tri-restriction-restricted-domain-new}~\autoref{lem:tri-restriction-part1-new}, $\bpi_{mq} \in{\bar\Pi^{{\E}\cap\frac1{mq}\Z^2}_{\mathcal{Z}\cap\frac1{mq}\Z^2,\P \cap \frac1{mq}\Z^2}}$ if and only if 
$\mathrm{pwl}(\bpi_{mq})\in \bar \Pi^\E_{\ZZ,\P}$.  Since $\mathrm{pwl}(\bpi_{mq})|_{\tfrac{1}{mq} \Z^2} = \bpi_{mq}$, we conclude that $\phi_2$ is indeed a bijection.

\medbreak
  
\emph{Map (3).} Clearly any piecewise linear function is continuous, so the mapping as an injection is trivial.
\end{proof}

\section{Proof of main result}\label{sec:main-proof}

We state and prove our main result about the connection of piecewise linear 2D functions with the finite group problem.

We will often derive results about perturbation functions when they are continuous.  For this we define the set of continuous perturbation functions
$$
\widehat{\Pi}^{\E}_{\mathcal{Z},\P} :=  \{ \bar \pi \in \bar \Pi^\E_{\mathcal{Z},\P}: \bpi \text{ is Lipschitz continuous}\}.
$$
As functions in this space are continuous, we can think of $\ZZ$ as a polyhedral complex instead of just a list of polytopes. Thus, for any $I \in \ZZ$, we can assume that any subface $I' \subseteq I$ is also contained in $\ZZ$.

\begin{theorem}[Main Result]\label{thm:extreme-facet-and-restriction-2d-all}
\label{thm:extreme-restriction-m-2d-all}
  Let $m \in \Z_{\geq 3}$.  Let $\pi$ be a continuous piecewise linear  minimal
  valid function for $R_\f(\R^2,\Z^2)$  over $\P_q$  and suppose $\f \in \tfrac{1}{q} \Z^2$.    The following are equivalent:
\begin{enumerate}
\item\label{result1} $\pi$ is a facet for $R_\f(\R^2,\Z^2)$,
\item\label{result2} $\pi$ is extreme for $R_\f(\R^2,\Z^2)$,
\item\label{result3} $\pi|_{\frac{1}{mq}\Z^2}$ is extreme for $R_\f(\tfrac{1}{mq} \Z^2, \Z^2)$.
\end{enumerate}
\end{theorem}

The proof will track the space of perturbation functions as we transform additivities and restrict the space.  Thus, the technique here will require moving between different function spaces and maintaining a connection between these spaces.

 We focus on proving that \ref{result3} implies \ref{result1}.

\begin{proof}[Proof of \autoref{thm:extreme-facet-and-restriction-2d-all}]
As mentioned in \autoref{subsec:relation-between-three-notions}, facets are extreme functions~\cite[Lemma 1.3]{bhkm}, and hence \ref{result1} $\Rightarrow$ \ref{result2}.
By \autoref{thm:GJ-restrictions-2D},  \ref{result2} $\Rightarrow $ \ref{result3}.
We now show \ref{result3} $\Rightarrow$ \ref{result1}.

Let $\pi|_{\frac{1}{mq} \Z^2}$ be extreme for $R_\f(\tfrac{1}{mq} \Z^2, \Z^2)$.
Since the domain is finite, we know by \autoref{thm:finite-extremality} that 
\begin{equation}
\label{eq:finite-grid}
    \bar \Pi^{\mathcal E(\pi) \cap \frac{1}{mq} \Z^2}_{\emptyset, \frac{1}{mq}\Z^2} = \{0\}.
\end{equation}

\emph{For the sake of deriving a contradiction,} suppose that $\pi$ is not a facet for $R_\f(\R^2,\Z^2)$. By the Facet Theorem (\autoref{thm:facet}), this means that 
\begin{equation}
\label{eq:facet-non-empty}
  \bar \Pi^{\mathcal E(\pi)} \neq \{0\} .
\end{equation}
We will show that \eqref{eq:finite-grid} and \eqref{eq:facet-non-empty} are in contradiction.

Since $\pi$ is piecewise linear, $\pi$ is Lipschitz continuous.   By \autoref{Theorem:functionContinuous-all-cases}, any $\bar \pi$ in $\bar \Pi^{\E(\pi)}$ is also Lipschitz continuous. Furthermore, since $\pi$ is piecewise linear, we can combinatorialize the additivities  via \autoref{lemma:covered-by-maximal-valid-triples}
we can write
$$
\widehat \Pi^{\E(\pi, \P_q)}_{\emptyset, \P_q} = \bar \Pi^{\E(\pi)}.
$$
If there exists $\bpi\in \widehat \Pi^{\E(\pi, \P_q)}_{\emptyset, \P_q}$ with $\bpi|_{\frac1q\Z^2} \ne 0$, then by \autoref{lem:tri-restriction-restricted-domain-new} part~\ref{lem:tri-restriction-part1-new} with $\T = \P_{mq}$, there is a nontrivial $\bpi|_{\frac{1}{mq}\Z^2}\in \bar \Pi^{\E(\pi) \cap \frac{1}{mq} \Z^2}_{\emptyset, \frac{1}{mq}\Z^2}$, which contradicts \eqref{eq:finite-grid}.
Therefore $\widehat \Pi^{\E(\pi, \P_q)}_{\frac1q\Z^2, \P_q} =\bar \Pi^{\E(\pi)} \ne\{0\}$.

By \autoref{lemma:cases}, we know the exact structure of $\E(\pi, \P_q)$ and can expand this to a union of 7-tuples as \eqref{eq:additive-union}.  We will sequentially address these types of 7-tuples - we will modify additive 7-tuples, the zero that we enforce, and the domain  to obtain simpler perturbations spaces that we can analyze.

We will initialize
$
\E \leftarrow \E(\pi, \P_q), \ \ \ZZ \leftarrow \tfrac1q\Z^2, \ \  \P \leftarrow \P_q
  $
  and maintain a triple $(\E, \ZZ, \P)$ such that 
   \begin{equation}
       \label{eq:perturbation-correspondance}
  \begin{tikzcd}
          {\widehat\Pi^\Etuple_{\mathcal{Z},\P}}
          && {\widehat\Pi^{\Etuple'}_{\mathcal{Z}',\P'}} &&
          \arrow["{\mathrm{project}}", shift left=3, tail, two heads, from=1-1, to=1-3]
          \arrow["{\mathrm{lift}}", shift left=3, tail, two heads, from=1-3, to=1-1]
  \end{tikzcd}
  \end{equation}
\begin{enumerate}
\item As we  update  $\E \leftarrow \E', \ZZ \leftarrow \ZZ'$, we will say that the \emph{update is equal} provided that  $\widehat \Pi^\E_{\ZZ, \P} = \widehat \Pi^{\E'}_{\ZZ', \P}$.
\item As we  update  $\E \leftarrow \E', \ZZ \leftarrow \ZZ', \P \leftarrow \P'$, we will say that the \emph{update is equivalent} provided that  \eqref{eq:perturbation-correspondance} holds.
\end{enumerate}
Clearly an equal update is equivalent.

For triple $(\E', \ZZ', \P')$ that projects, we can derive a contradiction by showing that 
$$
\widehat \Pi^{\E'\cap \frac{1}{mq}\Z^2}_{\ZZ'\cap \frac{1}{mq}\Z^2, \P'\cap \frac{1}{mq}\Z^2} \neq \{0\}
\ \ 
\Rightarrow
\ \ 
\widehat \Pi^{\E(\pi, \P_q) \cap \frac{1}{mq}\Z^2}_{\frac1q\Z^2, \P_q\cap \frac{1}{mq}\Z^2} \neq \{0\}
\ \ 
\Rightarrow
\ \ 
\widehat \Pi^{\E(\pi) \cap \frac{1}{mq}\Z^2}_{\emptyset, \frac{1}{mq}\Z^2} \neq \{0\}.
$$
Throughout we assume  $\bar \pi \in \widehat \Pi^{\E(\pi)} \setminus \{0\}$ as our representative that $\pi$ is not a facet.

Before the mathematical description of each step, we will provide a pictorial description and a summary to explain the concepts of that step.
In the diagrams, we use green to denote faces of $\P_q$ in 7-tuple additivity, red to denote when they must take the value zero in the perturbation space, and a crosshatch on triangles when we remove it from the domain of the function via a projection.  The diagrams do not represent $\ve \sigma$ and $\t$ since these are not key elements in the arguments (although they are important in definitions of valid 7-tuples).
\begin{center}
    \begin{tikzpicture}[scale=1.2*\myscale]
        \def\spacing{9.5}
        
        \def\ta{0}
        \def\ea{\ta + \spacing}
        \def\tzez{\ea + \spacing}
        \def\tr{\tzez + \spacing}
        
        \mypoint{\additive}{\ta+1}{1};
        \mytriangle{\additive}{0.5}{\ta}{0};
        
        \node[right] at (\ta + 1.2, 0.25) {Additive};

        \myhorEdge{\edgeadditive}{\ea}{0};
        \node[right] at (\ea + 1.2, 0.25) {Edge Additive};

        \mytriangle{\zero}{0.5}{\tzez}{0};
        \myhorEdge{\zero}{\tzez}{1};
        \node[right] at (\tzez + 1.2, 0.25) {Zero};

        \mytriangleRemove{\tr}{0};
        \node[right] at (\tr+1.2, 0.25) {Project};
    \end{tikzpicture}
\end{center}

\parbox{\textwidth}{ \emph{Step 1: Remove $\E_{\point,\point,\point}$.}}
\begin{tabular}{|m{6cm}|m{6cm}|m{4.5cm}|}
\hline
Diagram & Summary & Update\\
\hline
\vspace{0.5cm}
\begin{center}
(no diagram needed for this step)
\end{center}
&  \parbox{6cm}{ We deduced already that $\bpi|_{\tfrac{1}{q} \Z} =0$.  Thus, $\bpi$ is trivially additive over any tuple with points from the grid.}
& \makecell{$\E \leftarrow \E \setminus \E_{\point,\point,\point}$\\[4pt]
(Equivalent update)}\\
\hline
\multicolumn{3}{|c|}{} \\ %
\multicolumn{3}{|c|}{Now we assume that $\E_{\point, \point, \point} = \emptyset$.}\\[-1ex] %
\multicolumn{3}{|c|}{\rule{\linewidth}{0.4pt}} \\ %
\hline
\end{tabular}
\begin{lemma}[Step 1 equal]
 The update $\E \leftarrow \E':=\E \setminus \E_{\point,\point,\point}$, $\ZZ \leftarrow \ZZ$ is an equal update.
\end{lemma}
\begin{proof}
Because $\E' \subseteq \E$, $\widehat \Pi^{\E}_{\ZZ,\P}\subseteq \widehat \Pi^{\E'}_{\ZZ,\P}$.
 
Alternatively, for any $\tilde \pi \in \widehat \Pi^{\E'}_{\ZZ,\P}$, $\tilde \pi$ trivially is additive over any $\tau \in \E_{\point, \point, \point}$, and thus $\tilde \pi \in \widehat \Pi^{\E}_{\ZZ, \P}$.
\end{proof}

\parbox{\textwidth}{ \emph{Step 2: Remove full dimensional additivities.}}
\begin{tabular}{|m{6cm}|m{6cm}|m{4.5cm}|}
\hline
\begin{center}
\stepTwo 
\end{center}
& \vspace{-5cm} \parbox{6cm}{\raggedright  \emph{Step 2a.  Remove $\E_{\tri,\tri,\tri}$:}
Use the generalized interval lemma to show that perturbations vanish on these triangles.\\
\emph{Step 2b. Remove $\E_{\EquiParam*, \EquiParam*, \EquiParam*}$ for $\EquiParam* \in \{\ver, \hor, \diag\}$:} Apply similar logic to show that perturbations vanish on such edges.}
& 
\makecell[l]{Set\\ $\E'' = \bigcup_{\EquiParam* \in \{\ver, \hor, \diag, \tri\} }\E_{\EquiParam*,\EquiParam*,\EquiParam*}$\\ \\ 
$\E \leftarrow \E \setminus \E''$\\[4pt]
$\ZZ\leftarrow \ZZ \cup p(\E'')$}\\
\hline
\multicolumn{3}{|c|}{} \\ %
\multicolumn{3}{|c|}{Now we assume that $\E_{\EquiParam*,\EquiParam*,\EquiParam*} = \emptyset$ for $\EquiParam* \in \{\ver, \hor, \diag, \tri\}$.} \\[-1ex] %
\multicolumn{3}{|c|}{\rule{\linewidth}{0.4pt}} \\ %
\hline
\end{tabular}\\
\begin{lemma}[Step 2a equal]
 The update $\E \leftarrow \E':=\E \setminus \E_{\tri,\tri,\tri}$, $\ZZ \leftarrow \ZZ':=\ZZ \cup p(\E_{\tri,\tri,\tri})$ is equal.
\end{lemma}
\begin{proof}
Suppose that $\bpi=0$ on $I, J, K$, we know that $\bpi$ satisfies trivially the equation by $(I, J, K, \ve \sigma, \0)$. Therefore, $\widehat\Pi^{\Etuple'}_{\mathcal{Z}',\P}\subseteq\widehat\Pi^\Etuple_{\mathcal{Z},\P}$.
  On the other hand, suppose that $\bpi\in \widehat\Pi^\Etuple_{\mathcal{Z},\P}$. By \autoref{lem:projection_interval_lemma_fulldim}, $\bpi$ are affine in the interiors of $I,J,K$. Because $\bpi$ is continuous, $\bpi$ is affine in $I,J,K$. As $\verts(I)\cup \verts(J) \cup \verts(K)\subseteq \P\cap\frac1q\Z^2\subseteq\mathcal{Z}$, $\bpi=0$ in $I, J, K$. Therefore, $\widehat\Pi^\Etuple_{\mathcal{Z},\P}
  \subseteq
  \widehat\Pi^{\Etuple'}_{\mathcal{Z}',\P}$.
  Hence, $\widehat\Pi^\Etuple_{\mathcal{Z},\P}
  =
  \widehat\Pi^{\Etuple'}_{\mathcal{Z}',\P}$.
\end{proof}

\begin{lemma}[Step 2b equal]
 The update $\E \leftarrow \E':=\E \setminus \bigcup_{\EquiParam* \in \{\ver, \hor, \diag\} }\E_{\EquiParam*,\EquiParam*,\EquiParam*}$, $\ZZ \leftarrow \ZZ':=\ZZ \cup p(\bigcup_{\EquiParam* \in \{\ver, \hor, \diag\} }\E_{\EquiParam*,\EquiParam*,\EquiParam*})$ is equal.
 \end{lemma}
\begin{proof}
 Clearly if $\bpi=0$ in $I_i \text{ for } i=1,2,3$, then $\bpi$ is additive over $(I_1, I_2, I_3,\ve\sigma, \t)$.    On the other hand, if $\bpi$ is additive over $(I_1, I_2, I_3,\ve\sigma, \t)$, then by \autoref{lem:projection_interval_lemma-corollary}, $\bpi$ is affine on $I_1, I_2, I_3$.  Since $\bpi$ is zero on the vertices of $I_1, I_2, I_3$, we have that $\bpi$ is zero on all $I_1, I_2, I_3$.
  Therefore, $\widehat\Pi^\Etuple_{\mathcal{Z},\P}
  =
  \widehat\Pi^{\Etuple'}_{\mathcal{Z}',\P}$.
\end{proof}

\parbox{\textwidth}{ \emph{Step 3: Remove $\E_{\ver\hor\diag,\tri,\tri}$}.}

\begin{tabular}{|m{6cm}|m{6cm}|m{4.5cm}|}
\hline
\begin{tikzpicture}
\node (pic0){
\begin{tikzpicture}[scale=\myscale]
\drawGridred
    \mytriangle{\additive}{0.4}{1}{1} %
    \myhorEdge{\edgeadditive}{2}{1} 
    \mytriangleReflect{\additive}{0.4}{4}{3} %
\end{tikzpicture}
\tikzLRarrow
};
    \node[right=0.1 of pic0] (pic1) {
            \begin{tikzpicture}[scale=\myscale]
            \drawGridred
            \myvertEdge{\edgeadditive}{4}{2}
            \myhorEdge{\edgeadditive}{2}{1}

            \mytriangleReflect{\additive}{0.4}{4}{3}
        \end{tikzpicture}
    };

    \node[above=-0.2cm of pic1] (pic2) {
        \begin{tikzpicture}[scale=\myscale]
            \drawGridred
            \mytriangle{\additive}{0.4}{1}{1} %
            \myhorEdge{\edgeadditive}{2}{1}
            \myvertEdge{\edgeadditive}{1}{1}

        \end{tikzpicture}
    };
    \node[below=-0.2cm of pic1] (pic3) {
        \begin{tikzpicture}[scale=\myscale]
            \drawGridred
            \myvertEdge{\edgeadditive}{1}{1}
            \myvertEdge{\edgeadditive}{4}{2}
            \mypoint{\edgeadditive}{0}{0}

        \end{tikzpicture}
    };
    \node[below=-0.2cm of pic3] (pic4) {
        \begin{tikzpicture}[scale=\myscale]
            \drawGridred
            \myhorEdge{\zero}{2}{1}
            \myhorEdge{\zero}{3}{4}
            \myhorEdge{\zero}{1}{1}

        \end{tikzpicture}
    };
\end{tikzpicture}
&
 Consider some $\tau^0 \in \E_{\ver\hor\diag,\tri,\tri}$. We use  \autoref{cor:two_triangle_one_edge} to deduce affine properties in the direction of the edge.  This allows us to set some edges to 0.
We then recover an equivalent update by adding appropriate 7-tuples.

We iteratively apply this logic.
&
\parbox{4.5cm}{For $\tau^0 \in \E_{\EquiParam{*}, \tri,\tri}$
with $* \in \{\ver,\hor,\diag\}$
find appropriate $\tau^1, \tau^2, \tau^3$ that each use at most 1 triangle.\\ \\
$\E \leftarrow (\E \setminus \{\tau^0\}) \cup \{\tau^1, \tau^2, \tau^3\}$\\[4pt]
$\ZZ \leftarrow \mathcal Z \cup \{\EquiParam{*} I_1, \EquiParam{*} J, \EquiParam{*} K_1\}$}\\
\hline
\multicolumn{3}{|c|}{} \\ %
\multicolumn{3}{|c|}{Now we assume $\E_{\ver\hor\diag,\tri,\tri} = \emptyset$.} \\[-1ex] %
\multicolumn{3}{|c|}{\rule{\linewidth}{0.4pt}} \\ %
\hline
\end{tabular}

Note, at this point, for any $\tau \in \E$, we have $\t = \0$. 
\begin{lemma}[Step 3 equal]
\label{lem:step3step}
  \label{lemTriTriEdge}
  For each $\tau^0=(\tri I, \EquiParam{*}J, \tri K, \ve \sigma, \0)\in \Etuple$ with $I, K \in \Itri$, $J \in \Iedge$, $\ve\sigma\in\{(1,1,-1),(1,-1,1)\}$, 
  there exist 7-tuples $\{\tau^1, \tau^2, \tau^3\}$ where 
  $\tau^1, \tau^2$ are of type (6)  and $\tau^3$ is of type (2) such that
   the update $\E \leftarrow \E':=(\E \setminus \{\tau^0\}) \cup \{\tau^1, \tau^2, \tau^3\}$, $\ZZ \leftarrow \ZZ':=\mathcal Z \cup \{\EquiParam{*} I_1, \EquiParam{*} J, \EquiParam{*} K_1\}$ is equal.
  \end{lemma}
\begin{proof}
We will prove the lemma only for $J \in \P_{q, \hor}$ as the other cases are similar.

\textbf{[Define subfaces of $I$ and $K$].}
Let $I_1, K_1 \in \P_{q, \hor}$, $I_2,  K_2 \in \P_{q, \ver}$, $I_3,  K_3 \in \P_{q, \diag}$ such that $I_1, I_2, I_3 \subseteq I$, and $K_1, K_2, K_3 \subseteq K$.

\textbf{[Deduce 0's from interval lemma].} 
By \autoref{cor:two_triangle_one_edge}, if $\bpi$ is additive on  $(\tri I, \hor J, \tri K, \ve \sigma, \0)$, then $\bpi$ is affine over $I,J,K$ in the $\hor$-direction. By continuity and $\bpi|_{\verts(\P_q)}\equiv0$, we know that $\bpi(x,y) = \bpi(a,y)$ for all $(x,y)\in I$, $\bpi(x,y) = \bpi(b,y)$ for all $(x,y)\in K$, and $\bpi|_{I_1\cup J\cup K_1}\equiv 0$. 
Equally, if $\bpi|_{I_1\cup J\cup K_1}\equiv 0$, then  $\bpi$ is additive on  $(\tri I, \hor J, \tri K, \ve \sigma, \0)$.
Thus, since $\tau^0 \in \E$, the update $\ZZ \leftarrow \ZZ\cup\{I_1, J, K_1\}$ is equal.

\textbf{[Show equivalent tuples].} 
Next, assuming the updated $\ZZ$, pick a vertex $\u$ of $J$ and the unique vectors $\t^i$ according to \autoref{lem:F-formulas}, to make these valid 7-tuples: 
  \begin{equation*}
  \begin{array}{rlllll}
    \tau^1 &:=
    (\hor J,& \ver I_2 ,& \tri I,& \ve \sigma,& \t^1),\\
\tau^2&:=(\hor J,& \ver K_2,& \tri K,& \ve \sigma,& \t^2),\\
\tau^3&:=(\ver I_2,& \point\{\u\},& \ver K_2,& \ve \sigma,& \0).
\end{array}
  \end{equation*}
We want to show that ${\widehat\Pi^\Etuple_{\mathcal{Z}',\P}}= {\widehat\Pi^{\Etuple'}_{\mathcal{Z}',\P}}$.

First, suppose $\bpi \in \widehat\Pi^\Etuple_{\mathcal{Z}',\P}$.  
Because $\bpi = 0$ on $I_1, J, K_1$, $\bpi$ is invariant over changes in the $\hor$-direction.  Hence, $\tau^1$ and $\tau^2$ are satisfied. By \autoref{remContinuity}, since $\bpi$ is additive on $\tau$, then it is also additive on $\tau^3$.
Thus, $\bpi \in \widehat\Pi^{\Etuple'}_{\mathcal{Z}',\P}$.

Next, suppose $\bpi \in \widehat\Pi^{\Etuple'}_{\mathcal{Z}',\P}$.  We just need to show that $\bpi$ is additive on $\tau$. Suppose that $(\x^1, \x^2, \x^3) \in F(\tau)$. Because $\bpi$ is additive on $\tau^3$, we have $\sigma_1 \bpi(\tilde{\x}^1) + \sigma_3 \bpi(\tilde{\x}^3)=0$, where $\tilde{\x}^1$ is the projection of $\x^1$ onto $I_2$ and $\tilde{\x}^3$ is the projection of $\x^3$ onto $K_2$. Because $\bpi$ is additive on $\tau^1$ and $\bpi=0$ on $J$, we have $\sigma_2\bpi(\tilde{\x}^1) +\sigma_3\bpi(\x^1) = 0$. Similarly, because $\bpi$ is additive on $\tau^2$ and $\bpi=0$ on $J$, we have $\sigma_2\bpi(\tilde{\x}^3) +\sigma_3\bpi(\x^3) = 0$.
Thus, $\sigma_1\bpi(\x^1)+\sigma_2\bpi(\x^2)+\sigma_3\bpi(\x^3)=-\sigma_2\sigma_3(\sigma_1\bpi(\tilde{\x}^1)+\sigma_3\bpi(\tilde{\x}^3))=0$, which implies that $\bpi\in\widehat\Pi^{\Etuple}_{\mathcal{Z}',\P}$.

Therefore, the update $\E \leftarrow \E \setminus \{\tau^0\} \cup \{\tau^1, \tau^2, \tau^3\}$ is an equal update.

\end{proof}
\begin{corollary}[Updates applied iteratively]
    By iteratively applying \autoref{lem:step3step}, we create an equal update with $\E_{\ver\hor\diag, \tri,\tri} = \emptyset$.
\end{corollary}

\parbox{\textwidth}{ \emph{Step 4: Remove $\E_{\tri,\point, \tri}$ by changing the domain.}}
\begin{tabular}{|m{6cm}|m{6cm}|m{4.5cm}|}
\hline
\begin{tikzpicture}
\node (pic0){
\begin{tikzpicture}[scale=\myscale]
\drawGridred
    \mytriangle{\additive}{0.4}{1}{1} %
        \mytriangle{\additive}{0.4}{4}{2} %
    \mypoint{\edgeadditive}{3}{1} %
\end{tikzpicture}
\tikzarrow
};
\node [right = 0.1 of pic0](pic1a){
\begin{tikzpicture}[scale=\myscale]
\drawGridred
    \mytriangleRemove{1}{1}  %
        \mytriangle{\additive}{0.4}{4}{2} %
\end{tikzpicture}
};
\node [below=0.1 of pic0](pic0b){
\begin{tikzpicture}[scale=\myscale]
\drawGridred
    \mytriangle{\zero}{0.4}{1}{1} %
\end{tikzpicture}
\tikzLRarrow
};
\node [right=0.1 of pic0b](pic1b){
\begin{tikzpicture}[scale=\myscale]
\drawGridred
        \mytriangle{\zero}{0.4}{4}{2} %
\end{tikzpicture}
};
\node [below=0.1 of pic0b](pic0c){
\begin{tikzpicture}[scale=\myscale]
\drawGridred
    \mytriangle{\additive}{0.4}{1}{1} %
    \myvertEdge{\edgeadditive}{0}{1} %
    \myhorEdge{\edgeadditive}{1}{0} %
\end{tikzpicture}
\tikzLRarrow
};
\node [right=0.1 of pic0c](pic1c){
\begin{tikzpicture}[scale=\myscale]
\drawGridred
        \mytriangle{\additive}{0.4}{4}{2} %
    \myvertEdge{\edgeadditive}{0}{1} %
    \myhorEdge{\edgeadditive}{1}{0} %
\end{tikzpicture}
};
\node [below=0.1 of pic0c](pic0d){
\begin{tikzpicture}[scale=\myscale]
\drawGridred
    \mytriangle{\additive}{0.4}{1}{1} %
        \mytriangle{\additive}{0.4}{3}{0} %
    \mypoint{\edgeadditive}{2}{4} %
\end{tikzpicture}
\tikzLRarrow
};
\node [right = 0.1 of pic0d](pic1d){
\begin{tikzpicture}[scale=\myscale]
\drawGridred
    \mytriangle{\additive}{0.4}{4}{2}  %
        \mytriangle{\additive}{0.4}{3}{0} %
    \mypoint{\edgeadditive}{2}{4} %
\end{tikzpicture}
};
\end{tikzpicture}
&

Consider some $\tau  \in \E_{\tri,\point, \tri}$ where
$\tau = (\tri I, \point \{\u\}, \tri K, \ve \sigma, \t)$.  

We will remove $I$ from the domain.  To do so, we need to do the following:

(1) If $I \in \mathcal Z$, add $K$ to $\mathcal Z$.

(2) Remove $\tau$ from $\E$.

(3) For any $\tau'$ with $I$ in it, replace $I$ by $K$ to create a new 7-tuple $\tau''$.
(Note, we need to update $\ve\sigma', \t'$ via $\ve\sigma, \t$ and $\u$.).

(4) By storing $\tau$, we can retain a mapping for how to recover function values on $I$, and for now remove $I$ from the domain.

\begin{center} \textbf{[This step repeats until there are no more $\tau  \in \E_{\tri,\point, \tri}$.  The sequence as to choices of $\tau$ to apply this to is not important, even though it will result in different representations.]}\end{center}

& 
\parbox{4cm}{While $\E_{\tri,\point, \tri} \neq \emptyset$, let $\tau \in \E_{\tri,\point, \tri}$.\\

(1)  If $I \in \mathcal Z$, then\\ $\mathcal Z \leftarrow \mathcal Z \setminus \{ I\} \cup \{ K\}$.
\\
\\
(2)  $\mathcal E \leftarrow \mathcal E \setminus \{\tau\}$.\\
\\
(3)   For all $\tau' \in \E$ containing $I$, derive $\tau''$ and \\
$\E \leftarrow \E \setminus \{\tau'\}\cup \{\tau''\}$.
\\
\\
(4)  $\mathcal P \leftarrow \mathcal P \setminus \{\tri I\}$.}\\
\hline
\multicolumn{3}{|c|}{} \\ %
\multicolumn{3}{|c|}{Now we assume $\E_{\tri,\point, \tri} = \emptyset$.} \\[-1ex] %
\multicolumn{3}{|c|}{\rule{\linewidth}{0.4pt}} \\ %
\hline
\end{tabular}
\begin{lemma}
\label{reduceTranzero}
Suppose that $\tau = (I, \point \{\u\}, K, \ve \sigma, \t) \in \E$ where $\u \in \tfrac{1}{q} \Z^2$.  Then $\ZZ \leftarrow \ZZ \setminus \{I\} \cup \{K\}$ is an equal update.
\end{lemma}
\begin{proof}
    This equivalence is obvious by exchanging function values through the additivity over $\tau$.
\end{proof}

\begin{lemma}
\label{reduceTran-general}
Suppose that $\tau = (I_1, \point \{\v\}, \tilde I_1,\ve \sigma, \t) \in \mathcal E$ where $\v \in \tfrac{1}{q} \Z^2$.  
Suppose that $\tau'=(I_1, I_2, I_3, \ve \mu, \u)\in\mathcal E$ containing $I_1$. Then there exists a valid 7-tuple $\tau'' = (\tilde I_1, I_2, I_3, \ve{\tilde \mu}, \ve{ \tilde u})$ with $\ve{\tilde \mu}=(\tilde \mu_1,\mu_2,\mu_3)$, where $\tilde \mu_1 = -  \sigma_1 \sigma_3\mu_1$ and $\ve{\tilde u} = \u - \mu_1 \sigma_1\t + \mu_1 \sigma_1 \sigma_2 \v$, such that $\E \leftarrow \E \setminus \{\tau'\}\cup \{\tau''\}$ is an equal update.
\end{lemma}

\begin{proof}
\textbf{Additivity.} Since $\pi$ is additive over $(I_1, \point \{\v\}, \tilde I_1,\ve \sigma, \t)$, we know that
\begin{align*}
&&&0 = \sigma_1\pi(\x^1) + \sigma_2\pi(\v) + \sigma_3\pi(\ve {\tilde x}^1)   && 
\text{ for all } \x^1 \in I_1,\, \ve{\tilde x}^1 \in \tilde I_1, \,
\sigma_1 \x^1 + \sigma_2 \v  + \sigma_3 \ve{\tilde x}^1 = \t.\\
\Leftrightarrow  &&&0 = \mu_1\pi(\x^1) +  \mu_1 \sigma_1 \sigma_3\pi(\ve{\tilde x}^1 )  &&
\text{ for all } 
\x^1 \in I_1,\, \ve{\tilde x}^1 \in \tilde I_1,\,  
\mu_1 \x^1  = \mu_1 \sigma_1\t - \mu_1 \sigma_1 \sigma_2 \v  - \mu_1 \sigma_1 \sigma_3 \ve{\tilde x}^1 .\\
\Leftrightarrow &&& 0 = \mu_1\pi(\x^1) -  c_1\pi(\ve{\tilde x}^1)  &&  \text{ for all } \x^1 \in I_1,\, \ve{\tilde x}^1 \in \tilde I_1,\,
\mu_1 \x^1  = \u - \ve{\tilde u}  + c_1 \ve{\tilde x}^1 .
\intertext{
The second equation follows since $\pi(\v) = 0$ as that $\v\in \verts(\P_q)$ and we also multiple through by $\sigma_1 \mu_1$.  The third equation follows from rewriting based on definitions.  We will now show the equivalence between the two 7-tuples.  Suppose $\pi$ is additive over $(I_1, I_2, I_3, \ve \mu, \u)$.  That is,}
&&&0= \mu_1\pi( \x^1) + \mu_2\pi(\x^2) + \mu_3\pi( \x^3)  && \text{ for all }  \x^1 \in I_1,\,   \x^2 \in I_2,\, \x^3 \in I_3,\,  \mu_1  \x^1 + \mu_2   \x^2  + \mu_3  \x^3 = \u.\\
\Leftrightarrow &&& 0= \tilde \mu_1\pi(\ve{\tilde x}^1) + \mu_2\pi( \x^2) + \mu_3\pi( \x^3)
&& \text{ for all } \ve{\tilde x}^1 \in \tilde I_1,\,  \x^2 \in I_2,\,   \x^3 \in I_3,
 \u - \ve{\tilde u}  + \tilde \mu_1 \ve{\tilde x}^1  + \mu_2   \x^2  + \mu_3  \x^3 = \u,
\end{align*}
where the equivalence follows from applying the last equation above. This is equivalent to $\pi$ is additive over $(\tilde I_1, I_2, I_3, \ve{\tilde \mu}, \ve{\tilde u})$.

\textbf{Validity.}
Finally, we show that $(I_1, I_2, I_3, \ve \mu, \u)$
is a valid  7-tuple if and only if
$(\tilde I_1, I_2, I_3, \ve{\tilde \mu}, \ve{\tilde u}))$
is a valid  7-tuple.
We establish the correspondence via \autoref{validCharacterization}.  First note that since
$(I_1, \point \{\v\}, \tilde I_1,\ve \sigma, \t)$ is a valid 7-tuple, by \autoref{validPointCharacterization}, we have
  $ I_1 = \sigma_1 \t + \sigma_1 \sigma_2 \v - \sigma_1 \sigma_3 \tilde I_1$.
  
Suppose \autoref{validCharacterization}\eqref{item:first_containment} holds for $(I_1, I_2, I_3, \ve \mu, \u)$.
Multiplying by $\mu_1$, we have
$ \mu_1 I_1 \subseteq  \u - ( \mu_2 I_2 +  \mu_3 I_3)$.
Then substituting for $I_1$ we have
\begin{align*}
&& \mu_1 \sigma_1 \t + \mu_1 \sigma_1 \sigma_2 \v - \mu_1 \sigma_1 \sigma_3 \tilde I_1 &\subseteq  \u - ( \mu_2 I_2 +  \mu_3 I_3),\\
\Leftrightarrow && - \mu_1 \sigma_1 \sigma_3 \tilde I_1 &\subseteq  \u-\mu_1 \sigma_1 \t - \mu_1 \sigma_1 \sigma_2 \v - ( \mu_2 I_2 +  \mu_3 I_3),\\
\Leftrightarrow && \tilde \mu_1 \tilde I_1 &\subseteq  \ve{\tilde u} - ( \mu_2 I_2 +  \mu_3 I_3).
\end{align*}
Therefore,  \autoref{validCharacterization}\eqref{item:first_containment} holds for $(I_1, I_2, I_3, \ve \mu, \u)$ if and only if condition 1 of \autoref{validCharacterization}\eqref{item:first_containment} holds for $(\tilde I_1, I_2, I_3, \ve {\tilde \mu}, \ve{\tilde u})$.  Conditions 2 and 3 are similar.   This finishes the proof.
\end{proof}

\begin{lemma}
\label{reduceTran}
Suppose that $\tau = (I, \point \{\u\}, K,\ve \sigma, \t) \in \mathcal E$ where $\u \in \tfrac{1}{q} \Z^2$.  
Suppose that $\tau'=(I, J', K', \ve \sigma', \t')\in\mathcal E$ containing $I$. Then there exists a valid 7-tuple $\tau'' = (K, J', K', \ve \sigma'', \t'')$ with $\ve{\sigma}''=(\sigma''_1,\sigma'_2,\sigma'_3)$, where $\sigma''_1 = -  \sigma_1 \sigma_3\sigma'_1$ and $\t'' = \t' - \sigma'_1 \sigma_1\t + \sigma'_1 \sigma_1 \sigma_2 \u$, such that $\E \leftarrow \E \setminus \{\tau'\}\cup \{\tau''\}$ is an equal update.
\end{lemma}
\begin{proof}
    It follows directly from \autoref{reduceTran-general}.
\end{proof}

\begin{lemma}[Step 4 equivalence]%
  \label{removeTriCont}
Suppose $\tau=(\tri I, \point \{\u\}, \tri K, \ve\sigma, \t)$ is the only 7-tuple in $\Etuple$ where $I$ appears and where $\u \in \frac1q\ZZ^2$.  Then there exist 7-tuples $\{\tau^1, \tau^2, \tau^3\}$ of type (2) such that the update $\E \leftarrow \E':=\E \setminus \{\tau\}\cup\{\tau^1,\tau^2,\tau^3\}$, $\ZZ \leftarrow \ZZ$, $\mathcal P \leftarrow \mathcal P' := \mathcal P \setminus \{I\}$ is an equivalent update. Moreover,
\begin{align*}
\mathcal E''=\{\tau^1 = &(\hor I_1, \point \{\u\}, \hor K_1, \ve\sigma, \t),\\
\tau^2 = & (\ver I_2, \point \{\u\}, \ver K_2, \ve\sigma, \t),\\
\tau^3 = &(\diag I_3, \point \{\u\}, \diag K_3, \ve\sigma, \t)\},
\end{align*}
where $I_1 \in \P_{q,\hor}, I_2\in \P_{q,\ver}, I_3 \in \P_{q,\diag}$ are the edges of $I$ and  $K_1 \in \P_{q,\hor}, K_2\in \P_{q,\ver}, K_3 \in \P_{q,\diag}$ are the edges of $K$.
\end{lemma}

\begin{proof}

Since
$\tau=(\tri I, \point \{\u\}, \tri K, \ve\sigma, \t)$, by \autoref{validPointCharacterization} is a vaild 7-tuple, we have
$$
  I = \sigma_1 \t - (\sigma_1 \sigma_2 \u + \sigma_1 \sigma_3 K).
$$
Since this equation is just an affine transformation, we must also have
$$
  I_i = \sigma_1 \t - (\sigma_1 \sigma_2 \u + \sigma_1 \sigma_3 K_i).
$$
Then by \autoref{validPointCharacterization}, the three 7-tuples in $\Etuple'':=\{\tau^1,\tau^2,\tau^3\}$ are valid 7-tuples.

Define the projection $h\colon \widehat\Pi^\Etuple_{\mathcal{Z},\P}  \to \widehat \Pi^{\Etuple'}_{\mathcal{Z},\P\setminus I}$ such that $\theta \mapsto \theta|_{\bigcup(\P \setminus I)}$.

Suppose $\theta \in \widehat\Pi^\Etuple_{\mathcal{Z},\P}$ and $\theta\ne 0$.
Clearly $h(\theta)$ is continuous since $h$ is just the restriction to a closed set.
By \autoref{remContinuity}, since $\theta$ is additive on $\tau$, then it is also additive on all of the 7-tuples in $\Etuple''$.  Therefore,  $\theta \in \widehat\Pi^{\Etuple\cup \Etuple''}_{\mathcal{Z},\P} \subseteq \widehat\Pi^{\Etuple'}_{\mathcal{Z},\P}$.   This inclusion holds since $\Etuple' \subseteq \Etuple \cup \Etuple''$. Since $\Etuple'$ has no 7-tuples containing $I$, the restriction $\theta|_{\bigcup\P\setminus I}$ must also be additive on all 7-tuples in $\Etuple'$.
Thus $h(\theta) \in \widehat\Pi^{\Etuple'}_{\mathcal{Z},\P\setminus I}$.
Also $h(\theta)\ne 0$ because otherwise $\theta = 0$ on $\P\setminus I$ and then by $\tau\in \E$, we have $\theta = 0$ on $I$ as well.

Now define the lift $h^{-1}\colon \bar\Pi^{\Etuple'}_{\mathcal{Z},\P\setminus I} \to \bar\Pi^\Etuple_{\mathcal{Z},\P}$ by
$
\hat \theta \mapsto \theta
$
where
$$
\theta(\x) = \begin{cases}
\hat \theta(\x) & \text{ for } \x \in \bigcup (\P\setminus I) = \bigcup\P\setminus \intr(I),\\
-\sigma_3 \sigma_1 \hat \theta(\y) & \text{ for } \x \in \intr(I), \text{with } \sigma_1 \x + \sigma_2 \u + \sigma_3 \y = \t.
\end{cases}
$$
Above, whenever $\x \in \intr(I)$, we have $\y \in \intr(K)$.
Hence, the values of $\theta$ on $\intr(I)$ are determined uniquely by the additive relation $(\tri I, \point \{\u\}, \tri K, \ve\sigma, \t)$.

Now consider any $\hat \theta \in \widehat\Pi^{\Etuple'}_{\mathcal{Z},\P\setminus I}$, $\hat\theta\ne 0$ and let $\theta = h^{-1}(\hat \theta)$. By construction of $h^{-1}$, we see that $\theta\ne 0$ and $\theta$ is additive on $\tau$ in the interior. Since $\hat \theta$, and hence $\theta$, is additive on the 7-tuples in $\Etuple''$, $\theta$ is also additive on the closed set, that is
$
\sum_{i=1}^3\sigma_i \theta(\x^i) = 0, \ \forall \x^1 \in I, \x^2 =\u, \x^3\in K, \sum_{i=1}^3 \sigma_i \x^3 = \t.
$
Since $\u \in \verts(\P_q)$ is only one point, this equation shows a affine relationship between the values $\theta$ on $K$ and on $I$.  Since $\theta$ is continuous on the closed set $K$, then it is also continuous on the closed set $I$. Thus, $\theta = h^{-1}(\hat \theta)\in \widehat\Pi^{\Etuple}_{\mathcal{Z},\P}$.

Therefore, the update is equivalent.
\end{proof}

\begin{corollary}[Updates applied iteratively]
    By iteratively applying \autoref{removeTriCont}, we create an equivalent update with $\E_{\tri,\point, \tri} = \emptyset$.
\end{corollary}
\begin{proof}
For a tuple $\tau=(\tri I, \point \{\u\}, \tri K, \ve\sigma, \t)$, we will apply \autoref{reduceTran} to all the tuples in $\E_{\tri, \point, \tri}$ and $\E_{\edge,\edge,\tri}$ containing $I$. Thus, after some equivalent updates, $\tau$ will be the only tuple containing $I$, and the projection $p(\E)$ will not change. If $I\in\mathcal Z$, we will apply \autoref{reduceTranzero} to replace $I$ with $K$ in $\mathcal Z$.
Then, we can apply \autoref{removeTriCont} to remove $I$ from the domain $\P$. By iteratively applying these procedures, we will remove all the tuples in $\E_{\tri,\point, \tri}$.
\end{proof}

\parbox{\textwidth}{ \emph{Step 5: Remove $\E_{\ver\hor\diag,\ver\hor\diag, \tri}$ by exchanging 7-tuples, inferring zeros, and changing the domain.}}
\begin{tabular}{|m{6cm}|m{6cm}|m{4.5cm}|}
\hline
\begin{tikzpicture}
\node (pic0){
\begin{tikzpicture}[scale=\myscale]
\drawGridred
    \mytriangle{\additive}{0.4}{3}{4} %
    \myhorEdge{\edgeadditive}{2}{1}
    \mydiagEdge{\edgeadditive}{1}{4}
\end{tikzpicture}
\tikzLRarrow
};
    \node[right=0.1 of pic0] (pic1) {
            \begin{tikzpicture}[scale=\myscale]
            \drawGridred
   
    \mydiagEdge{\edgeadditive}{3}{5}
    \mydiagEdge{\edgeadditive}{1}{4}
    \mypoint{\edgeadditive}{0}{0} %
        \end{tikzpicture}
    };

    \node[above=-0.2cm of pic1] (pic2) {
        \begin{tikzpicture}[scale=\myscale]
            \drawGridred
    \mytriangle{\additive}{0.4}{3}{4} %
    \myhorEdge{\edgeadditive}{3}{4}
    \mydiagEdge{\edgeadditive}{3}{5}

        \end{tikzpicture}
    };
    \node[below=-0.2cm of pic1] (pic3) {
        \begin{tikzpicture}[scale=\myscale]
            \drawGridred
    \mypoint{\edgeadditive}{0}{0} %
    \myhorEdge{\edgeadditive}{3}{4}
    \myhorEdge{\edgeadditive}{2}{1}

        \end{tikzpicture}
    };
\end{tikzpicture}
&
\emph{Step 5a. Canonical representation}
Consider  $\E_{\ver\hor\diag,\ver\hor\diag, \tri}$.
Suppose $\tau = (\diag I, \hor J, \tri K, \ve \sigma, \t)$. 
We now rewrite this by using edges of $K$.  For this, we need extra 7-tuples with edges and a point to relate values on corresponding edges.  These follow from using continuity arguments, and we can use the notation to absorb the point into the $\t$ argument and assume the point is $\ve 0$.
\begin{center} \textbf{[Use of continuity]}\end{center}
&
$\E \leftarrow \E \setminus \{\tau\} \cup \{\tau^1, \tau^2, \tau^3\}$
\\
\hline
 \begin{tikzpicture}
\node (pic0){
\begin{tikzpicture}[scale=\myscale]
\drawGridred
    \mytriangle{\additive}{0.4}{3}{4} %
    \myhorEdge{\edgeadditive}{3}{4}
    \mydiagEdge{\edgeadditive}{3}{5}
\end{tikzpicture}
};
    \node[below=0.1 of pic0] (pic0a) {
            \begin{tikzpicture}[scale=\myscale]
            \drawGridred
    \mytriangle{\additive}{0.4}{3}{4} %
    \myhorEdge{\edgeadditive}{3}{4}
    \myvertEdge{\edgeadditive}{3}{4} 
        \end{tikzpicture}
    };
        \node[right=1.3cm of $(pic0.north)!0.5!(pic0a.south)$] (arrow) {\Large$\Rightarrow$};

    \node[right=1.3cm of pic0] (pic1) {
        \begin{tikzpicture}[scale=\myscale]
            \drawGridred

             \mytriangle{\additive}{0.4}{3}{4} %
    \myhorEdge{\edgeadditive}{3}{4}
    \mydiagEdge{\edgeadditive}{3}{5}

        \end{tikzpicture}
    };
        \node[below=0.1cm of pic1] (pic2) {
        \begin{tikzpicture}[scale=\myscale]
            \drawGridred
    \myhorEdge{\zero}{3}{4}

        \end{tikzpicture}
    };

\end{tikzpicture}
 & \emph{Step 5b(i).} Two canonical tuples with two edges of the same triangle.  We can use hidden interval lemma~\autoref{lem:hidden_interval_lemma} to deduce an affine property.  This allows us to set the edge in common to 0.  This zero together with an additivity implies the other additivity.

 Note that if all three canonical tuples exist, we can apply this step twice.  This actually induces the function to be zero on the entire triangle.
 &
 \parbox{4.5cm}{$\tau^1, \tau^2$ with common $\tri K$ triangle and common edge $\ver\hor\diag I$.\\
 \\
 $\E \leftarrow \E \setminus \{\tau^2\}$\\
 $\ZZ \leftarrow \ZZ\cup \{\ver\hor\diag I\}$}
 \\
\hline
\begin{tikzpicture}
\node (pic0){
\begin{tikzpicture}[scale=\myscale]
\drawGridred
    \mytriangle{\additive}{0.4}{3}{4} %
    \myhorEdge{\edgeadditive}{3}{4}
    \mydiagEdge{\edgeadditive}{3}{5}
\end{tikzpicture}
\tikzLRarrow
};
    \node[right=0.1 of pic0] (pic1) {
            \begin{tikzpicture}[scale=\myscale]
            \drawGridred
   
    \mytriangleRemove{3}{4} %
    \myhorEdge{\edgeadditive}{3}{4}
    \mydiagEdge{\edgeadditive}{3}{5}
        \end{tikzpicture}
    };
\end{tikzpicture}
 & \emph{Step 5b(ii).} Exactly 1 tuple of this type
 In this case, we can just project out the triangle and record the formula update as a combination of function values from the edges.
 &
 $\mathcal P \leftarrow \mathcal P \setminus \{\tri K\}$\\
 \hline
\multicolumn{3}{|c|}{} \\ %
\multicolumn{3}{|c|}{Now we assume $ \E_{\ver\hor\diag,\ver\hor\diag, \tri} = \emptyset$.}\\
\multicolumn{3}{|c|}{$\Rightarrow$ Thus, $\E = \E_{\point, \ver\hor\diag, \ver\hor\diag} \cup \E_{\ver\hor\diag, \ver\hor\diag, \ver\hor\diag}$ (Type (2) and (3))} \\[-1ex] %
\multicolumn{3}{|c|}{\rule{\linewidth}{0.4pt}} \\ %
\hline
\end{tabular}

There are three types of separability relations.  After permutation, the 7-tuples that are separability relations look like
$$
 \text{(1) $ (\ver I, \hor J, \tri K, \ve\sigma, \t)$,} \hspace{1cm}
\text{(2) $ (\hor I, \diag J, \tri K, \ve\sigma, \t)$,} \hspace{1cm}
\text{(3) $ (\diag I, \ver J, \tri K, \ve\sigma, \t)$.}
$$
We next show that every separability relation including two edges $I,J$ and a triangle $K$ can be exchanged for  separability relation including two edges of $K$ and the triangle $K$, along with relations between the edges of $K$ and the edges $I,J$.

\begin{lemma}[Step 5a equal]
\label{lemCanonical}
  For each $\tau=(\EquiParam{1} I, \EquiParam{2}J, \tri K, \ve \sigma, \t)\in \Etuple$ with $K \in \Itri$, and $I\in \P_{q,\EquiParam{1} }$, $J \in \P_{q,\EquiParam{2}}$ for distinct $\EquiParam{1}, \EquiParam{2}\in \{\ver, \hor, \diag\}$, 
  there exists a set of 7-tuples
$$
\begin{array}{r@{\hskip -.5pt}l@{\hskip -.4pt}rrrl}
\mathcal E''=\{\tau^1=&(&\EquiParam{1} K_1, &\EquiParam{2} K_2, &\tri K, &\ve \mu, \w),\\
\tau^2=&(&\EquiParam{1} I, &\{\u'\}, &\EquiParam{1} K_1, &\ve\sigma, \t),\\
\tau^3=&(&\{\u\}, &\EquiParam{2} J, &\EquiParam{2} K_2, &\ve\sigma, \t)\}
\end{array}
$$
where $ K_1,  K_2 \in \Iedge$ are the edges of $K$ such that $K_i \in \P_{q,\EquiParam{i}}$ for $i=1,2$, $\u \in \verts(I)$, $\u' \in \verts(J)$ are chosen by \autoref{lemSepEdge} such that $\Etuple''$
is a set of valid 7-tuples,  $\ve \mu = (1,1,-1)$ and $\w = \sigma_3(\t - \sigma_2 \u' - \sigma_1 \u)$.
Then the update $\E \leftarrow \E':=(\E \setminus \{\tau\}) \cup \{\tau^1, \tau^2, \tau^3\}$ is equal.
\end{lemma}

\begin{proof}
Consider any $\bpi \in \widehat\Pi^\Etuple_{\mathcal{Z},\P}$.  Since $\bpi$ is additive over $(I, J, K, \ve\sigma, \t)$, 
by \autoref{remContinuity}, we also have that $\bpi$ is additive over $(I, \{\u'\}, K_1, \ve\sigma, \t)$ and $(\{\u\}, J, K_2, \ve\sigma, \t)$.

Therefore, since $(I, \{\u'\}, K_1, \ve\sigma, \t), (\{\u\}, J, K_2, \ve\sigma, \t) \in \Etuple'$, finishing the proof is equivalent to showing the following claim:

\textbf{Claim:} Let $\bpi \colon \bigcup\P \to \R$ with $\bpi_{\P_q} \equiv 0$ such that $\bpi$ is additive over $(I, \{\u'\}, K_1, \ve\sigma, \t)$ and $(\{\u\}, J, K_2, \ve\sigma, \t)$.

Then $\bpi$ is additive over $(I, J, K, \ve\sigma, \t)$ if and only if $\bpi$ is additive over $(K_1, K_2, K, (1,1,-1), \w)$.

\textbf{Claim proof:}\\
Since $\bpi$ is additive over $(I, \{\u'\}, K_1, \ve\sigma, \t)$ and $(\{\u\}, J, K_2, \ve\sigma, \t)$, we have the following two equations:
\begin{align*}
\sigma_1 \bpi(\x) + \sigma_2 \bpi(\u') + \sigma_3 \bpi(\bar \x) &= 0 & \forall \x \in I, \bar \x \in K_1, \sigma_1 \x + \sigma_2 \u' + \sigma_3 \bar \x = \t,\\
  \sigma_1 \bpi(\u) + \sigma_2 \bpi(\y) + \sigma_3 \bpi(\bar \y) &= 0 & \forall \x \in I, \bar \y \in K_2, \sigma_1 \u + \sigma_2 \y + \sigma_3 \bar \y = \t.
\end{align*}
Since $\bpi|_{\frac{1}{q} \Z^2} \equiv 0$, these become the following
\begin{align}
\label{eq1}
\sigma_1 \bpi(\x)  &= -  \sigma_3 \bpi(\bar \x)  & \forall \x \in I, \bar \x \in K_1, \sigma_1 \x + \sigma_2 \u' + \sigma_3 \bar \x = \t,\\
\label{eq2}
\sigma_2 \bpi(\y) &= -  \sigma_3 \bpi(\bar \y) & \forall \x \in I, \bar \y \in K_2, \sigma_1 \u + \sigma_2 \y + \sigma_3 \bar\y = \t.
\end{align}
Now suppose $\bpi$ is additive over $(I, J, K, \ve\sigma, \t)$.  This is equivalent to
\begin{align*}
&& \sigma_1 \bpi(\x) + \sigma_2 \bpi(\y) + \sigma_3 \bpi(\z) = 0 &&&  \forall \x \in I,\y \in J, \z \in K, \sigma_1 \x + \sigma_2 \y + \sigma_3 \z = \t,\\
\Leftrightarrow && -\sigma_3 \bpi(\bar \x)  -\sigma_3 \bpi(\bar \y) + \sigma_3 \bpi(\z) = 0 &&&  \forall \bar \x \in K_1,\bar \y \in K_2, \z \in K,\\ 
&& &&&\hspace{1cm}  (\t - \sigma_2 \u' - \sigma_3 \bar \x ) + (\t - \sigma_1 \u - \sigma_3 \bar \y) + \sigma_3 \z = \t,\\
\Leftrightarrow &&  \bpi(\bar \x) +  \bpi(\bar \y) - \bpi(\z) = 0 &&& \forall \bar \x \in K_1,\bar \y \in K_2, \z \in K,
  \bar \x + \bar \y -  \z =  \w,
\end{align*}
The first equivalence, we use \eqref{eq1} to substitute for $\bpi(\x)$ and  \eqref{eq2} to substitute for $\bpi(\y)$.  The second is just a rewriting
  where $\w = \sigma_3(\t - \sigma_2 \u' - \sigma_1 \u)$.
By definition, this means $\bpi$ is additive over $(K_1, K_2, K, (1,1,-1), \w)$.
Therefore, the claim is proved and we are done.
\end{proof}

We now prove a more complicated lemma on just having two separability relations.

\begin{lemma}[Step 5b(i) equal]
\label{lemTwoSep}
  Suppose that we have a set of valid 7-tuples
  $$\{\tau^1=(\EquiParam{1} K_1, \EquiParam{2}K_2, \tri K, \ve \mu, \t^1), \tau^2=(\EquiParam{1} K_1, \EquiParam{3}K_3, \tri K, \ve \mu, \t^2)\} \subseteq \Etuple$$
  with $\ve \mu=(1,1,-1)$, $K \in \Itri$, and $ K_1,  K_2,  K_3\in \Iedge$ be the edges of $K$ such that $K_i \in \P_{q,\EquiParam{i}}$ for $i=1,2,3$.
  Then the update $\E \leftarrow \E':=(\E \setminus \{\tau^2\})$, $\ZZ \leftarrow \ZZ':=\mathcal Z \cup \{\EquiParam{1} K_1\}$ is equal.
\end{lemma}

\begin{proof}
We will prove this for the case where $K_1 \in \P_{q,\hor}$, $K_2 \in \P_{q,\ver}$, $K_3 \in \P_{q,\diag}$.  The other cases are similar so we will omit their proofs.

First, suppose $\bpi \in \widehat\Pi^\Etuple_{\mathcal{Z},\P}$.  By \autoref{lemILedges}, $\bpi$ is affine on $K_1$ and is $\hor$-affine on $K$.   Since $\bpi$ is zero on $\verts(\P_q)$, this implies that $\bpi$ is zero on all of $K_1$. Thus, $\bpi \in \widehat\Pi^{\Etuple'}_{\mathcal{Z}',\P}$.

Next, suppose $\bpi \in \widehat\Pi^{\Etuple'}_{\mathcal{Z}',\P}$.  We just need to show that $\bpi$ is additive on $\tau^2$.

By \autoref{lem:F-formula-one-tri}, we can assume that
\begin{equation*}
\begin{array}{r@{\hskip -.5pt}rrlr@{\hskip -.25pt}l@{\hskip -.25pt}l@{\hskip -.5pt}rrr@{\hskip -.5pt}r}
F(&\hor K_1,& \ver K_2,& \tri K,& \ve \mu, \t^1) 
&=& \{ (\t^1, \t^1, \t^1) + \frac{1}{q}(&\lambda_1 \ve d^1,&\lambda_2 \ve d^2,&(\lambda_1 \ve d^1+\lambda_2 \ve d^2)) 
 &: \ve \lambda \in \Delta_2\},\\
 F(&\hor K_1,& \diag K_3,& \tri K,& \ve \mu, \t^2) 
&=& \{ (\t^2, \t^2, \t^2) + \frac{1}{q}(&-\lambda_1 \ve d^1,&\lambda_2 (\ve d^2 - \ve d^1),&(-\lambda_1\ve d^1+\lambda_2 (\ve d^2-\ve d^1))) 
 &: \ve \lambda \in \Delta_2\},
\end{array}
\end{equation*}
where $\d^1=\sigma\e^1, \d^2=\sigma\e^2, \sigma\in\{\pm1\}$. Then $\t^2 = \t^1+\frac1q\d^1$, $K_3 = \{\t^1 + \frac1q\d^1 + \frac1q\lambda(\d^2-\d^1):\lambda\in[0,1]\}$.

Since $\bpi$ is additive over $(\hor K_1, \ver K_2, \tri K, \ve\mu, \t^1),$ we have
\begin{align*}
&&\bpi(\t^1 + \tfrac1q\lambda_1\d^1) + &\bpi(\t^1 + \tfrac1q\lambda_2\d^2) - \bpi(\t^1 + \tfrac1q\lambda_1\d^1 + \tfrac1q\lambda_2\d^2) = 0   && 
\forall \ve \lambda \in \Delta_2 \\
\Leftrightarrow  &&& \bpi(\t^1 + \tfrac1q\lambda_2\d^2) - \bpi(\t^1 + \tfrac1q\lambda_1\d^1 + \tfrac1q\lambda_2\d^2) = 0   &&
\forall \ve \lambda \in \Delta_2\\
\Leftrightarrow &&& \bpi(\t^1 + \tfrac1q(1-\lambda_2)\d^1 +\tfrac1q\lambda_2\d^2) - \bpi(\t^1 + \tfrac1q\lambda_1\d^1 + \tfrac1q\lambda_2\d^2) = 0   &&
\forall \ve \lambda \in \Delta_2\\
\Leftrightarrow && \bpi(\t^1 + \tfrac1q(\lambda_1+\lambda_2)\d^1) + &\bpi(\t^2 + \tfrac1q\lambda_2(\d^2-\d^1)) - \bpi(\t^1 + \tfrac1q\lambda_1\d^1+ \tfrac1q\lambda_2\d^2) = 0
&& \forall \ve \lambda \in \Delta_2 %
\end{align*}
The second equation follows since $\bpi$ is zero on $K_1$.  Therefore, the value of $\bpi$ on $K$ is the constant along horizontal lines. Thus, we can obtain the third equation.
Because $\bpi$ is zero on $K_1$, we obtain the last equation, which is exactly the definition of $\bpi$ being additive on $\tau^2$.
\end{proof}

\begin{lemma}[Step 5b(ii) equivalence]%
  \label{removeTriFromEdgesCont}
  Let $K \in \Itri$ be a triangle with edges $K_1, K_2, K_3 \in\Iedge$.  Therefore, $K_1, K_2, K_3 \subseteq K$.
  Suppose $\tau^0 = (K_1, K_2, K, (1,1,-1), \t)$ is the \underline{\emph{only}} 7-tuple in $\Etuple$ containing $K$. Then for $\tau^1 = (K_1, K_2,  K_3, (1,1,-1), \t)$, the update 
    $\E \leftarrow \E':=\left(\E \setminus \{\tau^0\} \right)\cup \{
  \tau^1\}, \ \ 
  \P \leftarrow \P\setminus \{K\}
  $
  is equivalent. 
\end{lemma}

\begin{proof}
We first show that $\tau^0$ is a valid 7-tuple given that $\tau^1$ is a valid 7-tuple.

Let $\v \in K_1 \cap K_2$ and consider the sets $\bar K_1 = K_1 - \v$, $\bar K_2 = K_2 - \v$, $\bar K_3 = K_3 - \v$, and $\bar K = K - \v$.
Since $(K_1, K_2, K, (1,1,-1), \t)$ is a valid 7-tuple, by \autoref{validCharacterization} we have
$ (\bar K + \v) \subseteq - \t + (\bar K_2 + \v)  + (\bar K_1 + \v)$
thus
$\bar K  \subseteq \v - \t + \bar K_2 + \bar K_1$.
By the geometry of $\P_q$, since $\bar K, \bar K_1, \bar K_2$ all touch the origin and $\bar K_1, \bar K_2$ are edges of $K$, this relation can only hold if $\v = \t$.  Therefore
$$
\bar K  \subseteq \bar K_2 + \bar K_1.
$$
Clearly $\bar K_3 \subseteq \bar K  \subseteq \bar K_2 + \bar K_1$, which is condition 3 of \autoref{validCharacterization}.
Furthermore, it can be checked that every triangle $\hat K$ in $\P_q$ touching the origin with edges $\hat K_1, \hat K_2, \hat K_3$ where $\hat K_1, \hat K_2$ both touch the origin, that $\hat K_1 \subseteq \hat K_3 + (-\hat K_2)$ and $\hat  K_2 \subseteq \hat  K_3 + (-\hat K_1)$.
Applying this to $\bar K_1, \bar K_2, \bar K_3$ and substituting back to $K_1, K_2, K_3$ then shows conditions 1 and 2 of \autoref{validCharacterization}.  Therefore, $(K_1, K_2, K_3, (1,1,-1), \t)$ is a valid 7-tuple.

Define the projection $h\colon \widehat \Pi^\Etuple_{\mathcal{Z},\P} \to \widehat \Pi^{\Etuple'}_{\mathcal{Z},\P\setminus \{K\}}$ such that $\theta \mapsto \theta|_{\bigcup(\P \setminus \{K\})}$.

Suppose $\theta \in \widehat \Pi^\Etuple_{\mathcal{Z},\P}$.
Clearly $h(\theta)$ is continuous since $h$ is just the restriction to a closed set.
By \autoref{remContinuity}, since $\theta$ is additive over $(K_1, K_2, K, (1,1,-1), \t)$, then it is also additive over $\tau^1=(K_1, K_2, K_3, (1,1,-1), \t)$.  Therefore,  $\theta \in \bar\Pi^{\Etuple\cup \{\tau^1\}}_{\mathcal{Z},\P} \subseteq \bar\Pi^{\Etuple'}_{\mathcal{Z},\P}$.   This inclusion holds since $\Etuple' \subseteq \Etuple \cup \{\tau^1\}$. Since $\Etuple'$ has no 7-tuples containing $K$, the restriction $\theta|_{\bigcup(\P \setminus \{K\})}$ must also be additive on all 7-tuples in $\Etuple'$.
Thus $h(\theta) \in \widehat\Pi^{\Etuple'}_{\mathcal{Z},\P\setminus \{K\}}$.

Now define the lift $h^{-1}\colon \bar\Pi^{\Etuple'}_{\mathcal{Z},\P\setminus \{K\}} \to \bar\Pi^\Etuple_{\mathcal{Z},\P}$ by
$
\hat \theta \mapsto \theta
$
where $\theta(\x) = \hat \theta(\x)$ for all $\x \in \bigcup(\P \setminus\{K\}) = \bigcup\P \setminus \intr(K)$, and $\theta$ is uniquely defined on $\intr(K)$
by the equation
\begin{align*}
\theta(\x) + \theta(\y) - \theta(\z) &= 0 &\forall \ \x \in K_1, \y \in K_2, \z \in \intr(K), \x + \y - \z = \t.
\end{align*}

Now consider any $\hat \theta \in \widehat\Pi^{\Etuple'}_{\mathcal{Z},\P\setminus \{K\}}$ and let $\theta = h^{-1}(\hat \theta)$.   By construction of $h^{-1}$, we see that $\theta$ is additive on the (non-valid) 7-tuple
$(K_1, K_2, \intr(K), (1,1,-1), \t)$. Since $\hat \theta$, and hence $\theta$, is additive on the 7-tuples in $\Etuple''$.
Specifically, $\theta$ is additive on
$(K_1, K_2, K_3, (1,1,-1), \t)$.  Also, since $\theta|_{\P_q} = 0$, it follows trivially that $\theta$ is additive on the (non-valid) 7-tuples $(K_1, K_2, K_1, (1,1,-1), \t)$ and $(K_1, K_2, K_2, (1,1,-1), \t)$.
Therefore, we obtain that $\theta$ is additive on the valid 7-tuple $(K_1, K_2,  K, (1,1,-1), \t)$.
Since $\theta$ is continuous on $K_1$ and $K_2$ it follows by construction that $\theta$ is continuous on $K$.
Therefore, $\theta \in \widehat\Pi^\Etuple_{\mathcal{Z},\P}$.

\end{proof}

After all the previous steps, we can assume that $\E = \E_{\point, \ver\hor\diag, \ver\hor\diag} \cup \E_{\ver, \hor, \diag}$.
\begin{lemma}
\label{lemEdges}
Let $\Etuple$ be the set of valid 7-tuples, $\P$ a collection of faces contains all faces in $\P_q$, and $\mathcal{Z} = \verts(\P_q)$. Then there exists
\begin{itemize}
\item a collection of faces $\P' \subseteq \P$ that is a union of faces in $\P_q$ and contains all faces in $\Iedge \cup \Ipoint$ and possibly some faces of $\Itri$, and
\item  a set $\Etuple'$ of $\pi$-additive 7-tuples containing no 7-tuples with triangles (Furthermore, we can assume $\Etuple'$ only contain $\Etuple_{\ver, \hor, \diag}$ and $\Etuple_{\point, \edge, \edge}$)
\item a set $\mathcal{Z}'\supseteq\mathcal{Z}$ (possibly some faces of $\Iedge \cup \Ipoint \cup \Itri$),
\end{itemize}
  such that
  \eqref{eq:perturbation-correspondance} holds.
\end{lemma}

\noindent  
\begin{tabular}{|m{6cm}|m{6cm}|m{4.5cm}|}
\hline
\begin{tikzpicture}
\node (pic0){
\begin{tikzpicture}[scale=\myscale]
\drawGridred
    \foreach \x in {0,1,2,3,4} {
        \foreach \y in {0,1,2,3,4} {
            \mytriangleRemove{\x}{\y}
        }
    }
\foreach \x in {1,2,3,4,5} {
        \foreach \y in {1,2,3,4,5} {
            \mytriangleRemoveReflect{\x}{\y}
        }
    }
    \mytriangle{white}{1}{4}{2}
     \mytriangle{\zero}{0.4}{4}{2}

         \mytriangle{white}{1}{3}{3}
     \mytriangle{\zero}{0.4}{3}{3}

         \mytriangle{white}{1}{1}{2}
     \mytriangle{\zero}{0.4}{1}{2}

         \mytriangle{white}{1}{4}{2}
     \mytriangle{\zero}{0.4}{4}{2}

         \mytriangleReflect{white}{1}{2}{2}
             \mytriangleReflect{\zero}{0.4}{2}{2}

        \mytriangleReflect{white}{1}{1}{4}
        \mytriangleReflect{\zero}{0.4}{1}{4}
             
     \mytriangleReflect{white}{1}{2}{2}

\end{tikzpicture}
\tikzLRarrow
};
    \node[right=0.1 of pic0] (pic1) {
            \begin{tikzpicture}[scale=\myscale]
            \drawGridred

                        \foreach \x in {0,1,2,3,4} {
        \foreach \y in {0,1,2,3,4} {
            \foreach \xx in {0, 0.33, 0.66,1} {
                \foreach \yy in {0, 0.33, 0.66,1} {
                    \mypoint{red}{\x + \xx}{\y + \yy}
            }       }
        }
    }
    
    \foreach \x in {0,1,2,3,4} {
        \foreach \y in {0,1,2,3,4} {
             \mypointHole{white}{\x + 0.33}{\y + 0.33}
             \mypointHole{white}{\x + 0.66}{\y + 0.66}
}
}

    \mytriangle{white}{1}{4}{2}
     \mytriangle{\zero}{0.4}{4}{2}

         \mytriangle{white}{1}{3}{3}
     \mytriangle{\zero}{0.4}{3}{3}

         \mytriangle{white}{1}{1}{2}
     \mytriangle{\zero}{0.4}{1}{2}

         \mytriangle{white}{1}{4}{2}
     \mytriangle{\zero}{0.4}{4}{2}

        \mytriangleReflect{white}{1}{1}{4}
        \mytriangleReflect{\zero}{0.4}{1}{4}

    \mypointHole{blue}{1.66}{1.66}
    
        \end{tikzpicture}
    };
\end{tikzpicture}
&
\emph{Step 6a. Case: Empty triangle}
Suppose there exists a triangle $\tri I \in \mathcal P$ with $I \notin p(\mathcal E)$ and $i \notin \mathcal Z$.

Then there exists a perturbation on this triangle.

Then this perturbation implies a lifting to a perturbation function in the original space.

\begin{center} \textbf{[Contradiction by perturbation]}\end{center}
&
Derives a contradiction!
\\
\hline
\begin{tikzpicture}
\node (pic0){
\begin{tikzpicture}[scale=\myscale]
\drawGridred
    \foreach \x in {0,1,2,3,4} {
        \foreach \y in {0,1,2,3,4} {
            \mytriangleRemove{\x}{\y}
        }
    }
\foreach \x in {1,2,3,4,5} {
        \foreach \y in {1,2,3,4,5} {
            \mytriangleRemoveReflect{\x}{\y}
        }
    }
            \drawGridredNoLines

\end{tikzpicture}
\tikzLRarrow
};
    \node[right=0.1 of pic0] (pic1) {
            \begin{tikzpicture}[scale=\myscale]
            \diagGrid[{}{}{}{}]    
    
            \foreach \x in {0,1,2,3,4} {
        \foreach \y in {0,1,2,3,4} {
            \foreach \xx in {0,  0.66,1} {
                \foreach \yy in {0, 0.33,1} {
                    \mypointHole{white}{\x + \xx}{\y + \yy}
            }       }
        }
    }
            \foreach \x in {0,1,2,3,4} {
        \foreach \y in {0,1,2,3,4} {
            \foreach \xx in {0, 0.33, 1} {
                \foreach \yy in {0,  0.66,1} {
                    \mypointHole{white}{\x + \xx}{\y + \yy}
            }       }
        }
    }

            \drawGridredNoLines

        \end{tikzpicture}
    };
\end{tikzpicture}
 & \emph{Step 6b. Case: No empty triangles} Suppose now that we can project out all triangles.  Then we have a system that is actually a 1-dimensional system.  Thus we can apply \autoref{lem:infinte-system-to-finite-system-edges} to show that the remaining infinite system of equations is nontrivial if and only if a finite system of equations is nontrivial.  This implies a perturbation on the finite system that maps to a subset of $\tfrac{1}{mq} \Z^2$.

 & {Derives a contradiction!} \\
 \hline
\multicolumn{3}{|c|}{} \\ %
\multicolumn{3}{|c|}{This completes the proof!} \\[-1ex] %
\multicolumn{3}{|c|}{\rule{\linewidth}{0.4pt}} \\ %
\hline
\end{tabular}

From the previous steps, we can assume that $\bar\Pi^{\Etuple'}_{\mathcal{Z}',\P'}\ne \{0\}$, and we only need to show that ${\bar\Pi^{\Etuple'\cap\frac1{mq}\Z^2}_{\mathcal{Z}'\cap \frac1{mq}\Z^2,\P' \cap \frac1{mq}\Z^2}}\ne \{0\}$. If there is a nontrivial $\bpi_{mq}\in {\bar\Pi^{\Etuple'\cap\frac1{mq}\Z^2}_{\mathcal{Z}'\cap \frac1{mq}\Z^2,\P' \cap \frac1{mq}\Z^2}}$, by \autoref{lemContinuousDiscrete}, we also have a continuous piecewise linear function $\bpi$ in $\widehat\Pi^{\Etuple'}_{\mathcal{Z}', \P'}$ constructed from the interpolation of $\bpi_{mq}$, whose restriction on $\frac1{mq}\Z^2$ is $\bpi_{mq}$.
And we can lift $\bpi$ back to $\bpi^0\in\widehat\Pi^{\Etuple}_{\mathcal{Z}, \P}$ such that the restriction of $\bpi^0|_{\frac1{mq}\Z^2}\in\bar\Pi^{\Etuple\cap\frac1{mq}\Z^2}_{\mathcal{Z}\cap \frac1{mq}\Z^2, \P \cap \frac1{mq}\Z^2}$ is nonzero because it is a lift of $\bpi_{mq}$.\\

\noindent {\emph{Step 6a. Interior Triangle Perturbation:}}\\
Suppose there exists a triangle $\tri K \in \P'$ and $K\notin\ZZ'$.
Let $\bpi\colon \bigcup\P'\cap\frac{1}{mq}\Z^2\mapsto \R$ such that $\bpi(\x) = 0$ for any $\x\in(\ZZ'\cup\Iedge\cup \Ipoint) \cap(\bigcup\P'\cap \frac{1}{mq}\Z^2)$, and $\bpi(\y + \t) = \bpi(\y)$ for any $\y\in\bigcup\P'\cap\frac{1}{mq}\Z^2$ and $\t\in\Z^2$.
Then for any such $\bpi$, $\bpi$ is trivially additive on all valid 7-tuples $(I_1, I_2, I_3, \ve\sigma, \t)$ where $I_1, I_2, I_3 \subseteq \Iedge \cup \Ipoint$.  Therefore,
$$
\bpi \in \bar\Pi^{\Etuple'\cap \frac1{mq}\Z^2}_{\mathcal{Z}'\cap \frac1{mq}\Z^2,\bigcup\P'\cap \frac1{mq}\Z^2}.
$$
In particular, assume that $K=\begin{psmallmatrix}x_0\\y_0\end{psmallmatrix} + \frac1q \conv(\{ \begin{psmallmatrix}0\\0\end{psmallmatrix}
, \begin{psmallmatrix}1\\0\end{psmallmatrix}
, \begin{psmallmatrix}0\\1\end{psmallmatrix}
 \})$. Then because $m\ge 3$, we can define $\bpi(\x) = 1$ for  $\x = \begin{psmallmatrix}x_0\\y_0\end{psmallmatrix} + \frac{1}{mq} \begin{psmallmatrix}1\\1\end{psmallmatrix}\notin \P_{q,\edge}$.
Therefore, there exists a non-trivial $\bpi$ for $m\ge 3$ in ${\bar\Pi^{\Etuple'\cap\frac1{mq}\Z^2}_{\mathcal{Z}'\cap \frac1{mq}\Z^2,\P' \cap \frac1{mq}\Z^2}}$. Thus, $\pi$ is not extreme.

\noindent {\emph{Step 6b. Edge perturbation:}}\\
We will now assume that $p(\E')$ does not contain any triangles $K \in \Itri$ that is not in $\ZZ'$ (i.e, $K$ is not used in any 7-tuple from $\E'$).

Let $\ver K_1, \hor K_2, \diag K_3$ be the edges of $K$.   We update $\ZZ'$ by adding 
 $K_1, K_2, K_3$  and remove $K$.  Then by setting $\P' \leftarrow \P' \setminus \{K\}$ we achieve an equivalent update where the map simply sets function values on $K$ to zero.

Therefore, we can assume $\P'=\P_{q, \edge}$ and only need to show there exists a nontrivial function $\bpi$ on $\P_{q, \edge}\cap\frac1{mq}\Z^2$. The function value in the interior of triangles is determined uniquely by interpolation.

The only remaining relations contain edges and points and are of the form of Items~\ref{item:point-edge},\ref{item:three-diags} from \autoref{lemma:cases}.
 We now map the problem to a different space using \autoref{lem:point-edge-system} and apply \autoref{lem:infinte-system-to-finite-system-edges}.  This shows that when looking for nontrivial solutions, it suffices to look at the discrete grid $\tfrac{1}{mq} \Z^2$.  In particular, we can now look for solutions to a finite problem on the restricted grid.  Solutions here exist if and only if solutions on the continuous domain exist.  

Finally then, we can conclude that 3 $\Rightarrow$ 1.
\end{proof}

As a corollary to this result, we could provide an algorithmic result, akin to {\cite[Theorem 1.8]{bhk-IPCOext}}, that shows that in time polynomial in $q$, we can detect if a given piecewise linear function is extreme or a facet.  But we don't feel the need to state that explicitly here.

\section*{Acknowledgments}
The authors gratefully acknowledge partial support from the National Science
Foundation through grants 
DMS-0914873 (R.~Hildebrand, M.~K\"oppe), DMS-1320051 (M.~K\"oppe), DMS-2012764
(M.~K\"oppe, L.~Xu). R. Hildebrand was also funded by ONR Grant N00014-20-1-2156. Any opinions, findings, and conclusions or recommendations expressed in this material are those of the authors and do not necessarily reflect the views of the Office of Naval Research.

\bibliographystyle{abbrv}
\providecommand\ISBN{ISBN }
\bibliography{MLFCB_bib}

\appendix

\section{Postponed Proofs}

\begin{proof}[Proof of \autoref{lem:F-formulas}]
\renewcommand{\qedsymbol}{$\triangle$} %

\begin{proof}[Proof for $F(\ver\hor\diag I_1, \point \{\u^2\}, \ver\hor\diag I_3, \ve \sigma, \t)$ ]
We show that the set of additivities is a 1-dimensional space.
Suppose that 
\(
I_3 = \{\x^3\in\R^2 : \x^3 = \u^3 + \frac{1}{q} \lambda \d, \lambda \in \Delta_1\},
\)
where \(\d\in\{(1,0), (0,1), (1,-1)\}\), and \(\lambda \in [0,1]\) are scalar parameters. 
By \autoref{validPointCharacterization}, we know that 
$ I_1 = \sigma_1 \t - (\sigma_1 \sigma_2 \{\u^2\} + \sigma_1 \sigma_3 I_3).$

Therefore, $I_1=\{\x^1\in\R^2 : \x^1 = \u^1 - \sigma_1\sigma_3\frac{1}{q} \lambda \d, \lambda \in \Delta_1\}$, where $\u^1 = \sigma_1\t - \sigma_1\sigma_2\u^2 -\sigma_1\sigma_3\u^3$.

If $F(\tau)$ can be parametrized, then $ I_1 = \sigma_1 \t - (\sigma_1 \sigma_2 \{\u^2\} + \sigma_1 \sigma_3 I_3)$. By \autoref{validPointCharacterization}, $\tau$ is valid.

Furthermore, $\t = \sigma_1\u^1+\sigma_2\u^2+\sigma_3\u^3$ is unique.
\end{proof}

    \begin{proof}[Proof for $F(\ver I_1, \hor I_2, \diag I_3, \ve \sigma, \t)$]
    We express each point \(\x^i \in I_i\) as:
\(
\x^i = \u^i + \frac{1}{q} \lambda_i \d^i,
\)
where \(\d^1 = (1,0)\), \(\d^2 = (0,1)\), and \(\d^3 = (1,-1)\). Here, \(\u^i\) is a fixed point on the edge \(I_i\), and \(\lambda_i\in [0,1]\) represents the scalar parameter for each direction \(\d^i\).
    Since \((\x^1, \x^2, \x^3) \in F\Tuple(I_1, I_2, I_3, \ve \sigma, \t)\), we have:
\[
\sum_{i=1}^3 \sigma_i \x^i = \t \quad \Rightarrow \quad 
\sigma_1 \x^1 + \sigma_2 \x^2 = \t - \sigma_3 \x^3.
\]
Substituting \(\x^i = \u^i + \frac{1}{q} \lambda_i \d^i\) for each \(i\), we obtain:
\[
\sigma_1 \left(\u^1 + \frac{1}{q} \lambda_1 \d^1\right) + \sigma_2 \left(\u^2 + \frac{1}{q} \lambda_2 \d^2\right) = \t - \sigma_3 \left(\u^3 + \frac{1}{q} \lambda_3 \d^3\right).
\]
Expanding and rearranging, we get:
\[
\sigma_1 \u^1 + \sigma_2 \u^2 + \frac{1}{q} \left( \sigma_1 \lambda_1 \d^1 + \sigma_2 \lambda_2 \d^2 \right) = \t - \sigma_3 \u^3 - \frac{1}{q} \sigma_3 \lambda_3 \d^3.
\]
We rearrange to express \(\lambda_1\) and \(\lambda_2\) as a linear system:
\[
\frac{1}{q} \begin{bmatrix} \sigma_1 \d^1 & \sigma_2 \d^2 \end{bmatrix} \begin{bmatrix} \lambda_1 \\ \lambda_2 \end{bmatrix} = \t - \sigma_1 \u^1 - \sigma_2 \u^2 - \sigma_3 \u^3 - \frac{1}{q} \sigma_3 \lambda_3 \d^3.
\]
Substitute \(\d^1 = (1,0)\), \(\d^2 = (0,1)\), and \(\d^3 = (1,-1)\), and simplify:
\[
\frac{1}{q} \begin{bmatrix} \sigma_1 & 0 \\ 0 & \sigma_2 \end{bmatrix} \begin{bmatrix} \lambda_1 \\ \lambda_2 \end{bmatrix} = \t - \sigma_1 \u^1 - \sigma_2 \u^2 - \sigma_3 \u^3 - \frac{1}{q} \sigma_3 \lambda_3 \begin{bmatrix}
    1\\
    -1
\end{bmatrix}.
\]
Since the matrix 
\(\begin{bmatrix} \sigma_1 & 0 \\ 0 & \sigma_2 \end{bmatrix}\) is non-singular with \(\begin{bmatrix} \sigma_1 & 0 \\ 0 & \sigma_2 \end{bmatrix}^{-1} = \begin{bmatrix} \sigma_1 & 0 \\ 0 & \sigma_2 \end{bmatrix}\),
we can solve for \(\lambda_1\) and \(\lambda_2\) uniquely in terms of \(\lambda_3\):
\[
\begin{bmatrix} \lambda_1 \\ \lambda_2 \end{bmatrix} = q \begin{bmatrix} \sigma_1 & 0 \\ 0 & \sigma_2 \end{bmatrix} \left( \t - \sigma_1 \u^1 - \sigma_2 \u^2 - \sigma_3 \u^3 - \frac{1}{q} \sigma_3 \lambda_3 \begin{bmatrix}
    1\\
    -1
\end{bmatrix} \right) 
= 
 \t' + \lambda_3  \begin{bmatrix} - \sigma_3\sigma_1\\
\sigma_3\sigma_2
\end{bmatrix},
\]
for some $\t' \in \Z^2$.
Note that the expression is only feasible (as a valid 7-tuple) if for each $i=1, 2$ we have either $\lambda_i = 1-\lambda_3$ or $\lambda_i = \lambda_3$
since we need $\lambda_i \in [0,1]$.

Thus \(\lambda_3\)
  serves as a single parameter from which 
$\lambda_1$
  and 
$\lambda_2$
  can be uniquely determined.
  Set $\lambda  := \lambda_3$. 
    \end{proof}
       \begin{proof}[Proof for 
       $F(\tri I_1, \point I_2, \tri I_3, \ve \sigma, \t)$]
 This follows similar to $F(\ver\hor\diag I_1, \point I_2, \ver\hor\diag I_3, \ve \sigma, \t)$.
       \end{proof}
   \begin{proof}[Proof for 
   $F(\hor I_1, \ver I_2, \tri I_3, \ve \sigma, \t)$]

Let $I_i' = \sigma_i I_i$ for $i=1,2,3$. Then \((\x^1, \x^2, \x^3) \in F\Tuple(I_1, I_2, I_3, \ve \sigma, \t)\) is equivalent to \((\sigma_1\x^1, \sigma_2\x^2, -\sigma_3\x^3) \in F\Tuple(I_1', I_2', I_3', (1,1,-1), \t)\).

   We express each point $\tilde{\x}^i\in I_i'$ as $\tilde{\x}^i=\u^i +\frac1q\lambda_i\e^i$ for $i=1,2$, where $\u^i$ is a vertex of $I_i$ and $\lambda_i\in\Delta_1$. 
   By \autoref{validCharacterization}, we know that $I_3'\subseteq -\t + I_2' + I_1'$. Thus,
   $I_3'$ is contained in the parallelogram $\u^3 + \frac1q\lambda_1\e^1 + \frac1q\lambda_2\e^2$, where $\u^3 = -\t + \u^2 + \u^1$, $\lambda_1,\lambda_2\in [0,1]$.

   If $I_3'$ is a lower triangle $\FundaTriangleLower$, then $I_3' = \{\x \in \R^2 : \x = \u^3 + \frac{1}{q} \left(\lambda_1 \ve{e}^1 + \lambda_2 \ve e^2\right), \ve \lambda \in \Delta_2\}$. Therefore, 
   $$
F(\hor I_1, \ver I_2, \tri I_3, \ve \sigma, \t) 
= \{ (\sigma_1\u^1, \sigma_2\u^2, -\sigma_3\u^3) + \frac{1}{q}(\lambda_1 \sigma_1 \ve e^1,\lambda_2 \sigma_2\ve e^2,-\sigma_3(\lambda_1 \ve e^1+\lambda_2 \ve e^2)): \ve \lambda \in \Delta_2\}.
 $$
 
   If $I_3'$ is an upper triangle $\FundaTriangleUpper$, then $I_3' = \{\x \in \R^2 : \x = \u^3 + \frac{1}{q} \left(\lambda_1 \ve{e}^1 + \lambda_2 \ve e^2\right), \mathbf{1}-\ve\lambda \in \Delta_2\}$, and
$$
F(\hor I_1, \ver I_2, \tri I_3, \ve \sigma, \t) 
= \{(\sigma_1\tilde{\u}^1, \sigma_2\tilde{\u}^2, -\sigma_3\tilde{\u}^3) + \frac{1}{q}(-\tilde{\lambda}_1 \sigma_1 \ve e^1,-\tilde{\lambda}_2 \sigma_2\ve e^2,\sigma_3(\tilde{\lambda}_1 \ve e^1+\tilde{\lambda}_2 \ve e^2)): \tilde{\ve \lambda} \in \Delta_2\},
 $$
 where $\tilde{\lambda}_i=1-\lambda_i$, $\tilde{\u}^i=\u^i+\frac1q\e^i$, $i=1,2$, and $\tilde{\u}^3 = -\t +\tilde{\u}^2 + \tilde{\u}^1$. Then $\tilde{\u}^i$, $i=1,2,3$ and $\d^i=-\e^i$, $i=1,2$ will be the desired parameterization.

On the other hand, if $F(\tau)$ can be parametrized
$$
F(\hor I_1, \ver I_2, \tri I_3, \ve \sigma, \t) 
= \{ (\sigma_1\u^1, \sigma_2\u^2, -\sigma_3\u^3) + \frac{1}{q}(\lambda_1 \sigma_1 \ve d^1,\lambda_2 \sigma_2\ve d^2,-\sigma_3(\lambda_1 \ve d^1+\lambda_2 \ve d^2)): \ve \lambda \in \Delta_2\},
$$
then $\sigma_1 I_1+\u^2=\t- \sigma_3 (I_3)_1$, $\u^1 + \sigma_2 I_2=\t- \sigma_3 (I_3)_2$, where $(I_3)_i=\{-\sigma_3(\u^3+\frac1q\lambda_i \d^i): \lambda_i\in \Delta_1\}$, $i=1,2$.
Thus, $ I_1 =\sigma_1 \t - (\sigma_1\u^2 + \sigma_1\sigma_3 (I_3)_1) \subseteq \sigma_1 \t - (\sigma_1 \sigma_2 I_2 + \sigma_1 \sigma_3 I_3)$. Similarly, $ I_2 \subseteq \sigma_2 \t - (\sigma_2 \sigma_1 I_1 + \sigma_2 \sigma_3 I_3)$. Also $I_3 \subseteq \sigma_3 \t - (\sigma_3 \sigma_2 I_2 + \sigma_3 \sigma_1 I_1)$. By \autoref{validCharacterization}, $\tau$ is valid.

Furthermore, note that we can find $\u^3, \d^1, \d^2$ from $I_3$ because the three vertices of $I_3$ is $-\sigma_3\u^3$, $-\sigma_3(\u^3+\frac1q\d^1)$, $-\sigma_3(\u^3+\frac1q\d^2)$. Then we can determine $\u^1, \u^2$ from $I_1$ and $I_2$.
Thus $\t = \u^1+\u^2 - \u^3$ is unique.
   \end{proof}
   \renewcommand{\qedsymbol}{$\square$}
\end{proof}

\begin{proof}{Proof of \autoref{lem:orbits}}
    We will prove each property of the lemma in turn.

    \textbf{(1) Symmetry of Orbits:}  
    For any \( x \in [0, 1] \), the orbit \( \mathcal{O}(x) \) is defined as:
    \[
    \mathcal{O}(x) = \{(x, i), (1 - x, i) : i = 1, \dots, \ell\}.
    \]
    By definition, this set includes both \(x\) and \(1 - x\) for each index \(i\). It follows immediately that:
    \[
    \mathcal{O}(x) = \mathcal{O}(1 - x),
    \]
    because swapping \(x\) and \(1 - x\) in the orbit definition yields the same set. Thus, the system behaves symmetrically with respect to \(x\) and \(1 - x\).

    \textbf{(2) Disjointness of Orbits:}  
    Suppose \( x \neq x' \) and \( x, x' \in [0, 1/2] \). We want to show that:
    \[
    \mathcal{O}(x) \cap \mathcal{O}(x') = \emptyset.
    \]
    If $x$ (or $x'$) is $1/2$, then the orbit is unique as $\{(1/2, i) : i \in [\ell]\}$
    By definition, for \( x \in [0, 1/2) \), the orbit \( \mathcal{O}(x) \) consists of the pairs \( (x, i) \) and \( (1 - x, i) \), where \( 1 - x \in (1/2, 1] \). Similarly, the orbit \( \mathcal{O}(x') \) consists of the pairs \( (x', i) \) and \( (1 - x', i) \), where \( 1 - x' \in (1/2, 1] \).

    Since \( x \neq x' \) and both \( x \) and \( 1 - x \) are distinct from \( x' \) and \( 1 - x' \), it follows that:
    \[
    \mathcal{O}(x) \cap \mathcal{O}(x') = \emptyset.
    \]
    This is because the pairs \( (x, i) \) and \( (1 - x, i) \) are distinct from \( (x', i) \) and \( (1 - x', i) \), as \( x \neq x' \) and \( 1 - x \neq 1 - x' \). Therefore, the orbits for distinct points in \( [0, 1/2] \) are disjoint.

    \textbf{(3) Union of Orbits for \( x \in [0, 1/2] \) Covers All Points:}  
    We need to show that:
    \[
    \{(x, i) :  x \in [0, 1], i = 1, \dots, \ell\} = \bigsqcup_{x \in [0, 1/2]} \mathcal{O}(x).
    \]
    For any \( x \in [0, 1] \), we have two cases:
    (1) If \( x \in [0, 1/2] \), then \( (x, i) \) is part of \( \mathcal{O}(x) \).
    (2) If \( x \in [1/2, 1] \), then \( (x, i) = (1 - x', i) \) for some \( x' \in [0, 1/2] \), and thus \( (x, i) \in \mathcal{O}(x') \).

    Therefore, every point \( (x, i) \) in \( [0, 1] \) can be expressed as part of the orbit \( \mathcal{O}(x') \) for some \( x' \in [0, 1/2] \). This shows that the union of orbits for \( x \in [0, 1/2] \) covers the entire set of points \( (x, i) \) for \( x \in [0, 1] \).
\end{proof}

\end{document}